\documentclass[11pt,a4paper,oneside,reqno]{amsart}

\usepackage{amsmath, amsthm, amsopn, amssymb}
\usepackage{mathtools, thmtools, thm-restate}
\usepackage[tworuled,vlined]{algorithm2e}
\usepackage{bbm}
\usepackage{enumerate}
\usepackage{enumitem}
\usepackage{hyperref}
\usepackage{stmaryrd} % llbracket, rrbracket
\usepackage{xcolor}

\usepackage{tikz}
\usetikzlibrary{topaths,calc}
\usepackage{subcaption}

\usepackage{geometry}
\declaretheorem[name=Theorem,within=section]{thm}
\newtheorem{lemma}[thm]{Lemma}
\newtheorem{prop}[thm]{Proposition}
\newtheorem{cor}[thm]{Corollary}
\newtheorem{fact}[thm]{Fact}
\newtheorem{problem}[thm]{Problem}
\newtheorem{claim}[thm]{Claim}
\newtheorem{obs}[thm]{Observation}

\newtheorem{step}{Step}

\theoremstyle{definition}
\newtheorem*{remark}{Remark}
\newtheorem{dfn}[thm]{Definition}

\newcommand{\ZZ}{\mathbb{Z}}

\newcommand{\cA}{\mathcal{A}}
\newcommand{\cB}{\mathcal{B}}
\newcommand{\cC}{\mathcal{C}}
\newcommand{\cF}{\mathcal{F}}
\newcommand{\cG}{\mathcal{G}}
\newcommand{\cH}{\mathcal{H}}
\newcommand{\cL}{\mathcal{L}}
\newcommand{\cM}{\mathcal{M}}
\newcommand{\cN}{\mathcal{N}}
\newcommand{\cP}{\mathcal{P}}
\newcommand{\cR}{\mathcal{R}}
\newcommand{\cS}{\mathcal{S}}
\newcommand{\cT}{\mathcal{T}}
\newcommand{\cU}{\mathcal{U}}
\newcommand{\cX}{\mathcal{X}}
\newcommand{\cY}{\mathcal{Y}}
\newcommand{\cZ}{\mathcal{Z}}

\newcommand{\Col}{\mathrm{Col}}

\newcommand{\Z}{\mathbb{Z}}

\newcommand{\vdW}{\mathcal{W}}

\newcommand{\bB}{\mathbb{B}}

\newcommand{\bA}{\mathbf{A}}
\newcommand{\bt}{\mathbf{t}}

\renewcommand{\Pr}{\mathbb{P}}
\newcommand{\Ex}{\mathbb{E}}
\newcommand{\Var}{\mathrm{Var}}
\newcommand{\pVar}{\Var'}
\newcommand{\1}{\mathbbm{1}} % indicator

\newcommand{\phat}{\hat{p}}

\newcommand{\eps}{\varepsilon}
\newcommand{\pH}{p_{\cH}}
\newcommand{\Aut}{\mathrm{Aut}}
\newcommand{\Hom}{\mathrm{Hom}}
\newcommand{\cBZ}{\cB_Z}

\newcommand{\kAP}{k-\mathrm{AP}}

\newcommand{\MS}{\cM_{\star}} % the hypergraph of monochromatic stars
\newcommand{\RS}{\cR_{\star}} % the hypergraph of rainbow stars
\newcommand{\RC}{\cR_{\star}^{\star}} % the hypergraph of rainbow constellations

\newcommand{\betas}{\beta_{s}}

\newcommand{\Seq}{\mathfrak{S}}

\newcommand{\betac}{\beta_{c}}

\newcommand{\con}{\mathrm{con}}

\newcommand{\supp}{\mathrm{supp}}

\newcommand{\br}[1]{\llbracket{#1}\rrbracket}

\renewcommand{\le}{\leqslant}
\renewcommand{\ge}{\geqslant}

\newcommand{\by}[2]{\overset{\mbox{\tiny{#1}}}{#2}}

\title{Sharp thresholds for Ramsey properties}

\author{Ehud Friedgut}
\address{Faculty of Mathematics and Computer Science, Weizmann Institute of Science, Rehovot 7610001, Israel}
\email{ehud.friedgut@weizmann.ac.il}

\author{Eden Kuperwasser}
\address{School of Mathematical Sciences, Tel Aviv University, Tel Aviv 6997801, Israel}
\email{kuperwasser@mail.tau.ac.il}

\author{Wojciech Samotij}
\address{School of Mathematical Sciences, Tel Aviv University, Tel Aviv 6997801, Israel}
\email{samotij@tauex.tau.ac.il}

\author{Mathias Schacht}
\address{Fachbereich Mathematik, Universit\"at Hamburg, Hamburg, Germany}
\email{schacht@math.uni-hamburg.de}

\thanks{This research was supported by the grant I-1358-304.6/2016 from the German--Israeli Foundation for Scientific Research and Development (GIF)}

\begin{document}

\begin{abstract}
In this work, we develop a unified framework for establishing sharp threshold results for various Ramsey properties.  To achieve this, we view such properties as non-colourability of auxiliary hypergraphs.  Our main technical result gives sufficient conditions on a sequence of such hypergraphs that guarantee that this non-colourability property has a sharp threshold in subhypergraphs induced by random subsets of the vertices. 

Furthermore, we verify these conditions in several cases of interest.  In the classical setting of Ramsey theory for graphs, we show that the property of being Ramsey for a graph $H$ in $r$ colours has a sharp threshold in $G_{n,p}$, for all $r \ge 2$ and all $H$ in a class of graphs that includes all cliques and cycles.  In the arithmetic setting, we establish sharpness of thresholds for the properties corresponding to van der Waerden's theorem and Schur's theorem, also in any number of colours.
\end{abstract}

\maketitle

\setcounter{tocdepth}{1}
\tableofcontents

\section{Introduction}

A typical result in Ramsey theory states that, given a structure $A$ and an integer $r \ge 2$, every colouring of the elements of any sufficiently `rich' set $V$ with $r$ colours must contain a monochromatic copy of $A$.  The most prominent example is Ramsey's theorem~\cite{Ram30}, which states that, for any graph $H$ and any integer $r \ge 2$, every $r$-colouring of the edges of a sufficiently large complete graph must yield a monochromatic copy of~$H$.  Two other famous instances, which actually predate~\cite{Ram30}, include van der Waerden's theorem~\cite{vdW27} on arithmetic progressions and Schur's theorem~\cite{Sch17} on additive triples.

In the 1980s, researchers have turned to studying Ramsey properties of random sets while trying to better understand what `richness' assumptions a set $V$ needs to satisfy so that it contains a monochromatic copy of a given structure $A$ in every $r$-colouring.  The seminal work of Frankl and R\"odl~\cite{FraRod86} proves the existence of a $K_4$-free graph whose every $2$-colouring contains a monochromatic triangle by considering the binomial random graph $G_{n,p}$ for an appropriately chosen edge density $p$.  Soon afterwards, {\L}uczak, Ruci\'nski, and Voigt~\cite{LucRucVoi92} initiated the systematic study of Ramsey properties of random graphs, which has quickly become one of the central topics in probabilistic combinatorics.

Given a finite set $V$ and a $p \in [0,1]$, we will write $V_p$ to denote the random subset of $V$ obtained by independently retaining each of its elements with probability $p$.  While investigating, for a sequence of sets $V$ whose sizes grow to infinity, the probability that $V_p$ has a given property $\cP$, one naturally encounters threshold phenomena.  We say that a sequence of probabilities $\hat{p}$ is a \emph{threshold} for a property $\cP$ if the following two statements hold:  On the one hand, for any $p \ll \hat{p}$, the probability that $V_p \in \cP$ tends to zero; on the other hand, for any $p \gg \hat{p}$, the same probability tends to one.  These two statements are aptly termed the $0$-statement and the $1$-statement, respectively.  Thresholds have been a central theme in probabilistic combinatorics since its very inception and date back to the seminal paper of Erd\H{o}s and R\'enyi~\cite{ErdRen60} which initiated the systematic study of random graphs.

The celebrated theorem of Bollob\'as and Thomason~\cite{BolTho87} asserts the existence of a threshold for any property of sets that is monotone and nontrivial; this includes all Ramsey properties.  However, this general theorem provides little clue regarding the location of this threshold.  As a result, the main focus of the vast majority of the many works on Ramsey properties of random sets was locating the corresponding threshold.  In particular, the locations of the thresholds for all of the aforementioned Ramsey properties were discovered in a series of papers by Graham, R\"odl, and Ruci\'nski~\cite{GraRodRuc96}, R\"odl and Ruci\'nski~\cite{RodRuc93,RodRuc94,RodRuc95,RodRuc97}, and Friedgut, R\"odl, and Schacht~\cite{FriRodSch10}.  Actually, these papers went one step further and showed that the $0$-statement and the $1$-statement hold already when $p \le c_0 \cdot \hat{p}$ and $p \ge c_1 \cdot \hat{p}$, respectively, for some sequence $\hat{p}$ and positive constants $c_0$ and $c_1$.

It is very natural to ask whether this gap can be reduced even further.  A property is said to have a \emph{sharp threshold} if, for some threshold $\hat{p}$ and every positive $\eps$, the $0$-statement holds for $p \le (1 - \eps) \hat{p}$ whereas the $1$-statement holds for $p \ge (1 + \eps) \hat{p}$;  otherwise, we say that the property has a \emph{coarse threshold}.  The notion of sharpness is closely reminiscent of the physical phenomenon of phase transition, where certain types of matter undergo a profound change in behaviour when their temperature crosses a certain point.  Sharpness of thresholds has been established for several natural graph properties, such as connectivity, the existence of a perfect matching, and Hamiltonicity.  On the other hand, many properties have been shown to have only a coarse threshold.

The presence of a sharp threshold, or a lack thereof, was demystified in the work of Friedgut~\cite{Fri99}.  Roughly speaking, the main result of~\cite{Fri99} states that a property has a coarse threshold if in only if it is `local' in the sense that it correlates with the property of containing a subset of a bounded size.  Friedgut's criterion, and Bourgain's formulation~\cite[Appendix]{Fri99} that extends it to a more general setting, have been an instrumental tool in proving that various properties have a sharp threshold.

Even though more than twenty years have passed since the work of Friedgut was published, only a handful of Ramsey properties have been shown to (or not to) have a sharp threshold:  First, Friedgut and Krivelevich~\cite{FriKri00} showed that for any tree $T$ (bar stars) and for any number of colours $r$ (except for $r=2$ in the case where $T$ is the path of length three), the property that any $r$-colouring of the edges of $G_{n,p}$ contains a monochromatic copy of $T$ has a sharp threshold.  Next, Friedgut, R\"odl, Ruci\'nski, and Tetali~\cite{FriRodRucTet06} established sharpness of the threshold for the corresponding property of the triangle, but only in the case where the number of colours $r$ is equal to two.  Much later,  Friedgut, H\`an, Person, and Schacht~\cite{FriHanPerSch16} proved sharpness of the threshold in the context of van der Waerden's theorem, again only in the two-colour case.  Building on ideas from~\cite{FriHanPerSch16}, Schacht and Schulenburg~\cite{SchSch18} returned to the setting of Ramsey's theorem and managed to extend the result of~\cite{FriRodRucTet06} from triangles to all nearly-bipartite graphs (see below) whereas Schulenburg~\cite{Sch16PhD} showed sharpness of the threshold in the context of Schur's theorem; both these results apply only to the two-colour case.

The main result of this paper is a common generalisation of all the above works, save for~\cite{FriKri00}.  We view Ramsey properties of random subsets as statements about non-$r$-colourability of subhypergraphs that these random subsets induce in the hypergraph $\cH$ that represents copies of a given structure $A$ in the ground set $V$.  Our main result supplies sufficient conditions on a sequence of uniform hypergraphs that guarantee that non-$r$-colourability of the random induced subhypergraph $\cH[V_p]$, and thus the corresponding Ramsey property of random sets, has a sharp threshold.  We postpone the exact statement of our theorem to Section~\ref{sec:main-result} and, in the remainder of this section, present several interesting corollaries of this general result.

\subsection{Graph properties}
\label{sec:graph-properties}

Given graphs $G$ and $H$ and an integer $r \ge 2$, we write $G \to (H)_r$ if every $r$-colouring of the edges of $G$ contains a monochromatic copy of $H$.  For the vast majority of pairs $H$ and $r$, the location of the threshold for the property $G_{n,p} \to (H)_r$ is determined by a simple parameter of $H$, called the \emph{$2$-density}, defined by
\[
  m_2(H) \coloneqq \max \left\{\frac{e_F-1}{v_F-2} : \emptyset \neq F \subseteq H\right\} \cup \left\{\frac{1}{2}\right\}.
\]
The following statement was proved in a series of papers of R\"odl and Ruci\'nski~\cite{RodRuc93, RodRuc94, RodRuc95}.  (The necessity for the special treatment of paths of length three in the case $r = 2$, originally missed by R\"odl and Ruci\'nski, was noticed by Friedgut and Krivelevich~\cite{FriKri00}.)

\begin{thm}[\cite{RodRuc95}]
  \label{thm:RodRuc}
  Let $r \ge 2$ be an integer and suppose that $H$ is a nonempty graph whose at least one component is not a star or (in the case $r=2$) a path of length three.  There exist positive constants $c_0$ and $c_1$ such that
  \[
    \lim_{n \to \infty} \Pr\big(G_{n,p} \to (H)_r\big) =
    \begin{cases}
      1 & \text{if $p \ge c_1 \cdot n^{-1/m_2(H)}$}, \\
      0 & \text{if $p \le c_0 \cdot n^{-1/m_2(H)}$}. \\
    \end{cases}
  \]
\end{thm}

In other words, Theorem~\ref{thm:RodRuc} states that, for most pairs $H$ and $r$, the function $n^{-1/m_2(H)}$ is a threshold for the property $G_{n,p} \to (H)_r$.  In the case where $H$ is a tree, Friedgut and Krivelevich~\cite{FriKri00} gave a complete characterisation of those pairs for which the corresponding threshold is coarse (when $H$ is a star or when $r=2$ and $H$ is a path of length three) or sharp (all other pairs $H$ and $r$).  Deciding the sharpness of the threshold for the property $G_{n,p} \to (H)_r$ turned out to be much harder in the case where $H$ contains a cycle.  Here, our knowledge is only fragmentary.  The monumental work of Friedgut, R\"odl, Ruci\'nski, and Tetali~\cite{FriRodRucTet06} established sharpness of the threshold in the case where $H$ is the triangle and $r = 2$ using a very elaborate, long, and technical argument.  The authors of~\cite{FriRodRucTet06} speculated that the threshold is sharp whenever $H$ contains a cycle, for any number of colours, but so far this has been confirmed only when $H$ is nearly-bipartite\footnote{A graph $H$ is \emph{nearly-bipartite} if $\chi(H \setminus e) \le 2$ for some edge $e$ of $H$.} and strictly-$2$-balanced\footnote{A graph $H$ is \emph{strictly $2$-balanced} if $m_2(F) < m_2(H)$ for every strict, nonempty subgraph $F \subseteq H$.} and $r = 2$ in the recent work of Schacht and Schulenburg~\cite{SchSch18}.

We prove that the threshold is sharp for a much broader family of graphs that includes all cliques and for any number of colours.  We call a graph $H$ \emph{collapsible} if, for every edge $e$ of $H$ and every endpoint $a$ of $e$, there is an edge $f$ of $H$ and a homomorphism from $H \setminus f$ to $H \setminus e$ that maps both endpoints of $f$ to $a$.  It is is not difficult to verify (see Section~\ref{sec:graphs}) that every graph that is either complete or nearly-bipartite is collapsible.  Unfortunately, not every graph is collapsible; for example, the Petersen graph is not collapsible (see Appendix~\ref{apx:graph_application}).

\begin{thm}
  \label{thm:main-Ramsey}
  Suppose that $H$ is a strictly $2$-balanced, collapsible graph that is not a forest and $r \ge 2$ is an integer.  There exist positive constants $c_0$ and $c_1$ and a function $c(n)$ satisfying $c_0 \le c(n) \le c_1$ such that, for every positive $\eps$,
  \[
    \lim_{n \to \infty} \Pr\big(G_{n,p} \to (H)_r\big) =
    \begin{cases}
      1 & \text{if $p \ge (1+\eps) c(n) \cdot n^{-1/m_2(H)}$}, \\
      0 & \text{if $p \le (1-\eps) c(n) \cdot n^{-1/m_2(H)}$}. \\
    \end{cases}
  \]
\end{thm}

\begin{remark}
  In fact, when $r = 2$, we may replace the assumption that $H$ is collapsible with a seemingly weaker assumption that $H$ is semi-collapsible (see Defintion~\ref{dfn:semi-collapsible}).  However, we did not find an example of a graph that is semi-collapsible and not collapsible.
\end{remark}

It would be extremely interesting to extend Theorem~\ref{thm:main-Ramsey} to a broader class of graphs as well as to verify whether or not the function $c$ from the statement of the theorem has a limit as $n \to \infty$.

\subsection{Arithmetic properties}
\label{sec:arithm-prop}

We say that a set $Y$ of elements of some ambient additive group is \emph{$r$-Schur}, for some integer $r \ge 2$, and write that $Y \in \cS_r$ if every $r$-colouring of the elements of $Y$ admits a monochromatic sum, by which we mean three \emph{distinct} elements $a,b,c \in Y$ such that $a+ b= c$, all coloured the same way.  Schur's theorem~\cite{Sch17} states that, for any fixed $r$, the set $\br{N} \coloneqq \{1, \dotsc, N\}$ is $r$-Schur whenever $N$ is sufficiently large.  Similarly, given integers $k \ge 3$ and $r \ge 2$ and a set $Y$ of elements of some additive group, we say that $Y$ is \emph{$(k, r)$-van der Waerden} and write $Y \in \vdW(k,r)$ if every $r$-colouring of the elements of $Y$ admits a monochromatic $k$-term arithmetic progression.  The well-known theorem of van der Waerden~\cite{vdW27} states that, for all $k$ and $r$, the set $\br{N}$ is $(k, r)$-van der Waerden provided that $N$ is sufficiently large (as a function of $k$ and $r$).

R\"odl and Ruci\'nski~\cite{RodRuc95} proved that, for any $k \ge 3$ and any number of colours $r \ge 2$, the function $N^{-1/(k-1)}$ is a threshold for the property $\vdW(k,r)$ in the set $\br{N}_p$.  Soon afterwards, Graham, R\"odl, and Ruci\'nski~\cite{GraRodRuc96} showed that the function $N^{-1/2}$ is a threshold for the property $\cS_r$, for any $r \ge 2$.  Friedgut, H\`an, Person, and Schacht~\cite{FriHanPerSch16} showed that the former threshold is sharp whereas Schulenburg~\cite{Sch16PhD} proved the analogous statement for the latter threshold.  Both of these results are valid only for random subsets of the cyclic group $\Z_N$ and, crucially, only in the case $r=2$.  Even though the results of~\cite{GraRodRuc96, RodRuc95} are established for random subsets of $\br{N}$, their proofs can be easily adapted to yield analogus statements for subsets of $\Z_N$.  (In fact, the $1$-statements in the non-modular setting imply the $1$-statements in the modular setting.  As for the $0$-statements, the results presented in Section~\ref{sec:list-ramsey-problems} below generalise and strengthen both these results.)

We establish sharpness of the thresholds for $\vdW(k,r)$ and $\cS_r$ in random subsets of $\Z_N$ (in the case of Schur's theorem, we additionally require $N$ to be prime) for all $k \ge 3$ and all $r \ge 2$.  As in~\cite{SchSch18,Sch16PhD}, the reason for replacing $\br{N}$ with $\Z_N$ is that our approach requires the ground set to have a transitive group of symmetries that preserves the structure defining the property (Schur triples or $k$-APs).

\begin{thm}
  \label{thm:main-vdW}
  For all integers $k \ge 3$ and $r \ge 2$, there are constants $c_0 \le c_1$ and a function $c_0 \le c(N) \le c_1$ such that for all $\eps > 0$,
  \[
    \lim_{N \rightarrow \infty} \Pr \big( (\mathbb{Z}_N)_p \in \vdW(k,r) \big) =
    \begin{cases}
      1, & p \ge (1 + \eps) c(N) \cdot N^{-1/(k-1)}, \\
      0, & p \le (1 - \eps) c(N) \cdot N^{-1/(k-1)}.
    \end{cases}
  \]
\end{thm}

\begin{thm}
  \label{thm:main-Schur}
  For any integer $r \ge 2$, there are constants $c_0 \le c_1$ and a function $c_0 \le c(N) \le c_1$ such that for all $\eps > 0$,
  \[
    \lim_{\substack{N \rightarrow \infty \\ \text{is prime}}} \Pr \big( (\mathbb{Z}_N)_p \in \cS_r \big) =
    \begin{cases}
      1, & p \ge (1 + \eps) c(N) \cdot N^{-1/2}, \\
      0, & p \le (1 - \eps) c(N) \cdot N^{-1/2}.
    \end{cases}
  \]
\end{thm}

\subsection{List Ramsey problems}
\label{sec:list-ramsey-problems}

A main new theme in our analysis that paves the way to proving sharp threshold results in the case where the number of colours is larger than two is a list-colouring generalisation of the Ramsey problem.  The main result of this work views Ramsey results, such as Ramsey's theorem, van der Waerden's theorem, or Schur's theorem mentioned above, as statements about non-$r$-colourability of certain hypergraphs.  In the proof of this result, however, we encounter the more general problem of \emph{list colouring} a hypergraph from a given assignment of lists (of size two) to its vertices, which can be viewed as a list Ramsey problem.  Let us mention that a list colouring variant of Ramsey's theorem that is closely related to the one considered here was recently introduced by Alon, Buci\'c, Kalvari, Kuperwasser, and Szab\'o~\cite{AloBucKalKupSza21} and subsequently studied by Fox, He, Luo, and Xu~\cite{FoxHeLuoXu}.

In this section, we consider threshold phenomena associated with such list Ramsey problems in the context of van der Waerden's and Schur's theorems.  We say that a set $Y$ of elements of an additive group is \emph{list-Schur} if there exists an assignment of two-element lists to the elements of $Y$ such that every colouring of the elements of $Y$ with colours from their lists must admit a monochromatic sum.  We define the notion of a \emph{list-$k$-van der Waerden} sets analogously.  Note that every $2$-Schur (resp.\ $(2,k)$-van der Waerden) set is also list-Schur (resp.\ list-van der Waerden), but the converse is not necessarily true.  Our arguments yield, with very little extra work, the following strengthenings of the $0$-statements of the aforementioned results of~\cite{GraRodRuc96,RodRuc95} that establish the location of the threshold for van der Waerden's and Schur's theorems in random sets of integers.

\begin{thm}
  \label{thm:list-vdW}
  For every integer $k \ge 3$, there is a constant $c$ such that, for every sequence $X$ of sets of elements of an additive group such that $|X| \to \infty$ and every $p \le c \cdot |X|^{-1/(k-1)}$,
  \[
    \Pr\big(X_p \text{ is list-$k$-van der Waerden}\big) \to 0.
  \]
\end{thm}

\begin{thm}
  \label{thm:list-Schur}
  There is a constant $c$ such that, for every sequence $X$ of sets of elements of an additive group such that $|X| \to \infty$ and every $p \le c \cdot |X|^{-1/2}$,
  \[
    \Pr\big(X_p \text{ is list-Schur}\big) \to 0.
  \]
\end{thm}

In fact, both Theorems~\ref{thm:list-vdW} and~\ref{thm:list-Schur} are straightforward consequences of the following more general statement, whose short (two and a half pages) proof is given in Section~\ref{sec:choosability-almost-linear}.

\begin{restatable}{thm}{choosability}
  \label{thm:choosability-linear-hypergraphs}
  Suppose that $s \ge 3$ and that a sequence of $s$-uniform hypergraphs $\cH$ satisfies $\Delta_2(\cH) = O(1)$.  There is a positive $c$ such that, for every $p \le  c \cdot v(\cH)^{-1/(s-1)}$,
  \[
    \Pr\big(\cH[V_p] \text{ is $2$-choosable}\big) \to 1.
  \]
\end{restatable}

\subsection{Organisation}
\label{sec:organisation}

The rest of the paper is organised as follows.  In Section~\ref{sec:main-result}, we introduce our general theorem, which gives sufficient conditions (which we discuss in detail) on a sequence of hypergraphs that guarantee a sharp threshold for the property of $r$-colourability.  Section~\ref{sec:outline} offers an outline of the proof of this general theorem.  At the end of that section, we formulate two key statements that imply our result.

The bulk of the work is spent proving these statements.  Section~\ref{sec:tools} provides some external tools, such as the sharp threshold criterion and the hypergraph container lemma, which we then spend some time honing to our needs.  Subsequently, we prove both statements in Sections~\ref{sec:containers} and~\ref{sec:rainb-stars-const}.

Finally, having wrapped up the proof of the main result, Section~\ref{sec:applications} turns to applying the theorem in the various settings we mentioned previously: for graphs, arithmetic progressions, and Schur triples.

\section{The main result}
\label{sec:main-result}

As we have mentioned above, we will view Ramsey properties of random sets as statements about non-$r$-colourability of random hypergraphs.  Given a hypergraph $\cH$ with vertex set $V$ and a real $p \in [0,1]$, we will denote by $\cH_p$ the subhypergraph of $\cH$ induced by the random set $V_p$.  If the edges of $\cH$ are all copies of a structure $A$ in a set $V$ (for example, the edge sets of all copies of a graph $H$ in $K_n$), then non-$r$-colourability of $\cH_p$ is equivalent to the random set $V_p$ having the corresponding $r$-colour Ramsey property with respect to $A$ (in our example, the property $G_{n,p} \to (H)_r$).  Since we are interested in threshold phenomena, we will almost always consider infinite sequences of hypergraphs whose sizes tend to infinity.

Our main result supplies a sufficient condition on a sequence $\cH$ of uniform hypergraphs that guarantees that the property that $\cH_p$ is not $r$-colourable has a sharp threshold.  This sufficient condition is a conjunction of five assumptions.  We first give a brief overview of these five assumptions and state our result and return to discussing them in detail in the remainder of this section.

\subsection{Overview}
\label{sec:overview}

Suppose that $\cH$ is a sequence of $s$-uniform hypergraphs and let $r \ge 2$ be an integer.  The following function takes centre stage in our considerations:
\[
  \pH \coloneqq \left(\frac{v(\cH)}{e(\cH)}\right)^{1/(s-1)}.
\]
In order to phrase the five assumptions on the sequence $\cH$ that guarantee that non-$r$-colourability has a sharp threshold in $\cH_p$, we need to introduce three simple notions.  A \emph{star} in $\cH$ is a collection of $r-1$ edges that pairwise intersect in a single vertex called the centre of the star.  A \emph{constellation} is a collection of $s$ disjoint stars whose centres form an edge of $\cH$.  A star formed by edges $A_1, \dotsc, A_{r-1}$ and centred at $v$ is \emph{rainbow} if there are distinct colours $i_1, \dotsc, i_{r-1} \in \br{r}$ such that, for each $j$, all vertices of $A_j \setminus \{v\}$ are coloured $i_j$.  A constellation is rainbow if its $s$ constituent stars are rainbow and have the same colour pattern (the set $\{i_1, \dotsc, i_{r-1}\}$). The conjunction of the following five assumptions implies that non-$r$-colourability has a sharp threshold in $\cH_p$:

\begin{enumerate}[label=(A\arabic*)]
\item
  \label{item:ass-symmetry}
  \emph{Symmetry.}
  The hypergraph $\cH$ is \emph{symmetric} in the sense that its group of automorphisms $\Aut(\cH)$ acts transitively on the vertex set of $\cH$.
\item
  \label{item:ass-non-clusteredness}
  \emph{Non-clusteredness.}
  The hypergraph $\cH$ is \emph{non-clustered}, which means (roughly speaking) that $\cH$ satisfies the assumptions of the hypergraph container lemma with density parameter $\pH$. (See Section~\ref{sec:non-clusterdness}.)
\item
  \label{item:ass-weak-threshold}
  \emph{Weak threshold.}
  The function $\pH$ is a threshold for the property that $\cH_p$ is not $r$-colourable.
\item
  \label{item:ass-choosability}
  \emph{Choosability of typical bounded-sized subsets.}
  The random set $V(\cH)_{\pH}$ a.a.s.\ does not contain any set $W$ with $O(1)$ vertices for which $\cH[W]$ is not choosable from $2$-element lists of colours in $\br{r}$. (See Section~\ref{sec:choos-typic-subs}.)
\item
  \label{item:ass-star-constellation}
  \emph{The rainbow star-constellation property.}
  Every partial $r$-colouring of the vertices of $\cH$ that makes a constant proportion of its stars rainbow must make a constant proportion of its constellations rainbow as well. (See Section~\ref{sec:rainb-star-const}.)
\end{enumerate}

\begin{thm}
  \label{thm:main}
  Let $s \ge 3$ and $r \ge 2$ be integers and and let $\cH$ be a sequence of $s$-uniform hypergraphs.  It $\cH$ satisfies assumptions~\ref{item:ass-symmetry}--\ref{item:ass-star-constellation}, then there exists a function $\hat{p} = \Theta(\pH)$ such that the following holds for every positive $\eps$:
  \[
    \Pr\left(\text{$\cH_p$ is $r$-colourable}\right) \to
    \begin{cases}
      1 & \text{if $p \le (1-\eps)\hat{p}$}, \\
      0 & \text{if $p \ge (1+\eps)\hat{p}$}.
    \end{cases}
  \]
\end{thm}

Verifying assumptions~\ref{item:ass-symmetry} and~\ref{item:ass-non-clusteredness} for our applications of the theorem will be completely straightforward.  Assumption~\ref{item:ass-weak-threshold} is not at all easy to check, but, for the three applications of the main theorem we consider in this work, it had been established by earlier works.  Moreover, it is now standard to derive the $1$-statement in~\ref{item:ass-weak-threshold} from assumption~\ref{item:ass-non-clusteredness} and a property we term \emph{robust non-colourability}, see Section~\ref{sec:weak-threshold}.  Verifying assumption~\ref{item:ass-choosability}, which is closely related to establishing the $0$-statement in~\ref{item:ass-weak-threshold}, takes the most effort.  Assumption~\ref{item:ass-star-constellation} holds trivially when every set of $\Omega(v(\cH))$ vertices induces $\Omega(e(\cH))$ edges, which is the case in the context of van der Waerden's theorem and Ramsey's theorem for bipartite graphs.  In the two remaining applications of the theorem discussed here---Schur's theorem and Ramsey's theorem for nonbipartite, collapsible graphs---establishing this assumption requires a nontrivial argument.  Finally, let us mention that there are natural sequences of hypergraphs, for which one would expect non-$r$-colourability to have a sharp threshold, that satisfy assumptions~\ref{item:ass-symmetry}--\ref{item:ass-choosability}, but fail to satisfy~\ref{item:ass-star-constellation}, see Appendix~\ref{apx:graph_application}.

\subsection{Non-clusteredness}
\label{sec:non-clusterdness}

Given a hypergraph $\cH$ and a set $T \subseteq V(\cH)$, we will denote by $\deg_\cH(T)$ the degree of $T$ in $\cH$, that is,
\[
  \deg_\cH(T) \coloneqq |\{A \in \cH : T \subseteq A\}|.
\]
Further, for an integer $t \ge 1$, we let $\Delta_t(\cH)$ be the maximum degree of a $t$-element set of vertices, defined by
\[
  \Delta_t(\cH) \coloneqq \max\{ \deg_\cH(T) : T \subseteq V(\cH) \text{ and } |T| =t\}.
\]
We are now ready to define the notion of non-clusteredness from assumption~\ref{item:ass-non-clusteredness}.

\begin{dfn}
  A sequence of nonempty, $s$-uniform hypergraphs $\cH$ is called \emph{non-clustered} if
  \[
    \Delta_1(\cH) =  O\left(\frac{e(\cH)}{v(\cH)}\right) \qquad \text{and} \qquad \Delta_t(\cH) \ll \pH^{t-1} \cdot \frac{e(\cH)}{v(\cH)} \quad \text{for $t \in \{2, \dotsc, s-1\}$}.
  \]
\end{dfn}

\begin{fact}
  \label{fact:non-clustered-properties}
  Suppose that $\cH$ is a non-clustered sequence of $s$-uniform hypergraphs.
  \begin{enumerate}[label=(\roman*)]
  \item
    We have
    \[
      \Delta_s(\cH) = 1 = \pH^{s-1} \cdot \frac{e(\cH)}{v(\cH)}.
    \]
  \item
    If $s \ge 3$, then $\pH \to 0$.
  \end{enumerate}
\end{fact}

\begin{fact}
  \label{fact:non-clustered-examples}
  The following sequences of hypergraphs are non-clustered:
  \begin{itemize}
  \item
    The hypergraph of (the edge sets of) copies of a strictly $2$-balanced graph in $K_n$.
  \item
    The hypergraph of $k$-term arithmetic progressions in the cyclic group $\ZZ_N$.
  \item
    The hypergraph of Schur triples in any Abelian group.
  \end{itemize}  
\end{fact}

\subsection{Choosability of typical bounded-sized subsets}
\label{sec:choos-typic-subs}

Recall that a hypergraph $\cG$ is \emph{$2$-choosable} from a set $C$ of colours if, for every assignment $L \colon V(\cG) \to \binom{C}{2}$ of $2$-element lists of colours to the vertices of $\cG$, there exists a proper colouring of $\cG$ that assigns to each vertex $v \in V(\cG)$ a colour from its list $L_v$.  Given a hypergraph $\cH$ and integers $k \ge 1$ and $r \ge 2$, define
\[
  \cN_k(\cH) \coloneqq \big\{W \subseteq V(\cH) : \text{$|W| \le k$ and $\cH[W]$ is not $2$-choosable from $\br{r}$}\big\}.
\]
The precise statement of assumption~\ref{item:ass-choosability} is that, for every $k \ge 1$,
\begin{equation}
  \label{eq:ass-choosability}
  \Pr\big(V(\cH)_{\pH} \supseteq W \text{ for some $W \in \cN_k(\cH)$}\big) \to 0.
\end{equation}
In fact, our argument may require that~\eqref{eq:ass-choosability} holds also when we replace $\pH$ with some $p = \Theta(\pH)$.  Fortunately, these two statements are completely equivalent, see Lemma~\ref{lemma:local-coarseness}.

Finally, it is worth pointing out that the assumption on choosability of typical bounded-sized subsets is necessary for non-$2$-colourability of $\cH_p$ to have a sharp threshold at some $\hat{p} = \Theta(\pH)$.  Indeed, if~\eqref{eq:ass-choosability} fails for some constant $k$, then the probability that $V(\cH)_p$ contains some $W \in \cN_k(\cH)$, which clearly makes $\cH_p$ not $2$-colourable, is bounded away from zero for every $p = \Omega(\pH)$.

\subsection{The rainbow star-constellation property}
\label{sec:rainb-star-const}

Our final assumption~\ref{item:ass-star-constellation} has a much less obvious connection to the problem at hand, but it conveniently fits into our framework.

\begin{dfn}
  A collection $A_1, \dotsc, A_k$ of edges of a hypergraph $\cH$ is called a \emph{$k$-star} (or simply a \emph{star}) if there exists a vertex $v$ of $\cH$ such that $A_i \cap A_j = \{v\}$ for every pair of distinct $i, j \in \br{k}$; the vertex $v$ is called the \emph{centre} of the star and $A_1 \cup \dotsb \cup A_k$ is called the \emph{support} of the star.
\end{dfn}

\begin{dfn}
  A collection of stars with pairwise-disjoint supports whose centres form an edge of $\cH$ is called a \emph{constellation}.  The edge induced by the centres of the stars forming a constellation is called the \emph{base} of the constellation.
\end{dfn}

\begin{dfn}
  Suppose that some vertices of a hypergraph $\cH$ are coloured with the elements of $\br{r}$, for some integer $r \ge 2$, and let $i \in \br{r}$ be an arbitrary colour.  We say that an $(r-1)$-star $\{A_j\}_{j \in \br{r} \setminus \{i\}}$ centred at $v$ is \emph{$i$-rainbow} if, for every $j \in \br{r} \setminus \{i\}$, all vertices of $A_j \setminus \{v\}$ are coloured $j$.  A constellation is \emph{$i$-rainbow} if all stars comprising it are $i$-rainbow.  Finally, a star/constellation is \emph{rainbow} if it is $i$-rainbow for some $i \in \br{r}$.
\end{dfn}

A fairly straightforward calculation (Lemma~\ref{lemma:RS-RC-edges}) shows that every non-clustered sequence of $s$-uniform hypergraphs $\cH$ contains $\Theta\big(e(\cH)^{r-1} / v(\cH)^{r-2}\big)$ many $(r-1)$-stars and $\Theta\big(e(\cH)^{s(r-1)+1} / v(\cH)^{s(r-1)}\big)$ constellations of $(r-1)$-stars.  We will say that such a sequence $\cH$ has the rainbow star-constellation property for $r$ colours if every partial $r$-colouring of the vertices of $\cH$ that makes a constant proportion of all its $(r-1)$-stars rainbow also makes a constant proportion of all its constellations rainbow.

\begin{dfn}
  \label{dfn:rainbow-stars-constellations}
  Given an integer $r \ge 2$ and a sequence of $s$-uniform hypergraphs $\cH$, we say that $\cH$ has the \emph{rainbow star-constellation property} for $r$ colours if every partial colouring of $V(\cH)$ with elements of $\br{r}$ that induces $\Omega\big(e(\cH)^{r-1} / v(\cH)^{r-2}\big)$ rainbow stars must also induce $\Omega\big(e(\cH)^{s(r-1)+1} / v(\cH)^{s(r-1)}\big)$ rainbow constellations.
\end{dfn}

\begin{figure}%[htb]
\centering
\begin{subfigure}[b]{0.6\textwidth}
\centering
\begin{tikzpicture}
    \node (v1) at (2,1) {};
    \node (v2) at (1,0.5) {};
    \node (u1) at (2.2,0) {};
    \node (u2) at (1.1,0) {};
    \node (w1) at (2,-1) {};
    \node (w2) at (1,-0.5) {};
    \node (x) at (0,0) {};

	\foreach \t in {0} {
	
    \begin{scope}[fill opacity=0.3]
    \filldraw[rounded corners=30, fill=gray!70] 
    ($(v1)+(0.9,\t + 0.45)$) 
        -- ($(v2) + (-0.3,\t + 0.3)$) 
        -- ($(x) + (-0.9,\t -0.45)$)
        -- ($(v2) + (0.3,\t -0.3)$)
        -- cycle;
     \filldraw[rounded corners=30, fill=gray!70] 
    ($(u1)+(1,\t + 0)$) 
        -- ($(u2) + (0,\t + 0.4)$) 
        -- ($(x) + (-1,\t + 0)$)
        -- ($(u2) + (0,\t -0.4)$)
        -- cycle;
     \filldraw[rounded corners=30, fill=gray!70] 
    ($(w1)+(0.9,\t -0.45)$) 
        -- ($(w2) + (0.3,\t + 0.3)$) 
        -- ($(x) + (-0.9,\t + 0.45)$)
        -- ($(w2) + (-0.3,\t -0.3)$)
        -- cycle;
    \end{scope}

    \foreach \v in {1,2} {
        \fill[fill=black] ($(\t,0) + (v\v)$) circle (0.1);
        \fill[fill=black] ($(\t,0) + (u\v)$) circle (0.1);
        \fill[fill=black] ($(\t,0)+ (w\v)$) circle (0.1);
    }
    \fill ($(x) + (\t,0)$) circle (0.1);
    }
\end{tikzpicture}
\caption{A star}
\end{subfigure}
\begin{subfigure}[b]{0.6\textwidth}
\centering
\begin{tikzpicture}
    \node (v1) at (1,2) {};
    \node (v2) at (0.5,1) {};
    \node (u1) at (0,2.2) {};
    \node (u2) at (0,1.1) {};
    \node (w1) at (-1,2) {};
    \node (w2) at (-0.5,1) {};
    \node (x) at (0,0) {};

	\foreach \t in {0,3,6} {
	
    \begin{scope}[fill opacity=0.3]
    \filldraw[rounded corners=30, fill=gray!70] 
    ($(v1)+(\t + 0.45,0.9)$) 
        -- ($(v2) + (\t + 0.3,-0.3)$) 
        -- ($(x) + (\t -0.45,-0.9)$)
        -- ($(v2) + (\t -0.3,0.3)$)
        -- cycle;
     \filldraw[rounded corners=30, fill=gray!70] 
    ($(u1)+(\t + 0,1)$) 
        -- ($(u2) + (\t + 0.4,0)$) 
        -- ($(x) + (\t + 0,-1)$)
        -- ($(u2) + (\t -0.4,0)$)
        -- cycle;
     \filldraw[rounded corners=30, fill=gray!70] 
    ($(w1)+(\t -0.45,0.9)$) 
        -- ($(w2) + (\t + 0.3,0.3)$) 
        -- ($(x) + (\t + 0.45,-0.9)$)
        -- ($(w2) + (\t -0.3,-0.3)$)
        -- cycle;
    \end{scope}

    \foreach \v in {1,2} {
        \fill[fill=teal] ($(\t,0) + (v\v)$) circle (0.1);
        \fill[fill=red] ($(\t,0) + (u\v)$) circle (0.1);
        \fill[fill=orange] ($(\t,0)+ (w\v)$) circle (0.1);
    }
    \fill ($(x) + (\t,0)$) circle (0.1);
    }
    
    \begin{scope}[fill opacity=0.3]
    \filldraw[rounded corners=50, fill=gray!70] 
    ($(x)+(-1.4,0)$) 
        -- ($(x) + (3,0.5)$) 
        -- ($(x) + (6 + 1.4,0)$)
        -- ($(x) + (3,-0.5)$)
        -- cycle;
    \end{scope}

\end{tikzpicture}
\caption{A rainbow constellation}
\end{subfigure}

\end{figure}

\subsection{The weak threshold assumption}
\label{sec:weak-threshold}

We conclude this section with a short discussion on how assumption~\ref{item:ass-weak-threshold} might possibly be derived from~\ref{item:ass-non-clusteredness} and \ref{item:ass-choosability} and yet another `supersaturation' assumption on $\cH$ that we term robust non-$r$-colourability.

\begin{dfn}
  Given an integer $r \ge 2$, we say that a sequence of hypergraphs $\cH$ is \emph{robustly non-$r$-colourable} if every $r$-colouring of the vertices of $\cH$ makes a constant proportion of the edges of $\cH$ monochromatic, that is, if every $c \colon V(\cH) \to \br{r}$ satisfies
  \[
    \sum_{i=1}^r e\big(\cH[c^{-1}(i)]\big) = \Omega\big(e(\cH)\big).
  \]
\end{dfn}

The following facts can be derived from the corresponding Ramsey statements using simple averaging arguments and are thus considered folklore.

\begin{fact}
  \label{fact:robustly-noncolourable-examples}
  The following sequences of hypergraphs are robustly non-$r$-colourable:
  \begin{itemize}
  \item
    The hypergraph of (the edge sets of) copies of a fixed nonempty graph in $K_n$.
  \item
    The hypergraph of $k$-term arithmetic progressions in the cyclic group $\ZZ_N$.
  \item
    The hypergraph of Schur triples in any Abelian group.
  \end{itemize}  
\end{fact}

If a sequence of hypergraph is non-clustered and robustly non-$r$-colourable, then $O(\pH)$ is an upper bound on any threshold function of non-$r$-colourability.  This fact can be shown by a straightforward adaptation of the argument of Nenadov and Steger~\cite{NenSte16}, who showed that robust $(r+1)$-colourability of the sequence $\cH$ of hypergraphs representing copies of a given graph $H$ in $K_n$ implies non-$r$-colourability of a typical $\cH_p$ for all $p \gg n^{-1/m_2(H)}$. (See also~\cite[Section~8]{BalSam20} for a slightly different version of this argument that shows the exact statement of the proposition below.)  For the sake of completeness, we recreate this argument in Appendix~\ref{apx:1-statement}.

\begin{prop}
  \label{prop:robust-non-col-1-statement}
  Let $r \ge 2$ and $s \ge 2$ be integers. For every non-clustered, robustly non-$r$-colourable sequence $\cH$ of $s$-uniform hypergraphs, there exists a constant $C$ such that, for every $p \ge C\pH$,
  \[
    \Pr\big(\text{$\cH_p$ is $r$-colourable}\big) \le \exp\left(-\Omega\big(p \cdot v(\cH)\big)\right).
  \]
\end{prop}

\begin{remark}
  If $s \ge 3$, then Fact~\ref{fact:non-clustered-properties} implies that $p \cdot v(\cH) \to \infty$ when $p = \Omega(\pH)$.
\end{remark}

It would be worth looking into the following problem, motivated by~\cite[Lemma~6]{NenSte16} and~\cite[Meta-Theorem]{NenPerSkoSte17}.

\begin{problem}
  Does assumption~\ref{item:ass-choosability}, with $r=2$, imply that $\Omega(\pH)$ is a lower bound on the threshold for every non-clustered sequence of $s$-uniform hypergraphs, provided that $s=3$?
\end{problem}

Perhaps one could show this after strengthening the assumption of being non-clustered by further assuming that, for each $t \in \{2, \dotsc, s-1\}$, the inequality $\Delta_t(\cH) \ll \pH^{t-1} \cdot \frac{e(\cH)}{v(\cH)}$ hides some polynomial (in $v(\cH)$) factor.  We remark that the three families of sequences of hypergraphs from the statement of Fact~\ref{fact:non-clustered-examples} all enjoy such strengthened non-clusteredness property.

\section{An outline of the proof}
\label{sec:outline}

Assume that $r \ge 2$ and $s \ge 3$ and suppose that $\cH$ is a sequence of $s$-uniform hypergraphs that satisfies assumptions~\ref{item:ass-symmetry}--\ref{item:ass-star-constellation}.  For the sake of brevity, we will denote the vertex set of $\cH$ by $V$, its cardinality by $N$, and write that a set $W \subseteq V$ is $r$-colourable if and only if the induced subhypergraph $\cH[W]$ is.  Finally, assume to the contrary that the threshold for $V_p$ not being $r$-colourable is coarse.

\subsection{Boosters and dichotomy}
\label{sec:boosters-dichotomy}

Our point of departure will be Friedgut's criterion, in Bourgain's formulation, which tells us that there is a positive constant $c$, a sequence $p$ satisfying
\begin{equation}
  \label{eq:Vp-colourable-bounded-away}
  c \le \Pr\left(V_p \text{ in not $r$-colourable}\right) \le 1- c,
\end{equation}
and a family $\bB$ of constant-sized subsets of $V$ such that
\begin{equation}
  \label{eq:booster-positive-prob}
  \Pr \left( \exists B \in \bB \text{ s.t.\ } B \subseteq V_p\right) > c
\end{equation}
and every $B \in \bB$ is a \emph{booster}.

\begin{dfn}
  Given $\delta > 0$ and $p \in [0,1]$, a set $B \subseteq X$ is called a \emph{$(p,\delta)$-booster} if
  \[
    \Pr \left(V_p \text{ is not $r$-colourable} \mid B \subseteq V_p \right) > \Pr\left(V_p \text{ is not $r$-colourable} \right) + \delta.
  \]
\end{dfn}

Observe that assumption~\ref{item:ass-weak-threshold} and~\eqref{eq:Vp-colourable-bounded-away} imply that $p = \Theta(\pH)$.  Using the symmetry assumption \ref{item:ass-symmetry}, we will expand on Friedgut's criterion and show that a typical sample $Z \sim V_p$ exemplifies a sort of dichotomy, in the following precise sense.

\begin{step}
  \label{step:dichotomy}
  There are constants $\alpha, \eps > 0$, an integer $K$, and $p = \Theta(\pH)$ such that, for any family $\cF \subseteq \binom{\cH}{\le K}$ with
  \begin{equation}
    \label{eq:boosters-typical}
    \Pr \left( \exists B \in \cF \text{ s.t.\ } B\subseteq V_p \right) < \alpha,
  \end{equation}
  there is a set $B_0 \in \binom{V(\cH)}{\le K} \setminus \cF$ such that the following holds.
  For infinitely many values~$N$, the set $Z \sim V_p$ satisfies the following with probability larger than $\alpha$:
  \[
    \Pr \left(Z \cup h(B_0) \text{ is not $r$-colourable} \mid Z \right) > \alpha,
  \]
  where $h$ is taken u.a.r.\ from the set of symmetries of $\cH$, and
  \[
    \Pr \left(Z \cup V_{\eps p} \text{ is not $r$-colourable} \mid Z\right) \le 1/2.
  \]
\end{step}

\begin{remark}
  Note that the second part of the dichotomy, stating that the probability that $Z \cup V_{\eps p}$ is not $r$-colourable is strictly less than one, implies that $Z$ must be $r$-colourable.
\end{remark}

In other words, we are guaranteed the existence of two sets, $B_0$ and $Z$, with two properties that seem at odds. While a positive proportion of the symmetric copies of the constant-sized $B_0$ \emph{interact} with $Z$---that is, $Z$ ceases to be $r$-colourable once we add them---the probability that the random set $V_{\eps p}$ interacts with $Z$ is bounded away from one.  We will call these interacting symmetric copies of $B_0$ \emph{activated} boosters.

Furthermore, we are allowed to trim some undesirable properties from both $B_0$ and $Z$.  In the case of $B_0$, this can be done by requiring that $B_0 \notin \cF$ whereas in the case of $Z$, this can be done as $Z \sim V_p$ satisfies the assertion with probability bounded away from zero.   We should remark at this point that the family $\cF$ we are going to choose will be \emph{symmetric}, i.e., if $B \in \cF$, then $h(B) \in \cF$ for every $h \in \Aut(\cH)$. Therefore, if $B_0 \notin \cF$, then the same is true for every other symmetric copy of it.

Our aim is to use these two statements to get a contradiction.  Specifically, we will show that, with a suitable choice of properties for $B_0$ and $Z$ that exploit assumptions~\ref{item:ass-non-clusteredness}--\ref{item:ass-star-constellation}, the existence of many activated boosters implies that $Z \cup V_{\eps p}$ is not $r$-colourable with probability arbitrarily close to one.  The methodology of the argument will be very much in tune with the previous works of Friedgut, H\'an, Person, and Schacht~\cite{FriHanPerSch16}, Schacht and Schulenburg~\cite{SchSch18}, and Schulenburg~\cite{Sch16PhD}, who established the existence of sharp thresholds for various Ramsey properties in the case where there are only two colours. However, in order to argue for sharpness in three or more colours, we require several novel ideas.

We will present our proof in two rounds. The first round will be a (spoiler alert) failed attempt, which will still show in essence how to utilise the assumption about choosability of typical subsets of bounded size (which we will enforce on $B_0$ via an appropriate choice of $\cF$) to gain structural information on proper $r$-colourings of $Z$.  The second round will address the breaking point of that approach, a very large union bound over $r$-colourings of $Z$, and remedy it using the Hypergraph Container Lemma of Saxton and Thomason~\cite{SaxTho15} and also of Balogh, Morris, and Samotij~\cite{BalMorSam15}. This approach will, in turn, require us to strengthen one of the claims made in the first round.

\subsection{First attempt}
\label{sec:failed-attempt}

We start with Step~\ref{step:dichotomy} and get $Z$ and a family of interacting boosters, all of which are symmetric copies of $B_0$.  Define the hypergraph $\cB$ on the vertex set $V$ whose edges are all symmetric copies of $B_0$, that is, all $h(B_0)$ with $h \in \Aut(\cH)$.  Since $\cH$ is symmetric, $\cB$ is regular and, therefore,
\[
  \Delta_1(\cB) = |B_0| \cdot \frac{e(\cB)}{v(\cH)} \le K\cdot N^{-1} \cdot e(\cB).
\]

We will impose some structural assumptions on $B_0$.  It is natural to require that $B_0$ is $r$-colourable, since otherwise $Z \cup h(B_0)$ would be not $r$-colourable with probability one, which would in turn suggest that the threshold is actually coarse.  (This was the only assumption on $B_0$ imposed in previous works~\cite{FriHanPerSch16, SchSch18, Sch16PhD}.)  We will go one step further.  Instead of ensuring that $B_0$ is only $r$-colourable, we will make use of assumption~\ref{item:ass-choosability} and require that it is $2$-choosable from lists in $\br{r}$.  We may do so as~\ref{item:ass-choosability} implies that the family $\cF$ comprising all non-$2$-choosable subsets of $V$ with at most $K$ vertices satisfies the condition~\eqref{eq:boosters-typical} in Step~\ref{step:dichotomy}.  Note that this property is symmetric, so requiring it from $B_0$ guarantees that it is fulfilled by all $B \in \cB$.

Let $\cBZ \subseteq \cB$ comprise only the copies of $B_0$ that interact with $Z$, that is,
\[
  \cBZ \coloneqq \big\{ B \in \cB : Z \cup B \text{ is not $r$-colourable}\big\}.
\]
The first assertion of Step~\ref{step:dichotomy} translates to $e(\cBZ) > \alpha \cdot e(\cB)$. One of the desirable properties of $Z$ would allow us to find a subfamily $\cBZ' \subseteq \cBZ$ of our activated boosters---satisfying $e(\cBZ') \ge \alpha/2 \cdot e(\cB)$---whose members interact with $Z$ in a very well-behaved manner.  First, if $B \in \cBZ'$, then $B$ and $Z$ are disjoint.  Further, suppose that some activated booster $B \in \cBZ$ does not intersect $Z$.  Since both $Z$ and $B$ are $r$-colourable, the fact that $Z \cup B$ is not means that there is an edge of $\cH[Z \cup B]$ that intersects both $B$ and~$Z$.  Call the set of all such edges the \emph{interface} between $B$ and $Z$.  The second desirable property, which we may impose using assumption~\ref{item:ass-non-clusteredness}, is that, if $B \in \cBZ'$, then each edge in the interface between $B$ and $Z$ has exactly one vertex in $B$ (and the remaining vertices in~$Z$).

\begin{dfn}
  Given $U \subseteq V$, let $\Col(U)$ denote the set of all proper $r$-colourings of $U$, i.e., colourings that do not admit a monochromatic edge.  
\end{dfn}

\begin{dfn}
  A set $U \subseteq V$ \emph{threatens} a vertex $v \in V$ if there are $u_1, \dots, u_{s-1} \in U$ such that $\{u_1, \dotsc, u_{s-1}, v\}$ is an edge of $\cH$.
\end{dfn}

Let $\psi \in \Col(Z)$ be a proper colouring of $Z$.  We say that $\psi$ \emph{forces} a vertex $v \in V$ to the colour $i \in \br{r}$ if $v$ is threatened by $\psi^{-1}(j)$ for every $j \neq i$.  (Equivalently, $\psi$ forces $v$ to the colour $i$ if and only if $v$ is the centre of an $i$-rainbow star.)  The motivation behind the definition is the following fact: If $\psi$ forces $v$ to the colour $i$, then every proper colouring of $Z \cup \{v\}$ that extends $\psi$ must assign the colour $i$ to $v$.  Let $F_i(\psi)$ denote the set of vertices that $\psi$ forced to the colour $i$ and let $F(\psi) \coloneqq F_1(\psi) \cup \dotsb \cup F_r(\psi)$ be the set of forced vertices.

\begin{step}[first attempt]
  \label{step:weak-forcing}
  There is some $\lambda > 0$ such that $|F(\psi)| \ge \lambda N$ for all $\psi \in \Col(Z)$.
\end{step}

We prove this statement in two steps.  First, we show that in order to force a constant fraction of the vertices, it is enough to force at least one vertex in a constant fraction of the activated boosters.  Second, we show that every colouring $\psi \in \Col(Z)$ forces a vertex inside every activated booster.

\begin{lemma}[Many forced boosters $\to$ many forced vertices]
  \label{lem:many_forced_boosters}
  If $F$ is a set which intersects a constant fraction of the activated boosters, then $|F| = \Omega(N)$.
\end{lemma}
\begin{proof}
  Let $\beta > 0$ be some constant such that $F$ intersects $\beta \cdot e(\cB'_Z)$ activated boosters.  Since each vertex belongs to at most $\Delta_1(\cB)$ boosters, we have
  \[
    \beta \cdot \alpha/2 \cdot e(\cB) \le \beta \cdot e(\cBZ') \le |F| \cdot \Delta_1(\cB) \le |F| \cdot K \cdot N^{-1} \cdot e(\cB),
  \]
  which implies that $F$ has $\Omega(N)$ elements.
\end{proof}

\begin{lemma}[Activated booster $\to$ forced booster]
  \label{lemma:activated-booster}
  Every $\psi \in \Col(Z)$ forces at least one vertex in every booster.
\end{lemma}
\begin{proof}
  Suppose that this is not true and there is a booster $B \in \cB'_Z$ whose every vertex is not threatened by at least two colour classes of $\psi$.  Since $B$ is $2$-choosable, we would be able to find a colouring $\phi \in \Col(B)$ using only these non-threatening colours.  However, since $Z \cup B$ is not $r$-colourable, there must be an edge $\{u_1, \dotsc, u_{s-1}, b\} \in \cH$, where $u_1, \dots u_{s-1} \in Z$ and $b \in B$, such that $\psi(u_1) = \dots = \psi(u_{s-1}) = \phi(b)$.  (Indeed, since $B \in \cBZ'$, every edge in the interface between $B$ and $Z$ has exactly one vertex in $B$.) This would mean that, contrary to our assumption, $\psi^{-1}(\phi(b))$ threatens $b$.
\end{proof}

By the pigeonhole principle, one of the forced sets, say $F_i(\psi)$, has $\Omega(N)$ vertices.  We would like to show that this set induces $\Omega(e(\cH))$ edges.  (This is essentially equivalent to $\Omega(e(\cH))$ edges being the base edge of an $i$-rainbow constellation.)  Such a claim is always true when $\cH$ describes a `degenerate' structure; e.g., if $\cH$ is the hypergraph of copies of a bipartite graph or when $\cH$ is the hypergraph of arithmetic progressions.  Unfortunately, some hypergraphs of interest (e.g., the hypergraph of copies of any non-bipartite graph) contain independent sets of cardinality $\Omega(N)$.  However, using the assumption that $\cH$ has the rainbow star-constellation property, we will be able to argue that a.a.s.\ $Z \sim V_p$ has the property that, for every $\psi \in \Col(Z)$, every large set $F_i(\psi)$ must induce many edges.  With foresight, we state a stronger version of this property that extends also to \emph{partial} $r$-colourings of $Z$.

\begin{step}
  \label{step:sparse-rsc}
  For any $\beta > 0$, there exists a $\gamma > 0$ such that a.a.s.\ $Z \sim V_p$ has the following property:  For any partial $r$-colouring $\psi$ of $Z$ and any $i \in \br{r}$, if $|F_i(\psi)| \ge \beta N$, then $F_i(\psi)$ induces at least $\gamma \cdot e(\cH)$ edges.
\end{step}

The proof of this statement will appear in Section~\ref{sec:rainb-stars-const}.  We will just mention that it employs the Hypergraph Container Lemma to transfer the supersaturation statement given by the rainbow star-constellation assumption into the sparse regime.

Note now that if $e \subseteq F_i(\psi)$ is an edge, then $\psi$ cannot be extended to $e$ while staying proper.  This is because the elements of $F_i(\psi)$ must all be coloured $i$, but this makes $e$ monochromatic.  The fact that $F_i(\psi)$ contains $\Omega(e(\cH))$ edges, together with the assumption that the hypergraph $\cH$ is not clustered, makes it extremely unlikely that such an edge will fail to appear in $V_{\eps p}$.  The following estimate follows from Janson's inequality (Theorem~\ref{thm:Janson}).

\begin{lemma}
  \label{lemma:extending-fixed-colouring}
  Suppose that $A \subseteq V$ induces $\Omega(e(\cH))$ edges. The probability that $V_{\eps p}$ avoids all edges of $A$ is bounded from above by $\exp(-\Omega(pN))$.
\end{lemma}
% \begin{proof}
%   This claim follows from Janson's inequality (see Theorem~\ref{thm:Janson}). Write $X$ for the number of edges that appear in $V_{\eps p} \cap T$.  We bound the expectation of $X$ from below as follows:
%   \[
%     \mu \coloneqq \Ex[X] = \Omega\big(p^s e(\cH)\big).
%   \]
%   Further, we bound the pseudo-variance (see Definition~\ref{dfn:pseudo-variance}) of the sequence $\bA$ of events $B \subseteq V_{\eps p}$ for all edges $B \in \cH[T]$ from above:
%   \[
%     \pVar(\bA) = \sum_{i= 1}^s \sum_{\substack{B, B' \in \cH[T] \\ |B \cap B'| = i}} \Pr(B \cup B' \subseteq V_{\eps p})  = O\left( \sum_{i=1}^s p^{2s-i} \cdot e(\cH)  \cdot \Delta_i(\cH) \right).
%   \]
%   Recalling that $p = \Theta(\pH)$, we have
%   \[
%     \frac{\mu^2}{\pVar(\bA)} = \Omega\left(\min_{i \in \br{s}} \frac{\pH^i \cdot  e(\cH)}{\Delta_i(\cH)}\right) = \Omega\big(\pH \cdot v(\cH)\big) = \Omega(pN).
%   \]
%   Thus, Janson's inequality gives us that $\Pr(X = 0) = \exp(- \Omega(pN))$.
% \end{proof}

Thus far, we have demonstrated that the probability that a \emph{given} $\psi \in \Col(Z)$ could be extended to $Z \cup V_{\eps p}$ is bounded from above by $\exp(-\Omega(pN))$.  However, in order to get the desired contradiction to the first assertion of Step~\ref{step:dichotomy}, we would like to show that the probability that \emph{some} proper colouring of $Z$ can be extended to $V_{\eps p}$ tends to zero.  To this end, let us take the union bound over all proper colourings.  Alas, the only bound we have at our disposal is $|\Col(Z)| = \exp(O(|Z|)) = \exp(O(pN))$, which is not good enough.  Moreover, we cannot significantly improve the upper bound established by Lemma~\ref{lemma:extending-fixed-colouring}, since the set $V_{\eps p}$ is empty with probability approximately $\exp(-\eps p N)$.

\subsection{Second attempt}
\label{sec:second-attempt}

Since the breaking point of our first attempt was the union bound over all proper colourings of $Z$, we will try to make this union bound more efficient by excluding many colourings in $\Col(Z)$ at once.  To this end, note that if $\psi_0$ were a partial colouring of $Z$ that forced $\Omega(N)$ vertices to some colour, we could still apply Step~\ref{step:sparse-rsc} and Lemma~\ref{lemma:extending-fixed-colouring} to learn that the probability that $\psi_0$ can be extended to $V_{\eps p}$ is at most $\exp(-\Omega(pN))$.  Moreover, if $\psi_0$ cannot be extended to $V_{\eps p}$, then neither can any proper colouring $\psi \in \Col(Z)$ that extends $\psi_0$.  It thus suffices to find a family of $\exp(o(|Z|)) = \exp(o(pN))$ partial colourings of $Z$, each forcing a constant proportion of the vertices to some colour, such that every element of $\Col(Z)$ is an extension of some member of the family.

As was hinted before, we will employ the Container Lemma~\cite{BalMorSam15, SaxTho15} to find such a family.  To this end, we will identify every colouring $\psi \colon Z \to \br{r}$ with the set $\{ (z, \psi(z)) \colon z \in Z\}$ and define a hypergraph $\cT$ on the vertex set $Z \times \br{r}$ such that every $\psi \in \Col(Z)$ will correspond to an independent set of $\cT$.  The Container Lemma will allow us to construct a family $\cC$ of subsets of $Z \times \br{r}$, which we will call \emph{containers}, such that:
\begin{enumerate}[label=(C\arabic*)]
\item
  \label{item:T-containers-1}
  Any independent set of $\cT$, and thus every $\psi \in \Col(Z)$, is contained in some $C \in \cC$.
\item
  \label{item:T-containers-2}
  Every container $C \in \cC$ induces only a small fraction of the edges of $\cT$.
\end{enumerate}

We will say that a set $C \subseteq Z \times \br{r}$ is a \emph{restricted} colouring if every $z \in Z$ has at least one `available' colour in $C$, i.e., if $(z, i) \in C$ for some $i \in \br{r}$. Note that a set $C \subseteq Z \times \br{r}$ that is not a restricted colouring cannot contain any colouring of $Z$.  Given a restricted colouring $C \subseteq Z \times \br{r}$, we define its \emph{determined} colouring $\psi_C$ to be the maximal partial colouring agreed upon by all the colourings contained by $C$.  (That is, $\psi_C(z) = i$ if and only if $i$ is the unique colour such that $(z, i) \in C$.)  Note that~\ref{item:T-containers-1} implies that every $\psi \in \Col(Z)$ extends the determined colouring $\psi_C$ of some $C \in \cC$.  The crux of our argument is showing that one may define $\cT$ in such a way that condition~\ref{item:T-containers-2} implies that, for every $C \in \cC$, the determined colouring $\psi_C$ forces many vertices to some colour, so that we can apply Step~\ref{step:sparse-rsc}.

The precise definition of the hypergraph $\cT$ is somewhat technical, but the idea behind it is fairly straightforward.  Given a booster $B \in \cBZ'$, we will say that colourings $\psi \colon Z \to \br{r}$ and $\phi \colon B \to \br{r}$ are \emph{consistent} if no edge in the interface between $B$ and $Z$ is monochromatic under $\psi \cup \phi$.  Since each such interface edge has exactly one vertex in $B$ (by the definition of $\cBZ'$), two colourings $\psi$ and $\phi$ are consistent if and only if no vertex $b \in B$ is threatened by $\psi^{-1}(\phi(b))$.  A key observation is that the fact that two colourings $\psi$ and $\phi$ as above are consistent always has a small `witness' set in $Z \times \br{r}$.  Indeed, for a given $\phi$, one can certify that $\psi$ is consistent with $\phi$ by specifying the values of $\psi$ on vertices of $Z$ that belong to edges of the interface between $B$ and $Z$.  Such minimal `witness' sets for all $B \in \cBZ'$ and all proper colourings $\phi \in \Col(B)$ are the edges of our hypergraph $\cT$.  The fact that every $B \in \cBZ'$ is an activated booster means that no two colourings $\psi \in \Col(Z)$ and $\phi \in \Col(B)$ are consistent and, consequently, every $\psi \in \Col(Z)$ is an independent set of $\cT$.

We say that a restricted colouring $C$ \emph{inter-activates} a booster $B \in \cBZ'$ if every proper colouring $\phi \in \Col(B)$ is inconsistent with each colouring of $Z$ contained in $C$.  Our definition of $\cT$ guarantees that a restricted colouring $C$ will induce an edge in $\cT$ for every booster $B \in \cBZ'$ which it fails to inter-activate.  (This edge will be a witness to some pair $\phi \in \Col(B)$ and $\psi \subseteq C$ being consistent.)  Since every container $C \in \cC$ induces only a small proportion of all edges of $\cT$, it must therefore inter-activate all but a small fraction of all boosters in $\cBZ'$.  Finally, an argument similar to the one used in the proof of Lemma~\ref{lemma:activated-booster} shows that if $C$ inter-activates a booster $B$, then $\psi_C$ must force at least one vertex in $B$ to some colour.  (A key insight here is realising that the property that $C$ inter-activates a booster depends only on the determined colouring $\psi_C$.)  This implies, by Lemma~\ref{lem:many_forced_boosters}, that $\psi_C$ must force $\Omega(N)$ vertices.

Our final concern is that the family $\cC$ of containers will have at most $\exp(o(|Z|))$ elements.  In order to guarantee this, we will need to demonstrate some control over the edge set of $\cT$.  In fact, it will be sufficient to bound the largest size of an edge of $\cT$ by an absolute constant and to show that $\Delta_1(\cT) \cdot |Z| = O(e(\cT))$ and $\Delta_2(\cT) \cdot |Z| \ll e(\cT)$.  Luckily, this will be possible yet again by trimming further undesirable properties from~$Z$ in Step~\ref{step:dichotomy}.  For example, the expected number of interacting edges between $V_p$ and $B$ is bounded by some constant and consequently the witness sets will also be constant sized. Using this and other related properties of $Z$ and the booster family, we will be able to show that the hypergraph $\cT$ is `well-behaved', which will allow us to use the Container Lemma to derive the following.

\setcounter{step}{1}
\begin{step}[refined]
  \label{step:containers}
  There is a family $\cC$ of $\exp(o(|Z||))$ subsets of $Z \times \br{r}$ with the following properties:
  \begin{enumerate}
  \item Every proper colouring of $Z$ is an extension of the determined colouring $\psi_C$ of some container $C \in \cC$.
  \item For every container $C \in \cC$, the determined colouring $\psi_C$ forces $\Omega(N)$ vertices.
  \end{enumerate}
\end{step}
Our union bound argument is now brought back to life and the proof is finally settled.

\subsection{Finalising the argument}

Our goal for the next three sections is to tie up the loose ends of the proof, by filling in the gaps in the outline presented above.  Recall that we want to show that, assuming non-$r$-colourability does not have a sharp threshold in $\cH_p$, the following statements hold for $Z \sim V_p$ with probability bounded away from zero:
\begin{enumerate}[label=(\Roman*)]
\item
  \label{item:finalising-1}
  For some constant $\eps > 0$, the hypergraph $\cH\big[Z \cup V_{\eps p}\big]$ is $r$-colourable with probability at least $1/2$.
\item
  \label{item:finalising-2}
  Each proper $r$-colouring of $Z$ extends one of $\exp(o(|Z|))$ partial colourings, each of which forces $\Omega(N)$ elements in $V$ to some colour.
\item
  \label{item:finalising-3}
  Each of our partial colourings admits a set of $\Omega(e(\cH))$ edges whose all vertices are forced to the same colour.
\end{enumerate}
As we argued above, this would lead to a contradiction.  Indeed, each individual partial colouring from the family cannot be extended to any set that contains one of the $\Omega(e(\cH))$ `forced' edges.  By Lemma~\ref{lemma:extending-fixed-colouring}, the probability that $V_{\eps p}$ omits all these edges is at most $\exp\big(-\Omega(pN)\big)$.  A union bound over all partial colourings yields that $\cH\big[Z \cup V_{\eps p}\big]$ is not $r$-colourable with probability tending to one.

We prove statements~\ref{item:finalising-1}, \ref{item:finalising-2}, and~\ref{item:finalising-3} over the course of Sections~\ref{sec:tools}, \ref{sec:containers}, and~\ref{sec:rainb-stars-const}.  More precisely, Section~\ref{sec:tools} contains the formal statements of Friedgut's sharp threshold criterion, along with a routine corollary thereof (Step~\ref{step:dichotomy} in the proof outline), as well as the statements of the Hypergraph Container Lemma and Janson's inequality, of which Lemma~\ref{lemma:extending-fixed-colouring} is a simple corollary.

Next, in Section~\ref{sec:containers}, we show how the abundance of boosters can be used to construct the efficient family of partial colourings (the refined version of Step~\ref{step:containers} in the proof outline).  In order to formally phrase this result using the notions from assumption~\ref{item:ass-star-constellation}, we shift the terminology from forced vertices to (centres of) rainbow stars.  The result is the following theorem.

\begin{restatable}{thm}{manyrainbowstars}
  \label{thm:many-rainbow-stars}
  Suppose that $s \ge 3$ and $r \ge 2$ and let $\cH$ be a sequence of $s$-uniform hypergraphs that satisfies assumptions~\ref{item:ass-symmetry}--\ref{item:ass-choosability}. If non-$r$-colourability does not have a sharp threshold in $\cH_p$, then there exist a~constant $\eps > 0$ and a~subsequence $p = \Theta(\pH)$ such that, letting $Z \sim V(\cH)_p$, the following holds with probability at least~$\eps$:
  \begin{enumerate}[label=(\roman*)]
  \item
    \label{item:many-rainbow-stars-eps-p}
    The hypergraph $\cH\big[Z \cup V(\cH)_{\eps p}\big]$ is $r$-colourable with probability at least $1/2$.
  \item
    \label{item:many-rainbow-stars-containers}
    There exists a family $\Psi$ of $\exp\big(o(|Z|)\big)$ partial $\br{r}$-colourings of $Z$ such that
    \begin{enumerate}[label=(\alph*)]
    \item
      \label{item:rainbow-stars-cont-extending}
      every proper $\br{r}$-colouring of $\cH[Z]$ extends some partial colouring in $\Psi$ and
    \item
      \label{item:rainbow-stars-cont-forcing}
      for every colouring in $\Psi$, some $\Omega\big(v(\cH)\big)$ vertices of $\cH$ are centres of rainbow stars.
    \end{enumerate}
  \end{enumerate}
\end{restatable}

Finally, Section~\ref{sec:rainb-stars-const} provides a proof of the following sparse analogue of the rainbow star-constellation property (Step~\ref{step:sparse-rsc} in the proof outline).  We remark again that this theorem is trivial if the hypergraph $\cH$ is `degenerate' in the sense that every subset of $\Omega(N)$ vertices induces $\Omega(e(\cH))$ edges.  As a result, we are spared from a bulk of the proof when $\cH$ is the hypergraph of copies of a bipartite graph or the hypergraph of arithmetic progressions of a prescribed length.

\begin{restatable}{thm}{rainbowstarsconstellations}
  \label{thm:rainbow-stars-constellations}
  Suppose that $s \ge 3$, $r \ge 2$, and $\eps > 0$. Let $\cH$ be a sequence of $s$-uniform hypergraphs that satisfies assumptions \ref{item:ass-non-clusteredness} and~\ref{item:ass-star-constellation} and suppose that $Z \sim V(\cH)_p$ for some $p = \Theta(\pH)$. With probability at least $1-\eps$, for every partial colouring $\psi$ of $Z$ with elements of $\br{r}$, if $\Omega\big(v(\cH)\big)$ vertices of $\cH$ are centres of rainbow stars, then $\Omega\big(e(\cH)\big)$ edges of $\cH$ are bases of rainbow constellations.
\end{restatable}

Theorems~\ref{thm:many-rainbow-stars} and~\ref{thm:rainbow-stars-constellations}, supplemented with Lemma~\ref{lemma:extending-fixed-colouring}, imply Theorem~\ref{thm:main}.

\section{Preliminaries and tools}
\label{sec:tools}

\subsection{Coarse thresholds}
\label{sec:coarse-thresholds}

The first piece of machinery that we require is a characterisation of properties with coarse thresholds. In the proof outline, this was captured by Step~\ref{step:dichotomy}, which we aim to formalise and prove here.  In general, Friedgut's work~\cite{Fri99} and Bourgain's subsequent extension of it~\cite[Appendix]{Fri99} provide a criterion for the appearance of a sharp threshold. The  criterion states that every property that fails to have a sharp threshold must correlate with a `local' property, i.e., the appearance of a bounded-sized subset.  In other words, if a property $\cP$ has a coarse threshold, then there is another, local, property $\cP'$---with the same threshold as $\cP$--- such that both are positively correlated.

For our purpose, we use a reformulation of Bourgain's aforementioned result, which appears in~\cite{Fr05}. We will introduce a relevant definition and then state the theorem.

\begin{dfn}[Boosters]
  Suppose that $\cP$ is a property of subsets of a finite set $V$. Given a $p \in [0,1]$ and a positive number $\delta$, we call a set $B \subseteq V$ a \emph{$(p,\delta)$-booster} if
  \[
    \Pr(V_p \cup B \in \cP) \ge \Pr(V_p \in \cP) + \delta.
  \] 
\end{dfn}

\begin{thm}[The Sharp Threshold Criterion]
  \label{thm:Friedgut-Bourgain}
  For all positive $\alpha$ and $C$, there exist positive $\delta$, $\eta$, $p_0$, and $K$ such that the following holds. Suppose that $\cP$ is a monotone property of subsets of a finite set $V$ and, for each $p \in [0,1]$, let $\mu(p) \coloneqq \Pr(V_p \in \cP)$. If, for some $0< p \le p_0$
  \[
    \alpha \le \mu(p) \le 1-\alpha \qquad \text{and} \qquad \mu'(p) \le C/p,
  \]
  then there is a family $\cB \subseteq \binom{V}{\le K}$ satisfying
  \[
    \Pr(B \subseteq V_p \text{ for some } B \in \cB) > \eta
  \]
  such that every $B \in \cB$ is a $(p,\delta)$-booster for $\cP$.
\end{thm}

We will now use the criterion to derive a general result in the spirit of Step~\ref{step:dichotomy} of the proof outline.  First, it guarantees the existence of boosters with `typical' characteristics.  Second, it asserts the following dichotomy for properties $\cP$ with coarse thresholds: while there are many (constant-sized) boosters whose addition to $V_p$ lands us immediately in $\cP$, adding to $V_p$ a random set of density $\eps p$ does not increase the probability of being in $\cP$ substantially.  We again introduce a definition and move on to stating and proving the result.

\begin{dfn}
We say that a set $B \subseteq V$ is an \emph{active booster} for a set $Z \subseteq V$ if $Z \cup B \in \cP$.
\end{dfn}

\begin{prop}
  \label{prop:Friedgut-Bourgain}
  Let $\cP$ be a nontrivial, monotone property of subsets of a (sequence of) finite set(s) $V$. For each $p \in [0,1]$, let $\mu(p) \coloneqq \Pr(V_p \in \cP)$ and let $\phat \coloneqq \mu^{-1}(1/2)$. If $\cP$ does not have a sharp threshold and $\phat = o(1)$, then there exist positive constants $\delta$, $\eps$, and $K$ and an infinite subsequence $p = \Theta(\phat)$ such that:
  \begin{enumerate}[label=(\arabic*)]
  \item
    \label{item:FB-boosters}
    For every family $\cF$ of subsets of $V$ satisfying
    \[
      \Pr\big(V_p \supseteq W \text{ for some $W \in \cF$}\big) < \eps,
    \]
    there is a $(p,\delta)$-booster in $\binom{V}{\le K} \setminus \cF$.
  \item
    \label{item:FB-active-boosters}
    For any family $\cB$ of $(p,\delta)$-boosters, letting $Z \sim V_p$, the following holds with probability at least $\eps$:
    \begin{enumerate}
    \item
      $\Pr(Z \cup V_{\eps p} \in \cP \mid Z) \le 1/2$ and
    \item
      at least $\eps |\cB|$ elements of $\cB$ are active boosters for $Z$.
    \end{enumerate}
  \end{enumerate}
\end{prop}
\begin{proof}
  Suppose that $\cP$ does not have a sharp threshold. This means that there are constants $c_1 < c_2$ and $\alpha > 0$ such that
  \[
    \alpha \le \mu(c_1 \cdot \phat) \le \mu(c_2 \cdot \phat) \le 1-\alpha
  \]
  on some infinite subsequence.  Let $C \coloneqq 4c_2/(c_2-c_1)$.

  \begin{claim}
    For every $\gamma > 0$, there is a $p \in (c_1 \cdot \phat, c_2 \cdot \phat)$ such that
    \begin{enumerate}[label=(\alph*)]
    \item
      \label{item:coarse-claim-1}
      $\mu'(p) \le C/p$ and
    \item
      \label{item:coarse-claim-2}
      $\mu(p+\gamma p) - \mu(p) \le C \cdot \gamma$
    \end{enumerate}
  \end{claim}
  \begin{proof}
    Fix an arbitrary $\gamma'$ satisfying $0 < \gamma' < (c_2-c_1)/2$. Since $0 \le \mu \le 1$, we have
    \[
      \int_{c_1}^{c_2-\gamma'} \big(\mu((x+\gamma') \cdot \phat) - \mu(x \cdot \phat)\big)\, dx \le \int_{c_2-\gamma'}^{c_2} \mu(x \cdot \phat) \, dx \le \gamma'.
    \]
    In particular, there must be a $p'$ satisfying $c_1 \cdot \phat < p' < (c_2-\gamma') \cdot \phat$ such that
    \[
      \mu(p' + \gamma' \cdot \phat) - \mu(p') \le \frac{\gamma'}{c_2-c_1-\gamma'} \le \frac{2\gamma'}{c_2-c_1}.
    \]
    Further, as $\mu$ is increasing, there must be a $p \in (p', p' + (\gamma'/2) \cdot \phat) \subseteq (c_1 \cdot \phat, c_2 \cdot \phat)$ with
    \[
      \mu'(p) \le \frac{\mu(p'+(\gamma'/2)\cdot\phat)-\mu(p')}{(\gamma'/2) \cdot \phat} \le \frac{4}{(c_2-c_1) \cdot \phat} \le \frac{4c_2}{(c_2-c_1) \cdot p} = \frac{C}{p},
    \]
    as required. Finally, we show that the value $p$ chosen above also satisfies inequality~\ref{item:coarse-claim-2}. If $\gamma \ge 1/C$, then the inequality is vacuous. Otherwise, if $\gamma < 1/C$, we let $\gamma' = 2c_2\gamma$. Since $\gamma' < (c_2-c_1)/2$, by the definition of $C$, we have
    \[
      p + \gamma p \le p + \gamma \cdot c_2 \phat = p + (\gamma'/2) \cdot \phat \le p' + \gamma' \cdot \phat.
    \]
    Since $p > p'$ and $\mu$ is increasing, we have
    \[
      \mu(p+\gamma p) - \mu(p) \le \mu(p' + \gamma' \cdot \phat) - \mu(p') \le \frac{2\gamma'}{c_2-c_1} = C\gamma,
    \]
    as desired.
  \end{proof}

  Let $\cB$ be an arbitrary nonempty subset of the family of $(p, \delta)$-boosters supplied by Theorem~\ref{thm:Friedgut-Bourgain}. Fix a small positive constant $\gamma$, let $B$ be a uniformly chosen random element of $\cB$, and, for each $Z \subseteq V$, define
  \[
    f(Z) = \Pr(Z \cup B \in \cP) \qquad \text{and} \qquad g(Z) = \Pr(Z \cup V_{\gamma p} \in \cP).
  \]
  Our goal is to show that, for some positive constant $\eps$,
  \[
    \Pr\big(f(V_p) \ge \eps \text{ and } g(V_p) \le 1/2\big) \ge \eps.
  \]
  To this end, note first that Markov's inequality and the definition of a $(p, \delta)$-booster imply
  \[
    \Pr\big(f(V_p) < \eps\big) = \Pr\big(1 - f(V_p) > 1 - \eps\big) \le \frac{1-\Ex[f(V_p)]}{1-\eps} \le \frac{1-\mu(p)-\delta}{1-\eps}.
  \]
  Since $g(Z) = 1$ when $Z \in \cP$ and since $V_p \cup V_{\gamma p}$ is stochastically dominated by $V_{p+\gamma p}$, we have, again by Markov's inequality,
  \[
    \begin{split}
      \Pr\big(g(V_p) > 1/2\big) & = \Pr(V_p \in \cP) + \Pr\big(g(V_p) - \1_{V_p \in \cP} > 1/2\big) \\
      & \le \mu(p) +  2\big(\Ex[g(V_p)] - \Ex[\1_{V_p \in \cP}]\big) \\
      & \le 2\mu(p+\gamma p) - \mu(p) \le \mu(p) + 2C\gamma.
    \end{split}
  \]
  Finally,
  \[
    \begin{split}
      \Pr\big(g(V_p) \le 1/2 \text{ and } g(V_p) \ge \eps\big) & \ge 1- \Pr\big(g(V_p) > 1/2\big) - \Pr\big(g(V_p) < \eps\big) \\
      & \ge 1-\mu(p) - \frac{1-\mu(p)-\delta}{1-\eps} - 2C\gamma \\
      & \ge \delta - \frac{\eps}{1-\eps} -2C\gamma \ge \delta/2 \ge \eps,
    \end{split}
  \]
  where the last two inequalities hold provided that $\eps$ and $\gamma$ are sufficiently small (as functions of $\delta$ and $C$ only).
\end{proof}

Finally, we will make use of the following simple lemma, which formalises the fact that all `local' properties (i.e., properties of containing a bounded-sized subset from a given family) have coarse thresholds.

\begin{lemma}[Local coarseness]
  \label{lemma:local-coarseness}
  Let $\cF$ be a family of subsets of a set $V$, all of which have at most $K$ elements, and let $\mu(p) \coloneqq \Pr(\exists B \in \cF \colon B \subseteq V_p)$.  Then, for any $c \in (0,1)$, we have $\mu(cp) \ge c^K \mu(p)$.
\end{lemma}
\begin{proof}
  We couple $V_p$ and $V_{cp}$ by viewing $V_{cp}$ as the random subset $(V_p)_c \subseteq V_p$. Therefore,
  \[
    \frac{\mu(cp)}{\mu(p)} = \Pr(\exists B \in \cF \colon B \subseteq (V_p)_c \mid \exists B \in \cF \colon B \subseteq V_p) \ge c^K. \qedhere
  \]
\end{proof}

\subsection{Containers}

Our second main tool is a hypergraph container lemma for almost independent sets.  The standard versions of the container lemma, proved independently by Balogh, Morris, and Samotij~\cite{BalMorSam15} and by Saxton and Thomason~\cite{SaxTho15}, assert that every uniform hypergraph $\cG$ admits a relatively small collection of \emph{containers} for independent sets, that is, a family $\cC$ of subsets of $V(\cG)$, each containing only few edges of $\cG$, such that every independent set of $\cG$ is contained in some member of $\cC$.  Here, we will require a strengthening of this result that supplies a small collection of containers for the larger family of almost independent sets.  Such stronger version of the container lemma was proved by Saxton and Thomason.  Since the precise phrasing of this result that best fits our framework, Theorem~\ref{thm:containers} below, differs somewhat from~\cite[Corollary~3.6]{SaxTho15}, we include a short derivation of the former from the latter in Appendix~\ref{sec:omitted-proofs}.

\begin{restatable}{thm}{containers}
  \label{thm:containers}
  For every positive integer $k$ and all positive reals $\eps$ and $K$, there exist an integer $t$ and a positive real $\delta$ such that the following holds. Suppose that a nonempty $k$-uniform (multi)hypergraph $\cG$ with vertex set $V$ and a positive real $\tau$ satisfy
  \[
    \Delta_\ell(\cG) \le K \tau^{\ell-1} \cdot \frac{e(\cG)}{v(\cG)}
  \]
  for every $\ell \in \br{k}$. Then, there exists a function $f \colon \cP(V)^t \to \cP(V)$ with the following properties:
  \begin{enumerate}[label=(\roman*)]
  \item
    For every set $I \subseteq V$ satisfying $e(\cG[I]) \le \delta \tau^k e(\cG)$, there are $S_1, \dotsc, S_t \subseteq I$ with at most $\tau v(\cG)$ elements each such that $I \subseteq f(S_1, \dotsc, S_t)$.
  \item
    For every $S_1, \dotsc, S_t \subseteq V$, the set $f(S_1, \dotsc, S_t)$ induces fewer than $\eps e(\cG)$ edges in $\cG$.
  \end{enumerate}
\end{restatable}

\subsection{Janson's inequality}
\label{sec:janson-inequality}

The final auxiliary result required by our argument is the well-known concentration inequality of Janson, which gives strong upper bounds on the lower tail probabilities of random variables counting how many sets from a given collection are contained in a binomial random subset.  The version of Janson's inequality stated below differs from~\cite[Theorem~1]{Jan90} in the choice of notation, but it is otherwise equivalent.

\begin{dfn}
  \label{dfn:pseudo-variance}
  Suppose that $\bA = (A_1, \dotsc, A_k)$ is a sequence of (not necessarily distinct) events.  The \emph{pseudo-variance} of $\bA$ is
  \[
    \pVar(\bA) \coloneqq \sum_{i \sim j} \Pr(A_i \cap A_j),
  \]
  where the sum ranges over all ordered pairs $i, j \in \br{k}$ such that $A_i$ and $A_j$ are not independent.
\end{dfn}

\begin{remark}
  Note that $\Var'(A_1, \dotsc, A_k)$ is an upper bound on the variance of $\1_{A_1} + \dotsb + \1_{A_k}$, hence the name pseudo-variance.
\end{remark}

\begin{thm}[Janson's inequality~\cite{Jan90}]
  \label{thm:Janson}
  Suppose that $\Omega$ is a finite set, let $B_1, \dotsc, B_k$ be a sequence of (not necessarily distinct) subsets of $\Omega$, and let $R \sim \Omega_p$ for some $p \in [0,1]$. For each $i \in [k]$, let $X_i$ be the indicator of the event $A_i$ that $B_i \subseteq R$ and let $X \coloneqq \sum_i X_i$. Then, for any $0 \le t \le \Ex[X]$,
  \[
    \Pr\big( X\le \Ex[X]-t \big) \le \exp \left(-\frac{t^2}{2\pVar(A_1, \dotsc, A_k)}\right).
  \]
\end{thm}

We finish this section by deriving the following generalisation of Lemma~\ref{lemma:extending-fixed-colouring}.

\begin{lemma}
  \label{lemma:Janson-non-clustered}
  Suppose that $s \ge 2$ and let $\cH$ be a sequence of $s$-uniform hypergraphs that satisfies assumption~\ref{item:ass-non-clusteredness}. If $p = \Theta(\pH)$, then, for every $\cH' \subseteq \cH$ with $\Omega(e(\cH))$ edges,
  \[
    \Pr\big(e(\cH'_p) = 0\big) = \exp\left(-\Omega\big(p \cdot v(\cH)\big)\right).
  \]
\end{lemma}
\begin{proof}
  Write $V$ for $V(\cH)$ and $X$ for the number of edges of $\cH'_p = \cH'[V_p]$.  We bound the expectation of $X$ from below as follows:
   \[
     \mu \coloneqq \Ex[X] = p^s e(\cH') = \Omega\big(p^s e(\cH)\big).
   \]
   Further, we bound the pseudo-variance of the sequence $\bA$ of events $B \subseteq V_p$ for all edges $B \in \cH'$ from above:
   \[
     \pVar(\bA) = \sum_{i= 1}^s \sum_{\substack{B, B' \in \cH' \\ |B \cap B'| = i}} \Pr(B \cup B' \subseteq V_p)  = O\left( \sum_{i=1}^s p^{2s-i} \cdot e(\cH)  \cdot \Delta_i(\cH) \right).
   \]
   Since $p = \Theta(\pH)$ and $\cH$ is non-clustered,
   \[
     \frac{\mu^2}{\pVar(\bA)} = \Omega\left(\min_{i \in \br{s}} \frac{\pH^i \cdot  e(\cH)}{\Delta_i(\cH)}\right) = \Omega\big(\pH \cdot v(\cH)\big).
   \]
   Thus, Janson's inequality gives us that $\Pr(X = 0) = \exp\left(- \Omega\big(p \cdot v(\cH)\big) \right)$.
\end{proof}

\section{Containers for colourings}
\label{sec:containers}

In this section, we prove the following theorem, which encapsulates Steps~\ref{step:dichotomy} and~\ref{step:containers} from the proof outline.

\manyrainbowstars*

\subsection{Setup}
\label{sec:setup}

Suppose that $s$, $r$, and $\cH$ are as in the statement of the theorem.  For the sake of brevity, throughout this section, we shall write $V$ in place of $V(\cH)$.  Assume that non-$r$-colourability does not have a sharp threshold in $\cH_p$.  Since $\pH$ is a threshold function for this property, by assumption~\ref{item:ass-weak-threshold}, Proposition~\ref{prop:Friedgut-Bourgain} supplies constants $\delta$, $\eps$, and $K$ and an infinite sequence $p = \Theta(\pH)$ satisfying~\ref{item:FB-boosters} and~\ref{item:FB-active-boosters} in the proposition.  Since $\cH$ satisfies assumption~\ref{item:ass-choosability}, invoking~\ref{item:FB-boosters} in Proposition~\ref{prop:Friedgut-Bourgain} with $\cF$ being the family $\cN_K(\cH)$ defined in Section~\ref{sec:choos-typic-subs}, we obtain a $(p,\delta)$-booster $B_0 \subseteq V$ of cardinality at most $K$ such that $\cH[B_0]$ is $2$-choosable from~$\br{r}$.

Let $\cB$ be the (multi)hypergraph on $V$ whose edges are the images of $B_0$ via all automorphisms of $\cH$.  Since non-$r$-colourability is preserved under automorphisms of $\cH$, every edge of $\cB$ is also a $(p, \delta)$-booster.  Moreover, since $\cH$ is symmetric, the hypergraph $\cB$ is degree-regular and, consequently,
\[
  \Delta_1(\cB) = |B_0| \cdot \frac{e(\cB)}{v(\cH)} \le K \cdot \frac{e(\cB)}{v(\cH)}.
\]

Let $\cZ$ be the family of all subsets $Z \subseteq V$ such that $\cH[Z \cup V_{\eps p}]$ is $r$-colourable with probability at least $1/2$, but, for at least $\eps$-proportion of $B \in \cB$, the hypergraph $\cH[Z \cup B]$ is not $r$-colourable. Property~\ref{item:FB-active-boosters} in Proposition~\ref{prop:Friedgut-Bourgain} states that $\Pr(V_p \in \cZ) \ge \eps$.

Finally, we define several constants. Let $K_\cH$ be a constant satisfying
\[
  \Delta_1(\cH) \le K_\cH \cdot \frac{e(\cH)}{v(\cH)} \qquad \text{and} \qquad \frac{p}{\pH} \le K_\cH.
\]
Further, pick $\lambda = \eps^2/8$ and let $L$ be an integer satisfying
\begin{equation}
  \label{eq:lambda-L}
  \frac{\big(K \cdot K_\cH^s\big)^L}{L!} \le \lambda.
\end{equation}

\subsection{A recap of the proof outline}
\label{sec:outline-1}

Our goal now is to utilize the boosters in $\cB$ in order to construct, for a typical $Z \in \cZ$, a family $\Psi$ of $\exp\big(o(|Z|)\big)$ partial colourings of $Z$ that satisfies~\ref{item:rainbow-stars-cont-extending} and~\ref{item:rainbow-stars-cont-forcing} from the statement of the theorem.

Given a $Z \in \cZ$, let
\[
  \cBZ \coloneqq \big\{B \in \cB : \text{$\cH[Z \cup B]$ is not $r$-colourable}\big\}
\]
be the family of boosters from $\cB$ that are active for $Z$.  By definition, no proper $\br{r}$-colouring of $Z$ can be extended to $Z \cup B$, for any $B \in \cBZ$.  The assumption that $\cH[B]$ is $2$-choosable from $\br{r}$ implies that, if a proper partial $\br{r}$-colouring $\psi_0$ of $Z$ cannot be extended to a booster in $B \in \cBZ$, then at least one vertex of $B$ must be the centre of a rainbow star (equivalently, it must be forced to some colour), see Lemma~\ref{lemma:activated-booster}.  Therefore, it will be sufficient to make sure that our partial colourings do not extend to a constant proportion of the boosters in $\cBZ$.

We say that $A \in \cH$ is an \emph{interacting} edge between two sets if $A$ is contained in their union and intersects both of them.  Note that a proper partial colouring $\psi_0$ cannot be extended to $B$ if and only if, for every proper colouring $\varphi$ of $B$, some edge is monochromatic under $\psi_0 \cup \varphi$.  This monochromatic edge cannot be fully contained in either $Z$ or~$B$, because both $\psi_0$ and $\varphi$ are proper, and is therefore an interacting edge.

Define the \emph{interface} of $Z$ with $B$ to be the set\footnote{This is a minor departure from the terminology used in the proof outline, where the interface between $Z$ and $B$ was the family of all edges $A$ that are interacting between $Z$ and $B$.}
\[
  I(B,Z) \coloneqq \{ A \setminus B : A \in \cH \text{ is an interacting edge between } Z \text{ and } B  \}.
\]
We will say that a partial colouring $\psi_0$ of $Z$ and a colouring $\varphi$ of $B$ are \emph{consistent} if no edge that is interacting between $Z$ and $B$ is monochromatic under $\psi_0 \cup \varphi$.  In this language, a partial colouring $\psi_0$ of $Z$ can be extended to $B$ precisely when $\psi_0$ is consistent with some proper colouring $\varphi$ of $B$.  Note that the interface $I(B, Z)$ contains all the information about $\psi_0$ that is needed to determine whether this is the case.

We say that $C \subseteq Z \times \br{r}$ is a \emph{restricted} colouring of $Z$ if, for every $z \in Z$, there is some colour $i \in \br{r}$ for which $(z, i) \in C$.  Recall that, in this context, we identify colourings $\psi$ of $Z$ with the restricted colouring $\{(z,\psi(z)) \colon z \in Z \}$.  Given a restricted colouring $C$, we define its \emph{determined} colouring to be the partial colouring $\psi_C$, where $\psi_C(z) = i$ if every proper colouring $\psi \subseteq C$ has $\psi(z) = i$ (equivalently, if $i \in \br{r}$ is the unique colour such that $(z,i) \in C$).

As we wrote in the outline, we will first define a hypergraph $\cT$ on the vertex set $Z \times \br{r}$ in such a way that proper colourings of $Z$ will be independent in $\cT$. We will then let the family $\Psi$ of partial colourings be the determined colourings of a family of (restricted colourings that are) containers for the independent sets of $\cT$. The edges of $\cT$ will correspond to colourings of interfaces $I(B,Z)$, with $B \in \cBZ$, that are consistent with some proper colouring of $B$.  In particular, we need to exercise control over the interfaces that $Z$ has with the boosters in order to invoke the container lemma in an efficient way.

The rest of the section is organised as follows.  We first prove that $Z$ is disjoint from almost all boosters in $\cBZ$ and that the interface between $Z$ and these boosters is bounded in size and well-behaved.  Next, we will formally define the hypergraph $\cT$ whose containers provide the family of partial colourings satisfying the first two properties. Finally, we will bound the degree sequence of $\cT$ in order show that the family of containers, and therefore partial colourings, is small.

\subsection{Active boosters and interactions}
\label{sec:active-boosters}

Let $\cBZ'$ be the hypergraph obtained from $\cBZ$ by retaining only those $B \in \cBZ$ that satisfy all of the following:
\begin{enumerate}[label=(B\arabic*)]
\item
  \label{item:booster-cleanup-1}
  The sets $B$ and $Z$ are disjoint.
\item
  \label{item:booster-cleanup-2}
  Every nonempty set in $I(B,Z)$ has $s-1$ elements.
\item
  \label{item:booster-cleanup-3}
  The sets in $I(B,Z)$ are pairwise disjoint.\footnote{However, there could still be pairs of different $A, A' \in \cH$ such that $A \setminus B = A' \setminus B \in I(B,Z)$.}
\item
  \label{item:booster-cleanup-4}
  The family $I(B,Z)$ contains at most $L$ sets.
\end{enumerate}
The following two lemmas show that $\cBZ \setminus \cBZ'$ is small for a typical $Z \sim V_p$.

\begin{lemma}
  \label{lemma:booster-cleanup-123}
  If $Z \sim V_p$, then the expected number of $B \in \cB$ that fail one of~\ref{item:booster-cleanup-1}--\ref{item:booster-cleanup-3} is $o\big(e(\cB)\big)$.
\end{lemma}

\begin{lemma}
  \label{lemma:booster-cleanup-4}
  If $Z \sim V_p$, then the expected number of $B \in \cB$ that satisfy~\ref{item:booster-cleanup-2} and~\ref{item:booster-cleanup-3} but fail~\ref{item:booster-cleanup-4} is at most $\lambda e(\cB)$.
\end{lemma}

Finally, let
\[
  \cZ' = \big\{Z \in \cZ : e(\cBZ') \ge (\eps/2) \cdot e(\cB)\big\}.
\]
By Lemmas~\ref{lemma:booster-cleanup-123} and~\ref{lemma:booster-cleanup-4} and since $\lambda = \eps^2/8$, we have, for $Z \sim V_p$,
\[
  \begin{split}
    \Pr(Z \in \cZ') & \ge \Pr(Z \in \cZ) - \Pr\big(e(\cBZ \setminus \cBZ') \ge (\eps/2) \cdot e(\cB)\big) \\
    & \ge \eps - \frac{\Ex\big[e(\cBZ \setminus \cBZ')\big]}{(\eps/2) \cdot e(\cB)} \ge \eps - 4\lambda/\eps \ge \eps/2.
\end{split}
\]

\begin{proof}[{Proof of Lemma~\ref{lemma:booster-cleanup-123}}]
  Suppose that $Z \sim V_p$ and let $X_1$, $X_2$, and $X_3$ denote the numbers of sets $B$ that fail conditions~\ref{item:booster-cleanup-1}, \ref{item:booster-cleanup-2}, and \ref{item:booster-cleanup-3}, respectively. We have
  \[
    \Ex[X_1] = \sum_{B \in \cB} \Pr(B \cap Z \neq \emptyset) \le \sum_{B \in \cB} \Ex[|B \cap Z|] = \sum_{v \in V} p \deg_\cB(v) \le v(\cH) \cdot p\Delta_1(\cB) \ll e(\cB),
  \]
  since $p \ll 1$ and $\Delta_1(\cB) = O\big(e(\cB)/v(\cH)\big)$. Since $X_2$ is at most the number of pairs $(A,B) \in \cH \times \cB$ such that $A \setminus B \subseteq Z$ and $1 \le |A \setminus B| \le s-2$, we have
  \[
    \Ex[X_2] \le \sum_{B \in \cB} \sum_{t=1}^{s-2} \binom{|B|}{s-t} \cdot \Delta_{s-t}(\cH) \cdot p^t \le e(\cB) \cdot \sum_{r=2}^{s-1} 2^K \cdot \Delta_r(\cH) \cdot \pH^{s-r},
  \]
  where we used the fact that every edge of $\cB$ contains at most $K$ vertices. Since $\cH$ is non-clustered, we have, for every $r \in \{2, \dotsc, s-1\}$,
  \[
    \Delta_r(\cH) \cdot \pH^{s-r} \ll \pH^{s-1} \cdot \frac{e(\cH)}{v(\cH)} = 1,
  \]
  implying that $\Ex[X_2] \ll e(\cB)$. Finally, we can bound $X_3 - X_2$ by the number of triples $(A,A',B) \in \cH^2 \times \cB$ such that $A \cup A' \subseteq B \cup Z$, $1 \le |(A \cap A') \setminus B| \le s-2$, and $|A \cap B| = |A' \cap B| = 1$. We therefore get the following upper bound by  counting the options for picking $B$, then $A$, then the size of the intersection $(A \cap A') \setminus B$, and finally $A'$:
  \[
    \Ex[X_3 - X_2] \le \sum_{B \in \cB} |B| \cdot \Delta_1(\cH) \cdot \sum_{r=1}^{s-2} \binom{s-1}{r} \cdot |B| \cdot \Delta_{r+1}(\cH) \cdot \pH^{2s-2-r}.
  \]
  Using our assumptions that every edge of $\cB$ contains at most $K$ vertices, $\Delta_1(\cH) \le O\big( e(\cH) / v(\cH) \big)$, and $\Delta_{r+1}(\cH) \ll \pH^r \cdot e(\cH) / v(\cH)$ for all $r \in \{1, \dotsc, s-2\}$, we have
  \[
    \Ex[X_3 - X_2] \ll e(\cB) \cdot \pH^{2s-2} \cdot \left(\frac{e(\cH)}{v(\cH)}\right)^2 = e(\cB).
  \]
  The proof of the lemma is now complete.
\end{proof}

\begin{proof}[{Proof of Lemma~\ref{lemma:booster-cleanup-4}}]
  For $B \in \cB$ and $Z \subseteq V$, let $i(B,Z)$ be the largest integer $\ell$ such that there are $A_1, \dotsc, A_\ell \in \cH$ satisfying:
  \begin{enumerate}[label=(\roman*)]
  \item
    \label{item:iBZ-1}
    $A_1 \cup \dotsb \cup A_\ell \subseteq B \cup Z$,
  \item
    \label{item:iBZ-2}
    $|A_i \cap B| = 1$ for every $i \in \br{\ell}$,
  \item
    \label{item:iBZ-3}
    $A_1 \setminus B, \dotsc, A_\ell \setminus B$ are pairwise disjoint.
  \end{enumerate}
  Observe that, if $B$ satisfies~\ref{item:booster-cleanup-2} and~\ref{item:booster-cleanup-3}, then $|I(B,Z)| = i(B,Z)$. In particular, the assertion of the lemma will follow if we show that $\Pr\big(i(B,Z) > L\big) \le \lambda$.

  To this end, for a positive integer $\ell$, let $\cA_\ell$ denote the collection of all sequences $A_1, \dotsc, A_\ell \in \cH$ that satisfy~\ref{item:iBZ-2} and \ref{item:iBZ-3} above and note that
  \[
    |\cA_\ell| \le \big(|B| \cdot \Delta_1(\cH)\big)^\ell \le \big(K \cdot \Delta_1(\cH)\big)^\ell.
  \]
  Since, for every $(A_1, \dotsc, A_\ell) \in \cA_\ell$, the set $(A_1 \cup \dotsb \cup A_\ell) \setminus B$ has precisely $(s-1)\ell$ elements, we have
  \[
    \Pr\big(A_1 \cup \dotsb \cup A_\ell \subseteq B \cup Z\big) = p^{(s-1)\ell}.
  \]
  Let $X_\ell$ be the number of sequences in $\cA_\ell$ that satisfy~\ref{item:iBZ-1}. Since each such sequence has distinct coordinates and~\ref{item:iBZ-1}--\ref{item:iBZ-3} are invariant under any permutation of the sequence~$(A_1, \dotsc, A_\ell)$, we may conclude that
  \[
    \Pr\big(i(B, V_p) \ge \ell\big) \le \frac{\Ex[X_\ell]}{\ell!} = \frac{|\cA_\ell| \cdot p^{(s-1)\ell}}{\ell!} \le \frac{\big(K \cdot \Delta_1(\cH) \cdot p^{s-1}\big)^\ell}{\ell!}.
  \]
  The assertion of the lemma follows because our assumptions imply that
  \[
    K \cdot \Delta_1(\cH) \cdot p^{s-1} \le K \cdot K_\cH^s \cdot \frac{e(\cH)}{v(\cH)} \cdot \pH^{s-1} = K \cdot K_\cH^s,
  \]
  see~\eqref{eq:lambda-L}.
\end{proof}
  
\subsection{The hypergraph of transversals}
\label{sec:transversals}

We will now formally define the hypergraph~$\cT$.  Our goal is the following: For every restricted colouring $C$, if its determined colouring $\psi_C$ is consistent with some proper colouring of a booster $B \in \cB'_Z$, then this fact will be evidenced by an edge of $\cT$ inside $C$.

Let us discuss what that should mean.  Suppose that $\varphi$ is a proper colouring of a booster $B \in \cB'_Z$.  We will show that, for every restricted colouring $C$, its determined colouring $\psi_C$ is consistent with $\varphi$ if and only if there is a colouring $\psi_{C,\varphi} \colon \bigcup I(B,Z) \to \br{r}$ that is consistent with $\varphi$ and further satisfies $\psi_{C,\varphi} \subseteq C$.  Since $\bigcup I(B,Z)$ contains at most $(s-1) \cdot L$ vertices, see~\ref{item:booster-cleanup-2}, these colourings $\psi_{C,\varphi}$ can serve as the kind of witnesses we are looking for.  As a result, one could actually stop here, and take the edges of the hypergraph to be all such sets $\big\{(z, \psi_{C,\varphi}(z)) \colon z \in \bigcup I(B,Z)\big\}$.  However, we will carry on and find witnesses that are minimal/irredundant.

To this end, suppose that an edge $A$ interacting between $B$ and $Z$ is not monochromatic.  Obviously, this means that two of its elements are coloured differently.  However, we will use the fact that $A$ intersects $B$ in exactly one vertex (see~\ref{item:booster-cleanup-2}) to conclude that either:
\begin{enumerate}
\item
  $\psi$ coloured two of the vertices of $A \setminus B$ in different colours, or
\item
  $\psi$ coloured some vertex of $A \setminus B$ in a colour different from $\varphi(A \cap B)$.
\end{enumerate}
One might be tempted to ignore the first case, which seems superfluous. Indeed, even if two vertices are coloured in different colours, one of these colours must be different from the one in $\varphi(A \cap B)$.  However, for a given $T \in I(B,Z)$, there can be many edges $A \in \cH$ with $A \setminus B = T$ and picking a colour that is different from the one in $\varphi(A \cap B)$ for \emph{every} such $A$ might not always be possible.

We will define the edges of $\cT$ using transversals of $I(B,Z)$.  Given a proper colouring $\varphi$ of $B$,  we will pick one of two obstructions for every $T \in I(B,Z)$: either two coloured vertices $(z,c)$, $(z',c')$, where $z,z' \in T$ and $c \neq c'$, or just one coloured vertex $(z,c)$, where $z \in T$ and $c$ different from the colour in $\varphi(A \cap B)$ for \emph{every} $A \in \cH$ such that $A \setminus B = T$.

We now define the hypergraph $\cT$ formally. For every $(s-1)$-element set $T \subseteq Z$, let
\[
  N_\cH(T;B) \coloneqq \{b \in B : T \cup \{b\} \in \cH\}
\]
and note that $I(B,Z)$ comprises precisely those sets $T \subseteq Z \setminus B$ for which $|N_{\cH}(T;B)| \ge 1$.  In particular, given $\psi$ and $\varphi$, there is an element $T \in I(B,Z)$ that $\psi$ colours monochromatically from the colours in $\varphi\big(N_\cH(T;B)\big)$. 

\begin{dfn}
  \label{dfn:cT}
  We let $\cT$ be the (multi)hypergraph with vertex set $Z \times \br{r}$ whose (multi)set of edges is defined as follows. For every:
  \begin{itemize}
  \item
    active booster $B \in \cBZ'$,
  \item
    proper colouring $\varphi \colon B \to \br{r}$ of $\cH[B]$,
  \item
    set $I' \subseteq I(B,Z)$ satisfying $I' \supseteq \big\{T \in I(B,Z) : \varphi\big(N_{\cH}(T;B)\big) = \br{r}\big\}$, and
  \item
    disjoint transversals\footnote{A \emph{transversal} of a family $\cF$ of sets is a subset $X$ of $\bigcup\cF$ such that $|X \cap F| = 1$ for every $F \in \cF$.} $\{v_T : T \in I(B,Z)\}$ of $I(B,Z)$ and $\{v_T' : T \in I'\}$ of $I'$, where $v_T \in T$ for every $T \in I(B,T)$ and $v_T' \in T$ for every $T \in I'$,
  \end{itemize}
  we add to $\cT$ all edges of the form
  \[
    \big\{(v_T, c_T) : T \in I(B,Z)\big\} \cup \big\{(v_T',c_T') : T \in I'\big\},
  \]
  where:
  \begin{itemize}
  \item
    for each $T \in I(B,Z) \setminus I'$, the colour $c_T$ is an arbitrary element of $\br{r} \setminus \varphi\big(N_{\cH}(T;B)\big)$;\footnote{The choice of $I'$ guarantees that the set $\br{r} \setminus \varphi\big(N_{\cH}(T;B)\big)$ is nonempty for every $T \in I(B,Z) \setminus I'$.}
  \item
    for each $T \in I'$, the colours $c_T$ and $c_T'$ are two arbitrary, distinct elements of $\br{r}$.
  \end{itemize}
\end{dfn}

An immediate consequence of the definition should be that every proper colouring $\psi$ of $Z$, when viewed as a restricted colouring, is an independent set of $\cT$.  This is because the determined colouring of $\psi$, which is just $\psi$ itself, cannot be extended to any $B \in \cBZ'$.

\begin{lemma}
  \label{lemma:psi-independent-cT}
  Every proper colouring $\psi \colon Z \to \br{r}$ of $\cH[Z]$, viewed as a subset of $Z \times \br{r}$, is an independent set in $\cT$.
\end{lemma}
\begin{proof}
  Suppose that some proper colouring $\psi$ of $\cH[Z]$ contains an edge of $\cT$. This means that there are an active booster $B \in \cBZ'$, a proper colouring $\varphi$ of $\cH[B]$, a set $I' \subseteq I(B,Z)$, and disjoint transversals $\{v_T : T \in I(B,Z)\}$ and $\{v_T' : T \in I'\}$ such that, for every $T \in I(B,Z)$, either $\psi(v_T) \notin \varphi\big(N_\cH(T;B)\big)$ or $\psi(v_T) \neq \psi(v_T')$. In particular, there is no $T \in I(B,Z)$ whose all elements receive the same colour from the set $\varphi\big(N_\cH(T;B)\big)$. On the other hand, since $B$ is disjoint from $Z$, see~\ref{item:booster-cleanup-1}, one may naturally define the colouring $\psi \cup \varphi \colon B \cup Z \to \br{r}$. Since $B$ is an active booster, this colouring $\psi \cup \varphi$ is not a proper colouring of $\cH[Z \cup B]$, which means that there is an $A \in \cH[Z \cup B]$ that $\psi \cup \varphi$ makes monochromatic. Since $\psi$ and $\varphi$ are proper colourings of $\cH[Z]$ and of $\cH[B]$, respectively, every such $A$ must have a nonempty intersection with both $Z$ and $B$. Our choice of $\cBZ'$, see~\ref{item:booster-cleanup-2}, implies that $|A \cap B| = 1$ and, consequently, that $A \setminus B \in I(B,Z)$. Moreover,
  \[
    (\psi\cup\varphi)(A) = \psi(A \setminus B) = \varphi(A \cap B),
  \]
  which means that all elements of $A \setminus B$ receive the same colour as the unique vertex in $A \cap B$, which belongs to $N_\cH(A \setminus B;B)$, a contradiction.
\end{proof}

More importantly, the definition implies that, for every restricted colouring $C$ that induces a small number of edges in $\cT$, its determined colouring $\psi_C$ forces a constant fraction of the vertices of $\cH$ to some colour.  Recall that we denote the set of vertices that $\psi_C$ forces to some colour by $F(\psi_C)$.

\begin{lemma}
  \label{lemma:containers-forced-vertices}
  If $C \subseteq Z \times \br{r}$ is a restricted colouring, then
  \[
    e(\cT[C]) \ge e(\cBZ') - |F(\psi_C)| \cdot \Delta_1(\cB).
  \]
\end{lemma}
\begin{proof}
  It suffices to find, for every active booster $B \in \cBZ'$ that does not contain any vertex of $F(\psi_C)$, an edge of $\cT[C]$ that corresponds to $B$, as in Definition~\ref{dfn:cT}. For every $b \in B$, let $L(b) \subseteq \br{r}$ be a list of some two colours that do not threaten $b$; there are at least two such colours since $b$ is not forced to any colour. Let $\varphi \colon B \to \br{r}$ be a proper colouring of $\cH[B]$ from these lists; such a colouring exists by our assumption that $\cH[B]$ is $2$-choosable from $\br{r}$, see~\ref{item:ass-choosability}.

  Fix an arbitrary $T \in I(B,Z)$ and let $N(T) = N_\cH(T;B)$. If $\psi_C$ colours all of $T$ with a colour $c_T$, then $c_T$ threatens every $b \in N(T)$ and hence $c_T \notin \bigcup_{b \in N(T)}L(b) \supseteq \varphi(N(T))$. In this case, we do not include $T$ in $I'$, we pick an arbitrary $v_T \in T$ and add $(v_T, c_T) \in C$ to the edge we construct. In the complementary case, either:
  \begin{enumerate}[label=(\roman*)]
  \item
    $\psi_C$ is not defined on all of $T$, which means that there is some vertex $v$ in $T$ and colours $c, c' \in \br{r}$ such that $(v, c), (v,c') \in C$ (this is because $\pi_1(C) = Z$ and hence the only reason why $\psi_C(v)$ is not defined is that $|\{v\} \times \br{r} \cap C| > 1$); or
  \item
    $\psi_C$ is defined on all of $T$ but there are distinct vertices $v, v' \in T$ such that $\psi_C(v) \neq \psi_C(v')$.
  \end{enumerate}
  Either way, since $|T| = s-1 \ge 2$, there must be distinct $v_T, v_T' \in T$ and $c_T, c_T' \in \br{r}$ such that $(v_T, c_T), (v_T', c_T') \in C$; we include $T$ in $I'$ and add $(v_T, c_T)$ and $(v_T', c_T')$ to the edge we construct.
\end{proof}

\subsection{The distribution of edges of $\cT$}
\label{sec:distr-edges-T}

The key step in the proof of Theorem~\ref{thm:many-rainbow-stars} will be to construct a small family of containers for proper $\br{r}$-colourings of $\cH[Z]$ using the information provided by the hypergraph $\cT$. In order to do this, we will first show that, for a typical choice of $Z \sim V_p$, we may find a subhypergraph $\cT' \subseteq \cT$ that contains almost as many edges as $\cT$ and satisfies
\[
  \Delta_1(\cT') = O\left(\frac{e(\cT')}{|Z|}\right) \qquad \text{and} \qquad \Delta_2(\cT') = o\left(\frac{e(\cT')}{|Z|}\right).
\]
We will then split $\cT'$ into uniform hypergraphs $\cT_1, \dotsc, \cT_{2L}$, where, for each $u \in \br{2L}$, the $u$-uniform hypergraph $\cT_u$ comprises all edges of $\cT'$ of cardinality $u$. Finally, we will apply the hypergraph container lemma, Theorem~\ref{thm:containers}, to each of the $\cT_u$ in order to build containers for independent sets of $\cT$: the container for an independent set $I$ will be the intersection of the $2L$ containers for $I$ in the hypergraphs $\cT_1, \dotsc, \cT_{2L}$.  The assumptions on $\Delta_1(\cT')$ and $\Delta_2(\cT')$ will guarantee that there are only $\exp(o(|Z|))$ different containers.  Lemma~\ref{lemma:containers-forced-vertices} provides a useful, structural description of all such containers.

In order to guarantee the existence of a $\cT' \subseteq \cT$ with the required properties, we shall estimate the $\ell^2$-norm of the sequences of vertex degrees and vertex-pair degrees of $\cT$ and invoke the following simple proposition.

\begin{prop}
  \label{prop:hypergraph-cleanup-l2-linf}
  Suppose that $\cG$ is a (multi)hypergraph with vertex set $V$. For all positive integers $t$ and $m$, there is a subhypergraph $\cG' \subseteq \cG$ with $e(\cG') \ge e(\cG) - m$ and
  \[
    \Delta_t(\cG') \le \frac{1}{m} \cdot \sum_{T \in \binom{V}{t}} \deg_\cG(T)^2.
  \]
\end{prop}
\begin{proof}
  We may assume that $m \le e(\cG)$; indeed, if $m > e(\cG)$, we may simply take $\cG'$ to be the empty hypergraph. We obtain $\cG'$ from $\cG$ by iteratively removing $m$ edges that contain some set $T \in \binom{V}{t}$ with largest degree. Since each of the $m$ removed edges contained a $t$-element subset of degree at least $\Delta_t(\cG')$, we have
  \[
    \sum_{T \in \binom{V}{t}} \deg_\cG(T)^2 \ge m \cdot \Delta_t(\cG'),
  \]
  as claimed.
\end{proof}

The following two crucial lemmas give upper bounds on the expectations of the $\ell^2$-norms of the sequences of vertex degrees and vertex-pair degrees of $\cT$.

\begin{lemma}
  \label{lemma:sum-deg-cT-squared}
  If $Z \sim V_p$, then
  \[
    \Ex\left[\sum_{v \in V(\cT)} \deg_{\cT}(v)^2\right] \le \frac{\Gamma \cdot e(\cB)^2}{\pH \cdot v(\cH)},
  \]
  where $\Gamma$ is a constant that depends only on $K$, $K_\cH$, $L$, $r$, and $s$.
\end{lemma}

\begin{lemma}
  \label{lemma:sum-codeg-cT-squared}
  If $Z \sim V_p$, then
  \[
    \Ex\left[\sum_{T \in \binom{V(\cT)}{2}} \deg_{\cT}(T)^2\right] \le \frac{\sigma \cdot  e(\cB)^2}{\pH \cdot v(\cH)}
  \]
  for some $\sigma = o(1)$.
\end{lemma}

\begin{proof}[{Proof of Lemma~\ref{lemma:sum-deg-cT-squared}}]
  For every $v \in V$, define
  \[
    \gamma(v) \coloneqq \left|\big\{A \in \cH :  \text{$v \in A$ and $|(A \setminus \{v\}) \cap Z| = s-2$}\big\}\right|.
  \]
  We claim that, for all $v \in Z$ and $i \in \br{r}$,
  \begin{equation}
    \label{eq:deg-cTL-gamma}
    \deg_{\cT}(v,i) \le \gamma(v) \cdot \Delta_1(\cB) \cdot (rs)^{2L+K},
  \end{equation}
  where $\deg_{\cT}$ counts edges with multiplicities. Indeed, if $v \in Z$, then $\gamma(v) \cdot \Delta_1(\cB)$ is an upper bound on the number of $B \in \cBZ'$ such that $v \in T$ for some $T \in I(B,Z)$. Every such $B$ gives rise to at most
  \[
    r^K \cdot \big((s-1)r\big)^L \cdot \sum_{i=0}^L \big((s-2)r\big)^{i} \le (rs)^{2L+K}
  \]
  edges of $\cT$, counting multiplicities; indeed, there are at most $r^K$ proper $\br{r}$-colourings of $B$, at most $(s-1)r$ choices for the pair $(v_T,c_T)$ for every $T \in I(B,Z)$, and at most $(s-2)r$ further choices for the pair $(v_T', c_T')$ for every $T \in I' \subseteq I(B,Z)$, see Definition~\ref{dfn:cT}. Moreover, all edges of $\cT$ that contain $(v,i)$ must come from one such set $B$.

  Since $V(\cT) = Z \times \br{r}$, inequality~\eqref{eq:deg-cTL-gamma} implies that
  \begin{equation}
    \label{eq:sum-deg-cT-gamma}
    \begin{split}
      \sum_{v \in V(\cT)} \deg_\cT(v)^2 & \le r \cdot \sum_{v \in Z} \gamma(v)^2 \cdot \Delta_1(\cB)^2 \cdot (rs)^{4L+2K} \\
      & \le \frac{K \cdot e(\cB)^2}{v(\cH)^2} \cdot (rs)^{4L+2K+1} \cdot \sum_{v \in Z} \gamma(v)^2.
    \end{split}
  \end{equation}
  In the remainder of the proof, we will bound the expected value of the sum in the right-hand side of~\eqref{eq:sum-deg-cT-gamma}.

  Fix an arbitrary $v \in V$ and observe that
  \[
    \Ex[\gamma(v)^2] = \sum_{\substack{A, A' \in \cH \\ v \in A \cap A'}} \underbrace{\Pr\big(|A \cap (Z \setminus \{v\})| = s-2 \text{ and } |A' \cap (Z \setminus \{v\})| = s-2\big)}_{P(A,A')};
  \]
  note that $P(A,A) \le (s-1) \cdot p^{s-2}$ and that, if $A \neq A'$, then
  \[
    P(A,A') \le (s-1)^2 \cdot p^{|A \cup A'| - 3} = (s-1)^2 \cdot p^{2s-3-|A \cap A'|}.
  \]
  Consequently,
  \[
    \Ex[\gamma(v)^2] = \deg_{\cH}(v) \cdot (s-1) \cdot p^{s-2} \cdot \left(1 + \sum_{t=1}^{s-1} \binom{s-2}{t-1} \cdot (s-1)  \cdot \Delta_t(\cH)\cdot p^{s-1-t} \right).
  \]
  Our assumption that $\cH$ is non-clustered implies that, for every $t \in \br{s-1}$,
  \[
    \Delta_t(\cH) \cdot p^{s-1-t} \le K_\cH^{s-1} \cdot \Delta_t(\cH) \cdot \pH^{s-1-t} \le K_\cH^s \cdot \frac{e(\cH)}{v(\cH)} \cdot \pH^{s-2} = \frac{K_\cH^s}{\pH}.
  \]
  and thus
  \[
    \Ex[\gamma(v)^2] \le \Gamma \cdot \deg_\cH(v) \cdot \pH^{s-3}
  \]
  for some constant $\Gamma$ that depends only on $K_\cH$ and $s$.  Since the variable $\gamma(v)$ is independent of the event $v \in Z$, we have
  \[
    \begin{split}
      \Ex\left[\sum_{v \in Z} \gamma(v)^2\right] & = \sum_{v \in V} p \cdot \Ex[\gamma(v)^2] \le \Gamma \cdot K_\cH \cdot \pH^{s-2} \cdot \sum_{v \in V} \deg_\cH(v) \\
      & = \Gamma \cdot K_\cH \cdot \pH^{s-2} \cdot s \cdot e(\cH) = \frac{\Gamma \cdot K_\cH \cdot s \cdot v(\cH)}{\pH}.
    \end{split}
  \]
  Taking expectations of both sides of~\eqref{eq:sum-deg-cT-gamma} and substituting the above inequality yields the assertion of the lemma.
\end{proof}

\begin{proof}[{Proof of Lemma~\ref{lemma:sum-codeg-cT-squared}}]
  For every pair of distinct $v, w \in V$, let $\gamma_2(v, w)$ be the number of triples $(A_v, A_w, B) \in \cH^2 \times \cB$ satisfying the following:
\begin{enumerate}[label=(\roman*)]
\item
  \label{item:gamma2-1}
  $v \in A_v \setminus B$ and $w \in A_w \setminus B$,
\item
  \label{item:gamma2-2}
  $|A_v \cap B| = |A_w \cap B| = 1$,
\item
  \label{item:gamma2-3}
  $A_v \cap A_w \subseteq B$ or $A_v \setminus B = A_w \setminus B$,
\item
  \label{item:gamma2-4}
  $(A_v \cup A_w) \setminus (B \cup \{v, w\}) \subseteq Z$.
\end{enumerate}
We claim that, for every pair of distinct $v, w \in Z$ and all $i, j \in \br{r}$,
\begin{equation}
  \label{eq:deg2-cTL-gamma2}
  \deg_{\cT}\{(v,i), (w,j)\} \le \gamma_2(v,w) \cdot (rs)^{2L+K}.
\end{equation}
Indeed, if  $v,w \in Z$, then $\gamma_2(v,w)$ is an upper bound on the number of $B \in \cBZ'$ such that $v \in T_v$ and $w \in T_w$ for some $T_v, T_w \in I(B,Z)$. Every such $B$ gives rise to at most $(rs)^{2L+K}$ edges of $\cT$, counting multiplicities, and all edges of $\cT$ that contain both $(v,i)$ and $(w,j)$ must come from one such set $B$.

Since $V(\cT) = Z \times \br{r}$ and no edge of $\cT$ contains a pair of vertices $\{(v,i), (v,j)\}$ with $i \neq j$, inequality~\eqref{eq:deg2-cTL-gamma2} implies that
\begin{equation}
  \label{eq:sum-codeg-cT-gamma2}
  \sum_{T \in \binom{V(\cT)}{2}} \deg_\cT(T)^2 \le r^2 \cdot \sum_{v,w \in Z} \gamma_2(v,w)^2 \cdot (rs)^{4L+2K}.
\end{equation}
In the remainder of the proof, we will bound the expected value of the sum in the right-hand side of~\eqref{eq:sum-codeg-cT-gamma2}.

\begin{claim}
  \label{claim:Ex-gamma2-squared}
  For every pair of distinct vertices $v, w \in V$,
  \[
    \Ex[\gamma_2(v, w)^2] \ll \frac{e(\cB)}{\pH \cdot v(\cH)} \cdot \Ex[\gamma_2(v,w)].
  \]
\end{claim}
\begin{proof}
  Let $\cX$ denote the family of all triples $(A_v, A_w, B)$ that satisfy~\ref{item:gamma2-1}--\ref{item:gamma2-3} above and observe that
  \[
    \gamma_2(v,w) = \sum_{(A_v, A_w, B) \in \cX} \1_{(A_v \cup A_w) \setminus (B \cup \{v,w\}) \subseteq Z},
  \]
  see~\ref{item:gamma2-4}. Therefore, it suffices to show that, for each $(A_v, A_w, B) \in \cX$,
  \begin{equation}
    \label{eq:gamma2-conditional}
    \Ex\big[\gamma_2(v,w) \mid (A_v \cup A_w) \setminus (B \cup \{v,w\}) \subseteq Z\big] \ll \frac{\Delta_1(\cB)}{\pH}.
  \end{equation}
  We partition the family $\cX$ according to the intersection pattern with our chosen triple $(A_v, A_w, B)$. First, for every $t \in \{0, \dotsc, s-3\}$, denote by $\cX_t$ the family of all triples $(A_v', A_w', B')$ such that $A_v' \setminus B = A_w' \setminus B$ and $A_v'$ intersects $(A_v \cup A_w) \setminus (B \cup \{v,w\})$ in exactly $t$ elements. Second, for every $(t_v, t_w) \in \{0, \dotsc, s-2\}^2$, denote by $\cX_{t_v, t_w}$ the family of all triples $(A_v', A_w', B')$ such that $A_v' \cap A_w' \subseteq B$ and, for $u \in \{v,w\}$, the set $A_u'$ intersects $(A_v \cup A_w) \setminus (B \cup \{v,w\})$ in exactly $t_u$ elements. Observe that
  \[
    \cX = \bigcup_{t=0}^{s-3} \cX_t \cup\bigcup_{t_v, t_w = 0}^{s-2} \cX_{t_v,t_w}.
  \]
  Further, note that, for every $t \in \{0, \dotsc, s-3\}$,
  \[
    |\cX_t| \le \Delta_{t+2}(\cH) \cdot \Delta_1(\cB) \cdot \max_{B \in \cB} |B| \ll \pH^{t+1} \cdot \frac{e(\cH)}{v(\cH)} \cdot \Delta_1(\cB) = \pH^{t+2-s} \cdot \Delta_1(\cB).
  \]
  Crucially, we claim that, for every $t_v, t_w \in \{0, \dotsc, s-2\}$,
  \[
    \begin{split}
      |\cX_{t_v, t_w}| & \le \min\big\{ \Delta_{t_v+1}(\cH) \cdot \Delta_{t_w+2}(\cH), \Delta_{t_w+1}(\cH) \cdot \Delta_{t_v+2}(\cH) \big\} \cdot \Delta_1(\cB) \\
      & \ll \pH^{t_v+t_w+1} \cdot \left(\frac{e(\cH)}{v(\cH)}\right)^2 \cdot \Delta_1(\cB) = \pH^{t_v+t_w-2s+3} \cdot \Delta_1(\cB).
    \end{split}
  \]
  Indeed, $\Delta_t(\cH) \le K \pH^{t-1} \cdot e(\cH) / v(\cH)$ for all $t \in \br{s}$ and $\Delta_t(\cH) \ll \pH^{t-1} \cdot e(\cH) / v(\cH)$ when $2 \le t \le s-1$ and one of $\max\{t_v, t_w\} + 1$ or $\min\{t_v, t_2\} + 2$ must belong to $\{2, \dotsc, s-1\}$. Since $p = O(\pH)$, estimate~\eqref{eq:gamma2-conditional} follows after noting that, first, for every $t \in \{0, \dotsc, s-3\}$ and every $(A_v', A_w', B') \in \cX_t$,
  \[
    \Pr\big((A_v' \cup A_w') \setminus (B' \cup \{v,w\}) \subseteq Z \mid (A_v \cup A_w) \setminus (B \cup \{v,w\}) \subseteq Z\big) = p^{s-3-t},
  \]
  and, second, for all $(t_v, t_w) \in \{0, \dotsc, s-2\}^2$ and every $(A_v', A_w', B') \in \cX_{t_v, t_w}$,
  \[
    \Pr\big((A_v' \cup A_w') \setminus (B' \cup \{v,w\}) \subseteq Z \mid (A_v \cup A_w) \setminus (B \cup \{v,w\}) \subseteq Z\big) = p^{2s-4-t_v-t_w}.
  \]
  Since $p \le K_\cH \cdot \pH$ and $\Delta_1(\cB) \le K \cdot e(\cB) / v(\cH)$, this concludes the proof of the claim.
\end{proof}

\begin{claim}
  \label{claim:sum-Ex-gamma2}
  We have
  \[
    \Ex\left[\sum_{v,w \in Z} \gamma_2(v,w) \right] = O\big(e(\cB)\big).
  \]
\end{claim}
\begin{proof}
  Let $\cX_1$ denote the family of all pairs $(A,B) \in \cH \times \cB$ such that $|A \cap B| = 1$ and let $\cX_2$ denote the family of all triples $(A', A'', B') \in \cH^2 \times \cB$ such that $|A' \cap B'| = |A'' \cap B'| = 1$ and $A' \cap A'' \subseteq B'$. Note that
  \begin{multline}
    \label{eq:sum-gamma2}
    \sum_{v,w \in Z} \gamma_2(v,w) = \binom{s-1}{2} \cdot \left|\big\{(A, B) \in \cX_1 : A \setminus B \subseteq Z \big\}\right| \\
    + (s-1)^2 \cdot \left|\big\{(A', A'', B') \in \cX : (A' \cup A'') \setminus B' \subseteq Z \big\}\right|.
  \end{multline}
  The claim now follows as
  \begin{align*}
    |\cX_1| & \le \sum_{B \in \cB} |B| \cdot \Delta_1(\cH) = O\left(e(\cB) \cdot \frac{e(\cH)}{v(\cH)}\right), \\
    |\cX_2| & \le \sum_{B \in \cB} |B|^2 \cdot \Delta_1(\cH)^2 = O\left( e(\cB) \cdot \frac{e(\cH)^2}{v(\cH)^2}\right),
  \end{align*}
  and, for all $(A,B) \in \cX_1$ and $(A', A'', B') \in \cX_2$,
  \begin{align*}
    \Pr\big(A \setminus B \subseteq Z\big) & = p^{s-1} = O\big(\pH^{s-1}\big), \\
    \Pr\big((A' \cup A'') \setminus B' \subseteq Z\big) & = p^{2s-2} = O\big(\pH^{2s-2}\big).
  \end{align*}
  Indeed, since $\pH^{s-1} \cdot e(\cH) = v(\cH)$, taking expectations of both sides of~\eqref{eq:sum-gamma2} and substituting the above estimates yields the claimed estimate.
\end{proof}

Since the variable $\gamma_2(v,w)$ is independent of the event $\{v,w\} \subseteq Z$, we have
  \[
    \begin{split}
      \Ex\left[\sum_{v,w \in Z} \gamma_2(v,w)^2\right] & = p^2 \cdot \sum_{v,w \in V} \Ex[\gamma_2(v,w)^2] \ll p^2 \cdot \frac{e(\cB)}{\pH \cdot v(\cH)}\cdot \sum_{v,w \in V} \Ex[\gamma_2(v,w)] \\
      & = \frac{e(\cB)}{\pH \cdot v(\cH)}  \cdot \Ex\left[\sum_{v,w \in Z} \gamma_2(v,w)\right] = O\left(\frac{e(\cB)^2}{\pH \cdot v(\cH)}\right),
    \end{split}
  \]
  where the first inequality follows from Claim~\ref{claim:Ex-gamma2-squared} and the second inequality follows from Claim~\ref{claim:sum-Ex-gamma2}. The assertion of the lemma follows by taking expectations of both sides of~\eqref{eq:sum-codeg-cT-gamma2} and substituting the above inequality.
\end{proof}

\subsection{Containers for colourings}
\label{sec:cont-colo}

Let $\sigma$ be the sequence from the statement of Lemma~\ref{lemma:sum-codeg-cT-squared} and let $\tau = \sigma^{1/(2L+1)} = o(1)$. Let $\cZ''$ be the collection of all $Z \in \cZ'$ such that
\[
  \sum_{v \in V(\cT)} \deg_\cT(v)^2 \le \Gamma_{\cH} \cdot \frac{e(\cB)^2}{|Z|}
  \qquad
  \text{and}
  \qquad
  \sum_{T \in \binom{V(\cT)}{2}} \deg_{\cT}(T)^2 \le \tau^{2L} \cdot \frac{e(\cB)^2}{|Z|}.
\]

Since $\cZ'' \subseteq \cZ$, the set $Z$ satisfies~\ref{item:many-rainbow-stars-eps-p} in the statement of Theorem~\ref{thm:many-rainbow-stars}.

\begin{lemma}
  \label{lemma:many-rainbow-stars-containers}
  Every $Z \in \cZ''$ satisfies~\ref{item:many-rainbow-stars-containers} in the statement of Theorem~\ref{thm:many-rainbow-stars}.
\end{lemma}

Before we prove the lemma, let us point out that it implies the assertion of the theorem.  Indeed, it follows from Lemmas~\ref{lemma:sum-deg-cT-squared} and~\ref{lemma:sum-codeg-cT-squared}, Markov's inequality, and standard estimates for the tails of the binomial distribution that, when $Z \sim V_p$,
\[
  \begin{split}
    \Pr(Z \in \cZ'') & \ge \Pr(Z \in \cZ') - \Pr\big(|Z| \ge 2\pH v(\cH)\big)-2\Gamma/\Gamma_\cH - 2\sigma \cdot \tau^{-2L} \\
    & \ge \eps/2 - o(1) - 2\Gamma/\Gamma_\cH - 2\tau \ge \eps/4.
  \end{split}
\]

\begin{proof}[{Proof of Lemma~\ref{lemma:many-rainbow-stars-containers}}]
  Suppose that $Z \in \cZ''$. We will construct the desired family of partial colourings of $Z$ by applying the container lemma, Theorem~\ref{thm:containers}, to a collection of uniform subhypergraphs of the hypergraph $\cT$.
  
  To this end, note first that Proposition~\ref{prop:hypergraph-cleanup-l2-linf}, invoked twice, implies that $\cT$ contains a subhypergraph $\cT'$ with at least $e(\cT) - 2 \cdot (\eps/16) \cdot e(\cB)$ edges that satisfies
  \[
    \Delta_1(\cT') \le \frac{16 \Gamma_\cH}{\eps} \cdot \frac{e(\cB)}{|Z|}
    \qquad
    \text{and}
    \qquad
    \Delta_2(\cT') \le \frac{16 \tau^{2L}}{\eps} \cdot \frac{e(\cB)}{|Z|} \le \tau^{2L-1} \cdot \frac{e(\cB)}{|Z|}.
  \]
  Fix one such hypergraph $\cT'$ and note that
  \begin{equation}
    \label{eq:ecT'-upper}
    e(\cT') \le r \cdot |Z| \cdot \Delta_1(\cT') \le \frac{16 \Gamma_\cH r}{\eps} \cdot e(\cB).
  \end{equation}
  
  Let $\cT_1, \dotsc, \cT_{2L}$ be the subhypergraphs of $\cT'$ that comprise all edges of $\cT'$ of cardinalities $1, \dotsc, 2L$, respectively, and note that $\cT' = \cT_1 \cup \dotsb \cup \cT_{2L}$, since every edge of $\cT$ has cardinality at most
  \[
    2 \cdot \max_{B \in \cBZ'} |I(B,Z)| \le 2L,
  \]
  see Definition~\ref{dfn:cT}. Define
  \[
    U \coloneqq \big\{u \in \br{2L} : e(\cT_u) \ge \eps/(16L) \cdot e(\cB)\big\}
  \]
  and fix an arbitrary $u \in U$. Since
  \[
    \Delta_1(\cT_u) \le \Delta_1(\cT') \le \frac{16\Gamma_\cH}{\eps} \cdot \frac{e(\cB)}{|Z|} \le \frac{256\Gamma_\cH L r}{\eps^2} \cdot \frac{e(\cT_u)}{v(\cT_u)}
  \]
  and, for every $\ell \in \{2, \dotsc, u\}$,
  \[
    \Delta_\ell(\cT_u) \le \Delta_2(\cT') \le \tau^{2L-1} \cdot \frac{e(\cB)}{|Z|} \le \frac{16Lr}{\eps} \cdot \tau^{\ell-1} \cdot \frac{e(\cT_u)}{v(\cT_u)},
  \]
  we may apply Theorem~\ref{thm:containers}, with $k_{\ref{thm:containers}} \coloneqq u$,
  \[
    K_{\ref{thm:containers}} \coloneqq \max\left\{\frac{256\Gamma_\cH L r}{\eps^2}, \frac{16Lr}{\eps}\right\},
    \quad
    \text{and}
    \quad
    \eps_{\ref{thm:containers}} \coloneqq \frac{\eps^2}{256\Gamma_\cH Lr} \by{\eqref{eq:ecT'-upper}}{\le} \frac{\eps/(16L) \cdot e(\cB)}{e(\cT')}
  \]
  to get an integer $t$ (that depends only on $\eps$, $\Gamma_\cH$, $r$, and $L$) and a collection $\cC_u$ of at most
  \[
    \left(\sum_{i=0}^{\tau r |Z|} \binom{r|Z|}{i}\right)^t \le \left(\frac{er|Z|}{\tau r |Z|}\right)^{t \tau r |Z|} = \exp\big(o(|Z|)\big)
  \]
  subsets of $Z \times \br{r}$ with the following properties:
  \begin{enumerate}[label=(\roman*)]
  \item
    \label{item:cT-containers-i}
    Every proper colouring $\psi \colon Z \to \br{r}$ of $\cH[Z]$, viewed as a subset of $Z \times \br{r}$, is contained in a member of $\cC_u$.
  \item
    \label{item:cT-containers-ii}
    Every member of $\cC_u$ induces fewer than $\eps/(16L) \cdot e(\cB)$ edges in $\cT_u$.
  \end{enumerate}

  Finally, we let $\Psi$ to be the collection of all partial colourings $\psi_C$, where $C$ is a set of the form
  \[
    C = \bigcap_{u \in U} C_u,
  \]
  where, for each $u \in U$, the set $C_u$ is a member of $\cC_u$, such that $\pi_1(C) = Z$; note that
  \[
    |\Psi| \le \prod_{u \in U} |\cC_u| = \exp\big(o(|Z|)\big).
  \]
  It follows from~\ref{item:cT-containers-i} that, for each $u \in U$, every proper colouring $\psi \colon Z \to \br{r}$ of $\cH[Z]$ is contained in some set $C_u \in \cC_u$ and, consequently, in some set $C$ of the above form; in particular $\psi$ extends the partial colouring $\psi_C \in \Psi$. We now show that, for every partial colouring $\psi_C \in \Psi$, we have $|F(\psi_C)| = \Omega\big(v(\cH)\big)$; this will conclude the proof of the lemma, as every vertex in $F(\psi_C)$ is the centre of a rainbow star.

  To this end, choose an arbitrary set $C$ of the above form. Since $e(\cT_u[C]) \le e(\cT_u[C_u]) \le \eps/(16L) \cdot e(\cB)$ when $u \in U$, see~\ref{item:cT-containers-ii} above, and  $e(\cT_u[C]) \le e(\cT_u) < \eps/(16L) \cdot e(\cB)$ when $u \notin U$, we have
  \[
    e(\cT'[C]) = \sum_{u = 1}^{2L} e(\cT_u[C]) \le 2L \cdot \eps/(16L) \cdot e(\cB) = (\eps/8) \cdot e(\cB).
  \]
  Consequently,
  \[
    e(\cT[C]) \le e(\cT'[C]) + e(\cT) - e(\cT') \le (\eps/4) \cdot e(\cB)
  \]
  and thus Lemma~\ref{lemma:containers-forced-vertices} implies that
  \[
    |F(\psi_C)| \ge \frac{e(\cBZ') - e(\cT[C])}{\Delta_1(\cB)} \ge \frac{\eps \cdot e(\cB)}{4 \cdot \Delta_1(\cB)} \ge \frac{\eps}{4K} \cdot v(\cH),
  \]
  since $Z \in \cZ'$ implies that $e(\cBZ') \ge (\eps/2) \cdot e(\cB)$.
\end{proof}

\section{Rainbow stars and constellations}
\label{sec:rainb-stars-const}

In this section, we prove Theorem~\ref{thm:rainbow-stars-constellations}, which encapsulates Step~\ref{step:sparse-rsc} from the proof outline.  We restate the theorem here for the reader's convenience.

\rainbowstarsconstellations*

The heart of the proof will be transferring the rainbow star-constellation property from $\cH$ to $\cH_p$.  We will make use of the container lemma (Theorem~\ref{thm:containers}) to prove the following sparse random analogue of the rainbow star-constellation property. 

\begin{thm}
  \label{thm:sparse-rainbow-star-constellation}
  Suppose that $r \ge 2$, $s \ge 3$, and $\cH$ is a sequence of $s$-uniform hypergraphs that satisfies assumptions~\ref{item:ass-non-clusteredness} and~\ref{item:ass-star-constellation}.  For every positive $\betas$, there exists a positive $\betac$ such that the following holds: Suppose that $p = \Omega(\pH)$ and let $Z \sim V(\cH)_p$. With probability $1-o(1)$, every partial $\br{r}$-colouring of $Z$ that admits at least $\betas \cdot (p/\pH)^{(r-1)(s-1)} v(\cH)$ rainbow stars must also admit at least $\betac \cdot (p/\pH)^{(r-1)(s-1)s} e(\cH)$ rainbow constellations.
\end{thm}

The above theorem will be combined with the following technical lemma, which implies that, in a typical set $Z \sim V(\cH)_p$, a family of $\Omega\big((p/\pH)^{(r-1)(s-1)s} e(\cH)\big)$ rainbow constellations has to determine $\Omega(e(\cH))$ different base edges.  We will say that a set $Z \subseteq V(\cH)$ \emph{admits} a (non-coloured) star $\{A_1, \dotsc, A_k\}$ with centre $v$ if $(A_1 \cup \dotsb \cup A_k) \setminus \{v\} \subseteq Z$; a set $Z$ admits a constellation if it admits each of the $s$ stars that comprise it.  Given a set $Z \subseteq V(\cH)$ and an edge $A \in \cH$, we denote by $\con(A,Z)$ the number of constellations with base edge $A$ that are admitted by $Z$.

\begin{lemma}
  \label{lemma:con-ell-2}
  Suppose that $r \ge 2$, $s \ge 3$, and $\cH$ is a sequence of $s$-uniform hypergraphs that satisfies assumption~\ref{item:ass-non-clusteredness}.  For every $c > 0$, there exists a constant $\Gamma$ such that the following holds:  If $p \ge c \cdot \pH$ and $Z \sim V(\cH)_p$, then, for every $A \in \cH$,
  \[
    \Ex\left[\sum_{A \in \cH} \con(A,Z)^2 \right] \le \Gamma \cdot \left(\frac{p}{\pH}\right)^{2(r-1)(s-1)s} \cdot e(\cH).
  \]
\end{lemma}

We postpone the proofs of Theorem~\ref{thm:sparse-rainbow-star-constellation} and Lemma~\ref{lemma:con-ell-2} to later subsections and first show how they imply Theorem~\ref{thm:rainbow-stars-constellations}.

\begin{proof}[Proof of Theorem~\ref{thm:rainbow-stars-constellations}]
  Suppose that $s \ge 3$, $r \ge 2$, and $\eps > 0$ and let $\cH$ be a sequence of non-clustered, $s$-uniform hypergraphs that satisfies the star-constellation property with $r$ colours.  Assume that $p = \Theta(\pH)$ and $Z \sim V(\cH)_p$.  It follows from Lemma~\ref{lemma:con-ell-2} and Markov's inequality that, for some constant $C$,
  \[
    \sum_{A \in \cH} \con(A,Z)^2 \le C e(\cH)
  \]
  with probability at least $1 - \eps/2$. By Theorem~\ref{thm:sparse-rainbow-star-constellation}, as $p/\pH = \Theta(1)$, with probability at least $1-\eps/2$, every partial $\br{r}$-colouring of $Z$ that admits $\Omega\big(v(\cH)\big)$ rainbow stars must also admit $\Omega\big(e(\cH)\big)$ rainbow constellations.  Assume that the random set $Z$ has both these properties, which happens with probability at least $1-\eps$.  Fix any colouring $\psi$ with $\Omega\big(v(\cH)\big)$ rainbow stars and let $\cH_\psi$ comprise all edges of $\cH$ that are the base of at least one rainbow constellation.  On the one hand, since every $A \in \cH_\psi$ is the base of at most $\con(A,Z)$ rainbow constellations, our assumption on $Z$ implies that
  \[
    \sum_{A \in \cH_\psi} \con(A,Z) \ge \Omega\big(e(\cH)\big).
  \]
  On the other hand, the Cauchy--Schwarz inequality yields
  \[
    \left(\sum_{A \in \cH_\psi} \con(A,Z)\right)^2 \le e(\cH_\psi) \cdot \sum_{A \in \cH_\psi} \con(A,Z)^2 \le e(\cH_\psi) \cdot  \sum_{A \in \cH} \con(A,Z)^2.
  \]
  Combining the three displayed inequalities gives $e(\cH_\psi) = \Omega\big(e(\cH)\big)$, as desired.
\end{proof}

\subsection{The hypergraphs of stars and constellations}
\label{sec:hypergr-stars-const}

Assume from now on that $s \ge 3$ and $r \ge 2$ are integers and that $\cH$ is a sequence of $s$-uniform hypergraphs that satisfies assumptions~\ref{item:ass-non-clusteredness} and~\ref{item:ass-star-constellation}.  In order to use the container lemma in the context of Theorem~\ref{thm:sparse-rainbow-star-constellation}, let us define two (multi)hypergraphs with vertex set $V(\cH) \times \br{r}$:
\begin{itemize}
\item
  the $(r-1)(s-1)$-uniform hypergraph $\RS$ of rainbow stars,
\item
  the $(r-1)(s-1)s$-uniform hypergraph $\RC$ of rainbow constellations.
\end{itemize}
The edges of $\RS$ are all sets of the form
\[
  E_i\left((A_j)_{j \in \br{r} \setminus \{i\}}\right) = \bigcup_{j \in \br{r} \setminus \{i\}} \big(A_j \setminus \{v\}\big) \times \{j\},
\]
where $i \in \br{r}$ and $\{A_j : j \in \br{r} \setminus \{i\}\}$ is an $(r-1)$-star with centre $v$.  We add such sets to the multihypergraph $\RS$ with their proper multiplicities, that is, $E_i(A_1, \dotsc, A_{r-1})$ is added to $\RS$ once for every ordering of the edges of each $(r-1)$-star $\{A_1, \dotsc, A_{r-1}\}$.\footnote{It is possible that $E_i(A_1, \dotsc, A_{r-1}) = E_i(A_1', \dotsc, A_{r-1}')$ when $\{A_1, \dotsc, A_{r-1}\}$ and $\{A_1', \dotsc, A_{r-1}'\}$ are two different stars that differ only in their centres.}  The edges of $\RC$ are all sets of the form
\[
  E_i(\bA_1) \cup \dotsb \cup E_i(\bA_s),
\]
where $\bA_1, \dotsc, \bA_s$ are arbitrary orderings of the edge sets of some $s$ stars that form a~constellation. As in the case of $\RS$, every edge of $\RC$ appears with its proper multiplicity, that is, once for every choice of a base edge formed by the centres of the stars in the~constellation.

Our next two technical lemmas provide lower bounds on the numbers of edges in the hypergraphs $\RS$ and $\RC$ and upper bounds on the sequence of maximum degrees of the latter hypergraph.  These estimates will allow us to apply the container lemma (Theorem~\ref{thm:containers}) to the $(r-1)(s-1)s$-uniform hypergraph $\RC$ and construct a family of only $\exp\big(o(p \cdot v(\cH))\big)$ containers that cover the family of all partial colourings of $V(\cH)$ that admit many fewer than $p^{(r-1)(s-1)s} \cdot e(\RC)$
% = p^{(r-1)(s-1)s} \cdot \frac{e(\cH)^{(r-1)s+1}}{v(\cH)^{(r-1)s}} =  \left(\frac{p}{\pH}\right)^{(r-1)(s-1)s} \cdot e(\cH)
rainbow constellations.

\begin{lemma}
  \label{lemma:RS-RC-edges}
  There is a positive constant $c$ that depends only on $r$ and $s$ such that
  \[
    e(\RS) \ge  c \cdot \frac{e(\cH)^{r-1}}{v(\cH)^{r-2}}
    \qquad \text{and} \qquad
    e(\RC) \ge c \cdot \frac{e(\cH)^{(r-1)s+1}}{v(\cH)^{(r-1)s}}.
  \]
\end{lemma}

\begin{lemma}
  \label{lemma:RC-degrees}
  There are constants $\Gamma$ and $\Gamma'$ that depend only on $r$ and the sequence $\cH$ and such that
  \[
    \Delta_1(\RC) \le \Gamma' \cdot \left(\frac{e(\cH)}{v(\cH)}\right)^{(r-1)s+1} \le \Gamma \cdot \frac{e(\RC)}{v(\RC)}.
  \]
  Moreover, for every $t \ge 2$,
  \[
    \Delta_t(\RC) \ll \pH^{t-1} \cdot \frac{e(\RC)}{v(\RC)}.
  \]
\end{lemma}

The next lemma translates the rainbow star-constellation property into the language of $\RS$ and $\RC$.

\begin{lemma}
  \label{lemma:RS-RC-property-edges}
  For every positive constant $\beta$, there exists a positive constant $\eps$ that depends only on $\beta$, $r$, and the sequence $\cH$ and such that the following holds for every $C \subseteq V(\cH) \times \br{r}$: If $e\big(\RS[C]\big) \ge \beta \cdot e(\RS)$, then $e\big(\RC[C]\big) \ge \eps \cdot e(\RC)$.
\end{lemma}

Furthermore, it will be convenient to define the $(r-1)(s-1)$-uniform (multi)hypergraph $\MS$ of uncoloured $(r-1)$-stars whose vertex set is $V(\cH)$. The edges of $\MS$ are all sets of the form
\[
  (A_1 \cup \dotsb \cup A_{r-1}) \setminus \{v\},
\]
where $A_1, \dotsc, A_{r-1}$ form a star in $\cH$ whose centre vertex is $v$.  Observe that $\MS$ can be thought of as the image of $\RS$ via the projection $\pi_1 \colon V(\cH) \times \br{r} \to V(\cH)$ of its vertex set onto the first coordinate.  More precisely, since every star admits $r!$ colourings that make it rainbow, we have that $\pi_1(\RS) = r! \cdot \MS$ as multihypergraphs.  This allows us to bound degrees of subhypergraphs of $\RS$ induced by various $\br{r}$-colourings of $Z$ from above using respective degrees of $\MS[Z]$.  The advantage of such an approach is that the latter degrees depend only on $Z$ rather than on the particular colouring of $Z$.  This motivates the two final lemmas of this subsection.

\begin{lemma}
  \label{lemma:MS}
  There is a constant $\Gamma$ that depend only on $r$ and the sequence $\cH$ such that
  \[
    \Delta_1(\MS) \le \Gamma \cdot \frac{e(\MS)}{v(\MS)}
    \qquad \text{and} \qquad
    e(\MS) \le \Gamma \cdot \frac{e(\cH)^{r-1}}{v(\cH)^{r-2}}.
  \]
  Moreover, for every $t \ge 2$,
  \[
    \Delta_t(\MS) \ll \pH^{t-1} \cdot \frac{e(\MS)}{v(\MS)}.
  \]
\end{lemma}

Recall the definition of pseudo-variance (Definition~\ref{dfn:pseudo-variance}) given in Section~\ref{sec:janson-inequality}.  In our final lemma, we abuse the notation somewhat and write $\pVar(\MS[Z])$ to denote the pseudo-variance of the sequence of events $(\{A \subseteq Z\})_{A \in \MS}$.

\begin{lemma}
  \label{lemma:MS-pseudo-variance}
  For every positive constant $c$, there is a constant $\Gamma$ that depends only on $c$, $r$, and the sequence $\cH$ such that the following holds:  If $Z \sim V(\cH)_p$ for some $p \ge c \cdot \pH$, then
  \[
    \pVar\big(\MS[Z]\big) \le \Gamma \cdot \frac{\Ex[e(\MS[Z])]^2}{pv(\MS)}.
  \]
\end{lemma}

\subsection{Proofs of Lemmas~\ref{lemma:con-ell-2}--\ref{lemma:MS-pseudo-variance}}
\label{sec:proofs-technical-lemmas}

In this subsection, we prove Lemmas~\ref{lemma:con-ell-2}--\ref{lemma:MS-pseudo-variance}.  All proofs are straightforward, albeit somewhat technical.  Since the proofs of Lemmas~\ref{lemma:RC-degrees} and~\ref{lemma:MS} are very similar, but the latter is simpler, we present them in reverse order.  Throughout this section, we assume that $r \ge 2$ and $s \ge 3$ are integers and that $\cH$ is a sequence of $s$-uniform hypergraphs that satisfies assumptions~\ref{item:ass-non-clusteredness} and~\ref{item:ass-star-constellation}.  In particular, there is a constant $K$ such that, for all $t \ge 1$,
\begin{equation}
  \label{eq:Delta-t-cH}
  \Delta_t(\cH) \le K \cdot \pH^{t-1} \cdot \frac{e(\cH)}{v(\cH)}.
\end{equation}

\begin{proof}[Proof of Lemma~\ref{lemma:con-ell-2}]
  Fix an edge $A \in \cH$ and let $\cX_A$ be the family of all constellations with base edge $A$.  Denoting by $\supp(\cC)$ the union of all $(r-1)s$ edges forming a constellation $\cC$, we have
  \[
    \begin{split}
      \Ex\left[\con(A,Z)^2\right] & = \sum_{\cC, \cC' \in \cX_A} p^{|(\supp(\cC) \cup \supp(\cC')) \setminus A|} \\
      & = \sum_{\cC \in \cX_A} p^{|\supp(\cC) \setminus A|} \cdot \sum_{\cC' \in \cX_A} p^{|\supp(\cC') \setminus \supp(\cC)|}.
    \end{split}
  \]
  For every $\cC \in \cX_A$ and $t \in \{0, \dotsc, (r-1)(s-1)s\}$, let $\cX_{\cC,t}$ denote the family of all $\cC' \in \cX_A$ satisfying $|\supp(\cC') \cap \supp(\cC)| = |A|+t$, so that
  \[
    \cX_A = \bigcup_{t=0}^{(r-1)(s-1)s} \cX_{\cC,t}
  \]
  and, consequently,
  \[
    \Ex\left[\con(A,Z)^2\right] = \sum_{\cC \in \cX_A} p^{(r-1)(s-1)s} \cdot \sum_{t = 0}^{(r-1)(s-1)s} |\cX_{\cC,t}| \cdot p^{(r-1)(s-1)s-t}.
  \]
  We may further partition each $\cX_{\cC,t}$ according to the intersection pattern of $\cC' \in \cX_{\cC,t}$ with $\supp(\cC)$. Namely, for every sequence $\bt = (t_{i,j} : i \in \br{r-1}, j \in \br{s})$ with $0 \le t_{i,j} \le s-1$, we let $\cX_{\cC, \bt}$ be the family of all $\cC' \in \cX_A$ such that the $i$th edge of the $j$th star in $\cC'$ (in some arbitrary ordering) intersects $\supp(\cC) \setminus A$ in $t_{i,j}$ vertices, so that
  \[
    \cX_{\cC,t} = \bigcup_{\bt : \sum \bt = t} \cX_{\cC, \bt}.
  \]
  Since each edge of every star comprising each $\cC' \in \cX_A$ intersects $A$ in one vertex, there is a constant $\Gamma'$ that depends only on $r$ and $s$ such that
  \[
    \begin{split}
      |\cX_{\cC,\bt}| & \le \Gamma' \cdot \prod_{i,j} \Delta_{t_{i,j}+1}(\cH) \le \Gamma' \cdot \pH^{\sum_{i,j}t_{i,}} \cdot \left(K \cdot \frac{e(\cH)}{v(\cH)}\right)^{(r-1)s} \\
      & = \Gamma' \cdot K^{(r-1)s} \cdot \pH^{\sum\bt-(r-1)(s-1)s}.
    \end{split}
  \]
  In particular, there is a constant $\Gamma''$ that depends only on $r$, $s$, and $K$ such that
  \begin{equation}
    \label{eq:conAZ-ell2}
    \begin{split}
      \Ex\left[\con(A,Z)^2\right] & \le \Gamma'' \cdot \sum_{\cC \in \cX_A} p^{(r-1)(s-1)s} \cdot \sum_{t = 0}^{(r-1)(s-1)s} \left(\frac{p}{\pH}\right)^{(r-1)(s-1)s-t} \\
      & \le  \Gamma'' \cdot |\cX_A| \cdot \left(\frac{p^2}{\pH}\right)^{(r-1)(s-1)s} \cdot \sum_{t=0}^{(r-1)(s-1)s} c^{-t}.
    \end{split}
  \end{equation}
  Since
  \[
    |\cX_A| \le \Delta_1(\cH)^{(r-1)s} \le \left(K \cdot \frac{e(\cH)}{v(\cH)}\right)^{(r-1)s} = K^{(r-1)s} \cdot \pH^{-(r-1)(s-1)s},
  \]
  summing~\eqref{eq:conAZ-ell2} over all $A \in \cH$ yields
  \[
    \Ex\left[\sum_{A \in \cH} \con(A,Z)^2 \right] \le \Gamma \cdot \left(\frac{p}{\pH}\right)^{2(r-1)(s-1)s}\cdot e(\cH),
  \]
  for some constant $\Gamma$ that depends only on $c$, $r$, $s$, and $K$, as claimed.
\end{proof}

\begin{proof}[Proof of Lemma~\ref{lemma:RS-RC-edges}]
  Let $\cH'$ be the hypergraph obtained from $\cH$ by iteratively removing vertices with degree smaller than $e(\cH) / (2 v(\cH))$. Observe that
  \[
    e(\cH') > e(\cH) - v(\cH) \cdot \frac{e(\cH)}{2v(\cH)} = \frac{e(\cH)}{2}
  \]
  and, consequently, that $\delta(\cH') \ge e(\cH) / (2v(\cH))$. In particular, there are at least
  \[
    e(\cH') \cdot \delta(\cH')^{r-2} \ge 2^{-r} \cdot \frac{e(\cH)^{r-1}}{v(\cH)^{r-2}}
  \]
  sequences $A_1, \dotsc, A_{r-1}$ of edges of $\cH'$ that satisfy $A_1 \cap \dotsb \cap A_{r-1} \neq \emptyset$. If $r=2$, then each such sequence corresponds to $2s$ edges of $\RS$ (there are two colours and $s$ different choices for the centre vertex that makes an edge into a star). Otherwise, if $r \ge 3$, then all but at most
  \[
    O(1) \cdot e(\cH) \cdot \Delta_1(\cH)^{r-3} \cdot \Delta_2(\cH) \ll \frac{e(\cH)^{r-1}}{v(\cH)^{r-2}}
  \]
  of those sequences are orderings of the edges of an $(r-1)$star (whose centre is the unique element of $A_1 \cap \dotsb \cap A_{r-1}$).  Moreover, as $A_1, \dotsc, A_{r-1}$ range over all such sequences and $i$ ranges over $\br{r}$, the sets $E_i(A_1, \dotsc, A_{r-1})$ are distinct edges of $\RS$.

  Similarly, there are at least
  \[
    e(\cH') \cdot \delta(\cH')^{(r-1)s} \ge 2^{-rs} \cdot \frac{e(\cH)^{(r-1)s+1}}{v(\cH)^{(r-1)s}}
  \]
  pairs comprising an edge $A = \{v_1, \dotsc, v_s\}$ of $\cH'$ and a set $\{\bA_1, \dotsc, \bA_s\}$ of sequences of $r-1$ edges of $\cH'$ such that, letting $\bA_j = (A_{j,1}, \dotsc, A_{j,r-1})$, we have $v_j \in A_{j,1} \cap \dotsb \cap A_{j,r-1}$. Moreover, for all but at most
  \[
    O(1) \cdot e(\cH) \cdot \Delta_1(\cH)^{s(r-1)-1} \cdot \Delta_2(\cH) \ll \frac{e(\cH)^{(r-1)s+1}}{v(\cH)^{(r-1)s}}
  \]
  of them, $\bA_1, \dotsc, \bA_s$ are orderings of the edges of $(r-1)$-stars that form a constellation with base edge $A$ and, as $\{\bA_1, \dotsc, \bA_s\}$ and $A$ range over all such pairs and $i$ ranges over~$\br{r}$, the sets $E_i(\bA_1) \cup \dotsb \cup E_i(\bA_s)$ are distinct edges of $\RC$.
\end{proof}

\begin{proof}[Proof of Lemma~\ref{lemma:MS}]
  The lower bound
  \[
    e(\MS) \ge \left(\frac{2^{-r}}{(r-1)!} - o(1)\right) \cdot \frac{e(\cH)^{r-1}}{v(\cH)^{r-2}},
  \]
  is proved analogously to the lower bound on $e(\RS)$ in Lemma~\ref{lemma:RS-RC-edges}. Since an $(r-1)$-star forms a connected hypergraph with $r-1$ edges, we have
  \[
    \Delta_1(\MS) \le \Gamma' \cdot \Delta_1(\cH)^{r-1} \le \Gamma' \cdot \left(K \cdot \frac{e(\cH)}{v(\cH)}\right)^{r-1} \le \Gamma \cdot \frac{e(\MS)}{v(\MS)}
  \]
  and, consequently,
  \[
    e(\MS) \le \Delta_1(\MS) \cdot v(\cH) \le \Gamma \cdot \frac{e(\cH)^{r-1}}{v(\cH)^{r-2}}
  \]
  where $\Gamma$ and $\Gamma'$ are constants that depend only on $r$ and $K$.

  Let $T \subseteq V(\cH)$ be an arbitrary set of size $t \ge 2$. For every sequence $\bt = (t_i)_{i =1}^{r-1}$ satisfying $0 \le t_i \le s-1$ for all $i$, we will bound from above the number of $(r-1)$-stars whose edges intersect $T$ according to $\bt$, that is, the $i$th edge of the star intersects $T$ in $t_i$ vertices (not counting the centre vertex).  Since $\sum t_i = t \ge 2$ and $s \ge 3$, we may assume (by symmetry) that either
  \begin{enumerate}[label=(\roman*)]
  \item
    \label{item:t1-large}
    $2 \le t_1 \le s-1$ or
  \item
    \label{item:t1-small}
    $t_1 = t_i = 1$ for some $i \ge 2$.
  \end{enumerate}
  We may enumerate all stars of the above form as follows:
  \begin{enumerate}[label=(\arabic*)]
  \item
    Choose a labeled partition of $T$ according to the intersection pattern $\bt$; there are at most $(r-1)^t$ such partitions. 
  \item
    Choose the first edge of the star and its centre vertex; there are at most $s \cdot \Delta_{t_1}(\cH)$ choices.
  \item
    Choose the remaining $r-2$ edges of the star; since the centre vertex is already fixed, there are at most $\Delta_{t_i+1}(\cH)$ choices for the $i$th edge.
  \end{enumerate}
  This gives
  \[
    \deg_{\MS}(T) \le O(1) \cdot \max_{\bt} \Big\{ \Delta_{t_1}(\cH) \cdot \prod_{i =2}^{r-1} \Delta_{t_i+1}(\cH) \Big\},
  \]
  where the maximum ranges over all sequences $\bt$ summing to $t$ and satisfying~\ref{item:t1-large} or~\ref{item:t1-small} above.  Since $\cH$ is non-clustered, if~\ref{item:t1-large} holds, then $\Delta_{t_1}(\cH) \ll \pH^{t_1-1} \cdot e(\cH) / v(\cH)$ and if~\ref{item:t1-small} holds, then $\Delta_{t_i+1}(\cH) \ll \pH^{t_i} \cdot e(\cH) / v(\cH)$, as $t_i + 1 = 2 \le s-1$ (we also recall that~\eqref{eq:Delta-t-cH} holds always).  Therefore, we may conclude that
  \[
    \deg_{\MS}(T) \ll O(1) \cdot \max_{\bt} \; \pH^{\sum t_i-1} \cdot \left(\frac{e(\cH)}{v(\cH)}\right)^{r-1} = O\left(\pH^{t-1} \cdot \frac{e(\MS)}{v(\MS)}\right),
  \]
  where the second inequality follows from the lower bound on $e(\MS)$ that we established at the beginning.  This completes the proof of the lemma.
\end{proof}

\begin{proof}[Proof of Lemma~\ref{lemma:RC-degrees}]
  Since a constellation together with its base edge forms a connected hypergraph with $(r-1)s+1$ edges and the number of different rainbow colourings of any given constellation can be bounded by a function of $r$ and $s$ only, there are constants $\Gamma'$ and $\Gamma$ that depend only on $r$, $s$, and $K$ such that
  \[
    \Delta_1(\RC) \le \Gamma' \cdot \Delta_1(\cH)^{(r-1)s+1} = \Gamma' \cdot \left(K \cdot \frac{e(\cH)}{v(\cH)}\right)^{(r-1)s+1} \le \Gamma \cdot \frac{e(\RC)}{v(\RC)},
  \]
  where the last inequality follows from Lemma~\ref{lemma:RS-RC-edges}.
  
  Let $T \subseteq V(\cH) \times \br{r}$ be an arbitrary set of size $t \ge 2$ and let $T_i = T \cap \big(V(\cH) \times \{i\}\big)$. If each of $T_1, \dotsc, T_r$ is nonempty, then $\deg_{\RC} (T) = 0$, so we may assume (by symmetry) that $T_r = \emptyset$. For every sequence
  \[
    \bt = \big(t_{i,j}: i \in \br{r-1}, j \in \br{s}\big)
  \]
  satisfying $0 \le t_{i,j} \le s-1$ for all $i$ and $j$, we will bound from above the number of rainbow constellations whose $r$-rainbow stars intersect the set $T$ according to $\bt$, that is, the edge coloured $i$ of the $j$th star intersects $T_i \subseteq T$ in $t_{i,j}$ vertices.  Since $\sum t_{i,j} = t \ge 2$ and $s \ge 3$, we may assume (by symmetry) that either
  \begin{enumerate}[label=(\roman*)]
  \item
    \label{item:t11-large}
    $2 \le t_{1,1} \le s-1$ or
  \item
    \label{item:t11-small}
    $t_{1,1} = t_{i,j} = 1$ for some $(i,j) \neq (1,1)$.
  \end{enumerate}
  We may enumerate all constellations of the form described above as follows:
  \begin{enumerate}[label=(\arabic*)]
  \item
    Choose a labeled partition of $T_1, \dotsc, T_{r-1}$ according to the intersection pattern $\bt$; there are at most $[(r-1)s]^t$ such partitions.
  \item
    Choose the edge coloured $1$ of the first star; there are at most $\Delta_{t_{1,1}}(\cH)$ such edges.
  \item
    Choose the base edge of the constellation; there are at most $(s-t_{1,1}) \cdot \Delta_1(\cH)$ choices.
  \item
    Choose all the $s(r-1)-1$ remaining edges of all the stars forming the constellation one-by-one; since the base edge is already fixed, there are at most $\Delta_{t_{i,j}+1}(\cH)$ choices for the edge coloured $i$ of the $j$th star.
  \end{enumerate}
  This gives
  \[
    \deg_{\RC}(T) \le O(1) \cdot \Delta_1(\cH) \cdot \max_{\bt} \Big\{ \Delta_{t_{1,1}}(\cH) \cdot \prod_{(i,j) \neq (1,1)} \Delta_{t_{i,j}+1}(\cH) \Big\},
  \]
  where the maximum ranges over all sequences $\bt$ summing to $t$ and satisfying either~\ref{item:t11-large} or~\ref{item:t11-small} above.  Since $\cH$ is non-clustered, if~\ref{item:t11-large} holds, then $\Delta_{t_{1,1}}(\cH) \ll \pH^{t_{1,1}-1} \cdot e(\cH) / v(\cH)$ and if~\ref{item:t11-small} holds, then $\Delta_{t_{i,j}+1}(\cH) \ll \pH^{t_{i,j}} \cdot e(\cH) / v(\cH)$, as $t_{i,j}+1 = 2 \le s-1$ (we also recall that~\eqref{eq:Delta-t-cH} holds always).   Therefore, we may conclude that
  \[
    \deg_{\RC}(T) \ll O(1) \cdot \max_{\bt} \; \pH^{\sum t_{i,j}-1} \cdot \left(\frac{e(\cH)}{v(\cH)}\right)^{(r-1)s+1} = O\left(\pH^{t-1} \cdot \frac{e(\RC)}{v(\RC)}\right),
  \]
  where the last inequality follows from Lemma~\ref{lemma:RS-RC-edges}.  This completes the proof of the lemma.
\end{proof}

\begin{proof}[{Proof of Lemma~\ref{lemma:RS-RC-property-edges}}]
  Suppose that $e\big(\RS[C]\big) \ge \beta \cdot e(\RS)$ and let $W = \pi_1(C)$, where $\pi_1 \colon V(\cH) \times \br{r} \to V(\cH)$ is the projection on the first coordinate.  Define a random colouring $\psi \colon W \to \br{r}$ as follows: For every $v \in W$, let $\psi(v)$ be the uniformly random colour $i \in \br{r}$ such that $(v,i) \in C$; in other words, $\psi(v)$ is the uniformly chosen random element of $\pi_2\big(C \cap (\{v\} \times \br{r})\big)$; here, $\pi_2$ is the projection of $V(\cH) \times \br{r}$ on the second coordinate.  Since $\RS$ is $(r-1)(s-1)$-uniform, we have
  \begin{equation}
    \label{eq:eRSC-random-colouring}
    \Ex\left[e\big(\RS[\psi]\big)\right] \ge r^{-(r-1)(s-1)} \cdot e\big(\RS[C]\big) \ge \beta r^{-(r-1)(s-1)} \cdot e(\RS).
  \end{equation}
  From now on, let $\psi \colon W \to \br{r}$ be an arbitrary colouring for which~\eqref{eq:eRSC-random-colouring} holds without the expectation.

  Since Lemma~\ref{lemma:RS-RC-edges} supplies a positive constant $c$ such that $e(\RS) \ge c \cdot e(\cH)^{r-1} / v(\cH)^{r-2}$, assumption~\ref{item:ass-star-constellation}, see Definition~\ref{dfn:rainbow-stars-constellations}, assures that
  \[
    e\big(\RC[\psi]\big) \ge \eps' \cdot \frac{e(\cH)^{(r-1)s+1}}{v(\cH)^{r-1)s}}
  \]
  for some positive constant $\eps'$ that depends only on $\beta$, $r$, and the sequence $\cH$. Finally, Lemma~\ref{lemma:RC-degrees} supplies a constant $\Gamma'$ that depends only on $r$, $s$, and $K$ such that $\Delta_1(\RC) \le \Gamma' \cdot e(\cH)^{(r-1)s+1} / v(\cH)^{(r-1)s+1}$. Consequently,
  \[
    e\big(\RC[C]\big) \ge e\big(\RC[\psi]\big) \ge \frac{\eps'}{\Gamma'} \cdot v(\cH) \cdot \Delta_1(\RC) = \frac{\eps'}{\Gamma' r} \cdot v(\RC) \cdot \Delta_1(\RC) \ge \frac{\eps'}{\Gamma' r} \cdot e(\RC),
  \]
  which concludes the proof of the lemma.
\end{proof}

\begin{proof}[Proof of Lemma~\ref{lemma:MS-pseudo-variance}]
  We have
  \[
    \begin{split}
      \pVar(\MS[Z]) & = \sum_{\substack{A, B \in \MS \\ A \cap B \neq \emptyset}} p^{|A \cup B|} \le \sum_{A \in \MS} p^{|A|} \sum_{\emptyset \neq T \subseteq A} \sum_{B \cap A = T} p^{|B \setminus A|} \\
      & \le \Ex[e(\MS[Z])] \cdot \sum_{t = 1}^{(r-1)(s-1)} \binom{(r-1)(s-1)}{t} \cdot \Delta_t(\MS) \cdot p^{(r-1)(s-1)-t}.
    \end{split}
  \]
  By Lemma~\ref{lemma:MS}, for some constant $\Gamma'$ that depends only on $r$ and the sequence $\cH$,
  \[
    \begin{split}
      \pVar(\MS[Z]) & \le \Gamma' \cdot \Ex[e(\MS[Z])] \cdot \sum_{t=1}^{(r-1)(s-1)} p^{(r-1)(s-1)-t} \cdot \pH^{t-1} \cdot \frac{e(\MS)}{v(\MS)} \\
      & = \frac{\Gamma' \cdot \Ex[e(\MS[Z])]^2}{p v(\MS)} \cdot \sum_{t=1}^{(r-1)(s-1)} \left(\frac{\pH}{p}\right)^{t-1},
    \end{split}
  \]
  which, by our assumption that $p \ge c \cdot \pH$ implies the claimed bound on $\pVar(\MS[Z])$.
\end{proof}

\subsection{Proof of Theorem~\ref{thm:sparse-rainbow-star-constellation}}
\label{sec:proof-sparse-rainbow-star-constellation}

We wish to show that, with probability close to one, every colouring of $Z \sim V(\cH)_p$ that admits only a small number of rainbow constellations will also admit only a small number of rainbow stars.  To do this, we will apply the container lemma to the hypergraph $\RC$ and conclude that every such colouring is contained in one of $\exp(o(pN))$ subsets of $V(\cH) \times \br{r}$, each of which induces $o(e(\RC))$ rainbow constellations.  By Lemma~\ref{lemma:RS-RC-property-edges}, each such container can only induce $o(e(\RS))$ rainbow stars.  Intuitively, it seems plausible that every colouring of $Z$ residing inside each container should also have a small number of rainbow stars, as we wanted.  In order to show this, however, we appear to need an upper bound on the upper tail of the number of rainbow stars in the intersection of $Z \times \br{r}$ with a given container $C$ that is strong enough to survive the union bound over all containers.  Unfortunately, the upper tail is most likely too heavy to permit such a naive union bound.  We will avert this problem by showing that the overall number of stars in $Z$ is concentrated (for our purposes, a simple second moment argument would do, which is the task of  Lemma~\ref{lemma:MS-pseudo-variance}) and then transform the question of bounding the upper tail of the number of rainbow stars in $C \cap (Z \times \br{r})$ to that of bounding the lower tail of the number of rainbow stars that are not contained in $C \cap (Z \times \br{r})$;  here, Janson's inequality provides an adequate, exponential bound.

\begin{proof}[Proof of Theorem~\ref{thm:sparse-rainbow-star-constellation}]
  For the sake of brevity, we write $V$ in place of $V(\cH)$.  Let $\Gamma \coloneqq \Gamma_{\ref{lemma:RC-degrees}}$, let $\beta \coloneqq \betas/(3\Gamma r!)$, let $\eps$ be the constant supplied by Lemma~\ref{lemma:RS-RC-property-edges} invoked with $\beta_{\ref{lemma:RS-RC-property-edges}} = \beta$,  let $c \coloneqq c_{\ref{lemma:RS-RC-edges}}$, and let $t$ and $\delta$ be the constants from the assertion of the container lemma (Theorem~\ref{thm:containers}) invoked with $k_{\ref{thm:containers}} = (r-1)(s-1)s$, $K_{\ref{thm:containers}} = \Gamma$, and $\eps_{\ref{thm:containers}} = \eps$.  Further, let $T = T(t, r, \beta, \Gamma)$ be sufficiently large so that
  \begin{equation}
    \label{eq:T}
    \left(\frac{2^teT}{t}\right)^{tr/T} \le \exp\left(\frac{\beta^2}{5\Gamma}\right)
  \end{equation}
  and let $\tau \coloneqq p/T$. Finally, let $\betac \coloneqq c\delta \cdot T^{-(r-1)(s-1)s}$ and note that
  \begin{equation}
    \label{eq:betac-delta}
    \frac{\betac}{c} \cdot p^{(r-1)(s-1)s} \le \delta \tau^{(r-1)(s-1)s}.
  \end{equation}

  Since Lemma~\ref{lemma:RC-degrees} implies that, for every $\ell \in \br{(r-1)(s-1)s}$,
  \[
    \Delta_\ell(\RC) \le \Gamma \cdot \tau^{\ell-1}\cdot \frac{e(\RC)}{v(\RC)},
  \]
  Theorem~\ref{thm:containers} supplies a function $f \colon \cP\big(V \times \br{r}\big)^t \to \cP\big(V \times \br{r}\big)$ such that:\footnote{Recall that we view partial $\br{r}$-colourings of $V$ as subsets of $V \times \br{r}$.}
  \begin{enumerate}[label=(\roman*)]
  \item
    \label{item:containers-RC-i}
    For every partial $\br{r}$-colouring $\psi$ with fewer than $\delta \tau^{(r-1)(s-1)s} \cdot e(\RC)$ rainbow constellations, there are $S_1, \dotsc, S_t \subseteq \psi$ with at most $r \tau v(\cH)$ elements each such that $\psi \subseteq f(S_1, \dotsc, S_t)$.
  \item
    \label{item:containers-RC-ii}
    For every $S_1, \dotsc, S_t \subseteq V \times \br{r}$, the set $f(S_1, \dotsc, S_t)$ induces fewer than $\eps e(\RC)$ edges in $\RC$ and thus, by Lemma~\ref{lemma:RS-RC-property-edges}, fewer than $\beta e(\RS)$ edges in $\RS$.
  \end{enumerate}
  
  Suppose now that $Z$ fails to satisfy the assertion of the theorem, that is, there is a partial $\br{r}$-colouring $\psi$ of $Z$ such that
  \[
    e\big(\RS[\psi]\big) \ge \betas \cdot \left(\frac{p}{\pH}\right)^{(r-1)(s-1)} \cdot v(\cH) = \betas \cdot p^{(r-1)(s-1)} \cdot \frac{e(\cH)^{r-1}}{v(\cH)^{r-2}}
  \]
  but, nevertheless,
  \[
    e\big(\RC[\psi]\big) < \betac \cdot \left(\frac{p}{\pH}\right)^{(r-1)(s-1)s} \cdot e(\cH) = \betac \cdot p^{(r-1)(s-1)s} \cdot \frac{e(\cH)^{(r-1)s+1}}{v(\cH)^{(r-1)s}}.
  \]
  Consequently, by Lemma~\ref{lemma:MS},
  \begin{equation}
    \label{eq:eRSpsi-lower}
    e(\RS[\psi]) \ge \frac{\betas}{\Gamma} \cdot p^{(r-1)(s-1)} \cdot e(\MS) = 3\beta \cdot r! \cdot \Ex\left[e\big(\MS[Z]\big)\right]
  \end{equation}
   and, by Lemma~\ref{lemma:RS-RC-edges} and~\eqref{eq:betac-delta},
  \begin{equation}
    \label{eq:eRCpsi-upper}
    e\big(\RC[\psi]\big) < \frac{\betac}{c} \cdot p^{(r-1)(s-1)s} \cdot e(\RC) \le \delta \tau^{(r-1)(s-1)s} \cdot e(\RC).
  \end{equation}
  Property~\ref{item:containers-RC-i} implies that there are sets $S_1, \dotsc, S_t \subseteq V \times \br{r}$ with at most $r\tau v(\cH)$ elements each such that $\pi_1$, the projection onto the first coordinate, maps $S_1 \cup \dotsb \cup S_t$ injectively to $Z$ and $\psi$ is contained in the set $f(S_1, \dotsc, S_t)$. In particular, by~\eqref{eq:eRSpsi-lower},
  \begin{equation}
    \label{eq:eRS-container-lower}
    e\left(\RS\big[f(S_1, \dotsc, S_t) \cap (Z \times \br{r})\big]\right) \ge e(\RS[\psi]) \ge 3\beta \cdot r! \cdot \Ex\left[e\big(\MS[Z]\big)\right].    
  \end{equation}
  On the other hand, property~\ref{item:containers-RC-ii} states that $f(S_1, \dotsc, S_t)$ induces fewer than $\beta e(\RS)$ edges in $\RS$.  We will now show that it is unlikely that this holds for any sequence $S_1, \dotsc, S_t$.

  Let $\cU$ be the event that
  \[
    e\big(\MS[Z]\big) \le (1 + \beta) \cdot \Ex\left[e\big(\MS[Z]\big)\right].
  \]
  It follows from Lemma~\ref{lemma:MS-pseudo-variance} and Markov's inequality that
  \[
    \Pr(\cU^c) \le \frac{4}{\beta^2} \cdot \frac{\Var\big(e(\MS[Z])\big)}{\Ex\big[e(\MS[Z])\big]^2} \le \frac{4\Gamma}{\beta^2 \cdot p v(\cH)} = o(1),
  \]
  as $pv(\cH) \to \infty$, see Fact~\ref{fact:non-clustered-properties}.

  Given a $C \subseteq V \times \br{r}$, define
  \[
    \RS^C \coloneqq \RS \setminus \RS[C]
  \]
  and let $\MS^C$ be the multiset projection of $\RS^C$ onto $V$, so that $e(\MS^C) = e(\RS^C)$. In particular,
  \begin{equation}
    \label{eq:RS-MS-edges}
    \begin{split}
      e\left(\RS\big[C \cap (Z \times \br{r})\big]\right) & = e\left(\RS\big[Z \times \br{r}\big]\right) - e\left(\RS^C\big[Z \times \br{r}\big]\right) \\
      & = r! \cdot e\big(\MS[Z]\big) - e\big(\MS^C[Z]\big),
    \end{split}
  \end{equation}
  where the final equality holds because every $(r-1)$-star admits exactly $r!$ rainbow colourings (and therefore every edge of $\MS[Z]$ corresponds to $r!$ edges of $\RS\big[Z \times \br{r}\big]$).

  We conclude this discussion with the following observation: If $\cU$ holds, then inequality~\eqref{eq:eRS-container-lower} and identity~\eqref{eq:RS-MS-edges} with $C = f(S_1, \dotsc, S_t)$ imply that
  \begin{equation}
    \label{eq:MS-container-bad-event}
    e\big(\MS^{f(S_1, \dotsc, S_t)}[Z]\big) \le \left(1 - 2\beta \right) \cdot r! \cdot \Ex\left[e\big(\MS[Z]\big)\right].
  \end{equation}

  Let $\Seq$ denote the collection of all sequences $(S_1, \dotsc, S_t)$ of $t$ subsets of $V \times \br{r}$ with at most $r \tau v(\cH)$ elements each such that $\pi_1$ restricted to $S_1 \cup \dotsb \cup S_t$ is injective. Given a sequence $\cS = (S_1, \dotsc, S_t) \in \Seq$, let $\cY_{\cS}$ denote the event that $\pi_1(S_1 \cup \dotsb \cup S_t) \subseteq Z$ and let $\cL_{\cS}$ denote the event that~\eqref{eq:MS-container-bad-event} holds; note that $\cY_S$ is increasing (in $Z$) whereas $\cL_{\cS}$ is decreasing. Finally, let $\cF$ denote the event that $Z$ fails the assertion of the theorem. The above discussion and Harris's inequality imply that
  \begin{equation}
    \label{eq:PrcXT-upper}
    \begin{split}
      \Pr(\cF) & \le \Pr(\cU^c) + \sum_{\cS \in \Seq} \Pr(\cL_{\cS} \cap \cY_{\cS}) \le \Pr(\cU^c) + \sum_{\cS \in \Seq} \Pr(\cL_{\cS}) \cdot \Pr(\cY_{\cS}) \\
      & \le \max_{\cS \in \Seq} \Pr(\cL_\cS) \cdot \sum_{\cS \in \Seq} \Pr(\cY_\cS) + o(1).
    \end{split}
  \end{equation}

  \begin{claim}
    \label{claim:Janson-RSC}
    For every $\cS \in \Seq$,
    \[
      \Pr(\cL_\cS) \le \exp\left(- \frac{\beta^2}{2\Gamma} \cdot p v(\cH)\right).
    \]
  \end{claim}
  \begin{proof}[Proof of Claim~\ref{claim:Janson-RSC}]
    Fix an arbitrary sequence $\cS = (S_1, \dotsc, S_t) \in \Seq$ and recall that
    \[
      \begin{split}
        e\big(\MS^{f(S_1, \dotsc, S_t)}\big) & = e\big(\RS^{f(S_1, \dotsc, S_t)}\big) = e(\RS) - e\left(\RS\big[f(S_1, \dotsc, S_t)\big]\right) \\
        & \ge \left(1-\beta\right) \cdot e(\RS) = (1 - \beta) \cdot r! \cdot e(\MS)
      \end{split}
    \]
    or, equivalently,
    \begin{equation}
      \label{eq:MS-container-bad-event-Ex}
      \Ex\left[e\big(\MS^{f(S_1, \dotsc, S_t)}[Z]\big)\right] \ge (1-\beta) \cdot r! \cdot \Ex\left[e\big(\MS[Z]\big)\right].
    \end{equation}
    Let $X \coloneqq e\big(\MS^{f(S_1, \dotsc, S_t)}[Z]\big)$ and let $\mu \coloneqq r! \cdot \Ex\left[e\big(\MS[Z]\big)\right]$, so that~\eqref{eq:MS-container-bad-event} and~\eqref{eq:MS-container-bad-event-Ex} can be rewritten as $X \le (1-2\beta) \mu$ and $\Ex[X] \ge (1-\beta) \mu$, respectively.  It follows from Janson's inequality (Theorem~\ref{thm:Janson}) that
    \[
      \Pr(\cL_\cS) = \Pr\left(X \le (1-2\beta)\mu\right) \le \Pr\left(X \le \Ex[X] - \beta \mu\right) \le \exp\left(-\frac{\beta^2 \mu^2}{2 \pVar\big(\MS^{f(S_1, \dotsc, S_t)}[Z]\big)}\right),
    \]
    where, similarly as in Lemma~\ref{lemma:MS-pseudo-variance}, we write $\pVar(\MS^{f(S_1, \dotsc, S_t)}[Z])$ to denote the pseudo-variance of the sequence of events $(\{A \subseteq Z\})_{A \in \MS^{f(S_1, \dotsc, S_t)}}$.  Since $\MS^{f(S_1, \dotsc, S_t)} \subseteq r! \cdot \MS$, we have (using an analogous notational convention)
    \[
      \pVar(\MS^{f(S_1, \dotsc, S_t)}[Z]) \le \pVar(r! \cdot \MS[Z]) = (r!)^2 \cdot \pVar(\MS[Z])
    \]
    Further, by Lemma~\ref{lemma:MS-pseudo-variance},
    \[
      \pVar(\MS[Z]) \le \Gamma \cdot \frac{\Ex[e(\MS[Z])]^2}{p v(\MS)} = \Gamma \cdot \frac{(\mu/r!)^2}{p v(\cH)}.
    \]
    Substituting this estimate back into the upper bound on $\Pr(\cL_\cS)$ gives the assertion of the claim.
  \end{proof}
  
  Finally, we derive an upper bound on the sum in the right-hand side of~\eqref{eq:PrcXT-upper}. To this end, for every integer $m$, we let
  \[
    \Seq_m \coloneqq \big\{(S_1, \dotsc, S_t) \in \Seq : |S_1 \cup \dotsb \cup S_t| = |\pi_1(S_1 \cup \dotsb \cup S_t)|= m\big\}
  \]
  and note that
  \[
    \Seq = \bigcup_{m=0}^{tr\tau v(\cH)} \Seq_m.
  \]
  It is not hard to see that, for every $m$,
  \[
    |\Seq_m| \le \binom{v(\cH)}{m} \cdot \left(2^t r\right)^m \le \left(\frac{2^terv(\cH)}{m}\right)^m
  \]
  and $\Pr(\cY_\cS) = p^m$ for every $\cS \in \Seq_m$. Consequently, since $\tau \le p$ and, for every positive real $a$, the function $x \mapsto (ea/x)^x$ is increasing on the interval $[0,a]$,
  \[
    \begin{split}
      \sum_{\cS \in \Seq} \Pr(\cY_\cS) & = \sum_{m=0}^{tr\tau v(\cH)} \left(\frac{2^terpv(\cH)}{m}\right)^m \le v(\cH) \cdot \left(\frac{2^tep}{t\tau}\right)^{tr\tau v(\cH)} \\
      & = v(\cH) \cdot \left(\frac{2^teT}{t}\right)^{(tr/T) \cdot pv(\cH)} \by{\eqref{eq:T}}{\le} \exp\left(\frac{\beta^2}{4\Gamma} \cdot pv(\cH)\right).
    \end{split}
  \]
  
  We may finally conclude that 
  \[
    \Pr(\cF) \le o(1) + \exp \left( \frac{\beta^2}{4\Gamma} \cdot pv(\cH) - \frac{\beta^2}{2\Gamma} \cdot pv(\cH)\right) = o(1),
  \]
  where we again used the assumption that $pv(\cH) \to \infty$, see Fact~\ref{fact:non-clustered-properties}.
\end{proof}

\newpage

\section{Applications}
\label{sec:applications}

In this section, we will use our general Theorem~\ref{thm:main} to prove Theorems~\ref{thm:main-Ramsey}, \ref{thm:main-vdW}, and~\ref{thm:main-Schur}.  In the following three subsections, we will verify that Theorem~\ref{thm:main} can be applied to hypergraphs that naturally arise in the context of Ramsey questions for: graphs (Section~\ref{sec:graphs}), arithmetic progressions (Section~\ref{sec:arithm-progr}), and Schur triples (Section~\ref{sec:schurs-theorem}).  In each case, we will verify the list of assumptions of the theorem, which will swiftly award us with a sharp threshold result for the corresponding Ramsey problem.  As a reminder, the assumptions on the hypergraph $\cH$ are:
\begin{enumerate}[label=(A\arabic*)]
\item is symmetric,
\item it is non-clustered, 
\item non-colourability has a threshold at $\pH$,
\item it satisfies $2$-choosability for typical bounded-sized subsets, and
\item it satisfies the rainbow star-constellation property.
\end{enumerate}

The first two assumptions are more technical in nature and will mostly set the framework of the application.  For example, in the arithmetic setting, \ref{item:ass-symmetry} will force us to work in $\ZZ_N$ instead of $\br{N}$ whereas, for graphs, \ref{item:ass-non-clusteredness} is akin to requiring the graph to be strictly $2$-balanced.  Assumption~\ref{item:ass-weak-threshold} was the subject of previous works.  Having said that, in the arithmetic setting of Theorems~\ref{thm:main-vdW} and~\ref{thm:main-Schur}, the $0$-statements implicit in~\ref{item:ass-weak-threshold} are marginally stronger than what was explicitly established by previous works, due to the fact that we are working in modular arithmetic.  Even though these earlier works can be adapted to yield~\ref{item:ass-weak-threshold}, we will not dwell on it and instead establish stronger forms of these $0$-statements that also imply~\ref{item:ass-choosability}, Theorems~\ref{thm:list-vdW} and~\ref{thm:list-Schur}.  These follow from the general Theorem~\ref{thm:choosability-linear-hypergraphs}, which will be proved in Section~\ref{sec:choosability-almost-linear}.  (We also recall that the $1$-statements can be derived using Proposition~\ref{prop:robust-non-col-1-statement} since the hypergraphs in all of our applications are robustly non-colourable.)  As for~\ref{item:ass-star-constellation}, even though it is rather straightforward to verify in the context of Theorem~\ref{thm:main-vdW}, checking it in the remaining two cases is far from easy.  (We recall here that this assumption is not satisfied in all cases of interest, see Appendix~\ref{apx:graph_application}.)  Summarising, the bulk of the work in this section will be spent in verifying the last two assumptions.

\subsection{Graphs}
\label{sec:graphs}

In this section, we will prove Theorem~\ref{thm:main-Ramsey}, which asserts that the property $G_{n,p} \to (H)_r$ has a sharp threshold for certain pairs of $H$ and $r$.  Given a strictly $2$-balanced graph $H$, let $\cH_H$ be the hypergraph of copies of $H$ in $K_n$ whose vertices are the edges of $K_n$ and whose hyperedges are (the edge sets of) all copies of $H$.  It is straightforward to check that $s = e_H$, $v(\cH_H) = \Theta(n^2)$, $e(\cH_H) = \Theta(n^{v_H})$, and $p_{\cH_H} = \Theta(n^{-1/m_2(H)})$.  Moreover, it is routine to verify that $\cH_H$ is non-clustered (as $H$ is strictly $2$-balanced) and that it is symmetric (due to the symmetries of $K_n$).  It was proved by R\"odl and Ruci\'nski~\cite{RodRuc95} that, for each $r \ge 2$, the threshold for non-$r$-colourability of $(\cH_H)_p$ is located at $\pH$.  This, assumptions \ref{item:ass-symmetry}--~\ref{item:ass-weak-threshold} are met.  We now turn to verifying the last two assumptions. 

\subsubsection{Rainbow star-constellation property for graphs}

A rainbow star in the hypergraph $\cH_H$ of copies of a graph $H$ in $K_n$ is comprised of $r-1$ monochromatic copies of $H$ minus some edge (not necessarily the same edge in different copies), each coloured with a different colour, that are glued on that missing edge, called the centre of the star.  A rainbow constellation is a union of edge-disjoint rainbow stars whose centres form a copy of $H$.  We will tacitly assume that all stars and constellations are generic in the sense that the copies of $H$ minus an edge that form them do not share more vertices than necessary: every rainbow star will have $v_S \coloneqq (v_H - 2) \cdot (r-1) + 2$ vertices and every rainbow constellation will have $v_C \coloneqq e_H \cdot (v_H - 2) \cdot (r-1) + v_H$ vertices.

To prove the rainbow star-constellation property in $r$ colours we will need to show that any collection of $\Theta(e(\cH_H)^{r-1} / v(\cH_H)^{r-2}) = \Theta(n^{v_S})$ many $(r-1)$-stars, induces $\Theta(e(\cH_H)^{s(r-1)+1} / v(\cH_H)^{s(r-1)}) = \Theta(n^{v_C})$ many $(r-1)$-constellations. Put differently, if we have a collection achieving the full count of rainbow stars in $K_n$, then it induces the full count of rainbow constellations.

Let $F$ and $G$ be edge-coloured graphs.  A homomorphism from $F$ to $G$ is a function $\varphi \colon V(F) \to V(G)$ such that, for every $uv \in F$, the pair $\{\varphi(u), \varphi(v)\}$ is an edge of $G$ with the same colour as $uv$. We wish to prove the following characterisation for the rainbow star-constellation property.  

\begin{prop}
  \label{prop:graph_rainbow_sc}
  The hypergraph $\cH_H$ of copies of a given graph $H$ in $K_n$ has the rainbow star-constellation property if and only if every rainbow star $S$ of $H$ admits a rainbow constellation $C$ of $H$ with a homomorphism $C \to S$.
\end{prop}

In order to prove this statement, we will first gather some definitions and tools.  Given edge-coloured graphs $F$ and $G$, let $\Hom(F,G)$ denote the family of all homomorphisms from $F$ to $G$ and let $\hom(F, G)$ denote the density of $\Hom(F,G)$ in $V(G)^{V(F)}$, that is,
\[
  \hom(F, G) \coloneqq \frac{|\Hom(F,G)|}{v_G^{v_F}}.
\]
In other words, $\hom(F, G)$ is the probability that a random function $\varphi \colon V(F) \to V(G)$ belongs to $\Hom(F, G)$.

\begin{obs}
  \label{obs:hom-monotone}
  Let $F$ and $G$ be graphs. If $F' \subseteq F$, then $\hom(F',G) \ge \hom(F,G)$.
\end{obs}
\begin{proof}
  Indeed, if $\varphi \in \Hom(F,G)$, then $\varphi|_{V(F')} \in \Hom(F',G)$.  Therefore, $|\Hom(F,G) | \le |\Hom(F',G)| \cdot {v_G}^{v_F - v_{F'}}$ and the result follows.
\end{proof}

For a graph $F$ and a positive integer $k$, the $k$-blowup of $F$, denoted by $F^{(k)}$, is the graph obtained from $F$ by replacing each vertex $a$ of $F$ by an independent set $V_a$ of size $k$ and every edge $ab$ of $F$ by a complete bipartite graph between $V_a$ and $V_b$.

\begin{obs}
  \label{obs:hom-blowup}
  $\Hom(F, G) \neq \emptyset$ if and only if $F \subseteq G^{(v_F)}$.
\end{obs}

The following lemma should be folklore.

\begin{lemma}
  \label{lemma:Hom-blowup}
  Suppose that $F$ and $G$ are graphs and let $k$ be a positive integer. Then
  \[
    \hom(F^{(k)}, G) \ge \hom(F,G)^{k \cdot v_F}.
  \]
\end{lemma}
\begin{proof}
  Let $v = v_F$ and let $u_1, \dotsc, u_v$ be an arbitrary ordering of the vertices of $F$.  For $i \in \{0, \dotsc, v\}$ let $F_i$ denote the graph obtained from $F$ by blowing up vertices $u_1, \dotsc, u_i$ by a factor of $k$, so that $F_0 = F$ and $F_v = F^{(k)}$. It suffices to show that, for every $i \in \br{v}$,
  \[
    \hom(F_i, G) \ge \hom(F_{i-1}, G)^k.
  \]
  Let $\varphi \colon V(F_i) \to V(G)$ be a random function, let $u_{i,1}, \dotsc, u_{i,k}$ be the $k$ copies of $u_i$ in $F_i$, and let $V_i' = V(F_i) \setminus \{u_{i,1}, \dotsc, u_{i,k}\} = V(F_{i-1}) \setminus \{u_i\}$. We have
  \begin{multline*}
    \hom(F_i, G) = \Pr\big(\varphi \in \Hom(F_i, G)\big) = 
    \Pr\big(\varphi|_{V_i'} \in \Hom(F_i[V_i'], G ) \big) \\
    \cdot \Pr\big(\varphi \in \Hom(F_i, G) \mid \varphi|_{V_i'} \in \Hom(F_i[V_i'], G ) \big)
  \end{multline*}
  Since $F_i[V_i'] = F_{i-1}[V_i']$, we have
  \[
    \Pr\big(\varphi|_{V_i'} \in \Hom(F_i[V_i'], G ) \big) = \Pr\big(\varphi|_{V_i'} \in \Hom(F_{i-1}[V_i'], G])\big) =  \hom(F_{i-1}[V_i'],G).
  \]
  Crucially, since $\varphi(u_{i,1}), \dotsc, \varphi(u_{i,k})$ are independent, uniformly random elements of $V(G)$ and $\{u_{i,1}, \dotsc, u_{i,k}\}$ is an independent set in $F_i$,
  \begin{multline*}
    \Pr\big(\varphi \in \Hom(F_i, G) \mid \varphi|_{V_i'} \in \Hom(F_i[V_i'], G ) \big) \\
    = \prod_{j=1}^k\Pr\big(\varphi|_{V_i' \cup \{u_{i,j}\}} \in \Hom(F_i[V_i' \cup \{u_{i,j}\}], G ) \mid \varphi|_{V_i'} \in \Hom(F_i[V_i'], G )\big).
  \end{multline*}
  Finally, since $F_i[V_i' \cup \{u_{i,j}\}] \cong F_{i-1}$, we conclude that, letting $\psi \colon V(F_{i-1}) \to V(G)$ be a uniformly random map,
  \[
    \begin{split}
      \hom(F_i, G) & = \hom(F_{i-1}[V_i'], G) \cdot \Pr\big(\psi \in \Hom(F_{i-1}, G) \mid \psi|_{V_i'} \in \Hom(F_{i-1}[V_i'], G)\big)^k \\
      & \ge \hom(F_{i-1}[V_i'], G)^k \cdot \Pr\big(\psi \in \Hom(F_{i-1}, G) \mid \psi|_{V_i'} \in \Hom(F_{i-1}[V_i'], G)\big)^k \\
      & = \hom(F_{i-1},G)^k,
    \end{split}
  \]
  as claimed.
\end{proof}

\begin{proof}[Proof of Proposition~\ref{prop:graph_rainbow_sc}]
  First, suppose that, for some rainbow star $S$, there is no constellation $C$ such that $C \to S$.  The blowup $S^{(\lfloor n/v_S \rfloor)}$ is an edge-coloured subgraph of $K_n$ with $\Omega(n^{v_S})$ rainbow stars but no rainbow constellations, see Observation~\ref{obs:hom-blowup}.

  For the other direction, suppose that $G$ is an edge-coloured subgraph of $K_n$ that has $\Omega(n^{v_S})$ rainbow stars.  Since there are only $O(1)$ isomorphism types of stars of $H$, there must be some rainbow star $S$ of $H$ such that $\hom(S,G)$ is bounded from below by a positive constant.  Let $C$ be a rainbow constellation satisfying $C \to S$, that is, $C \subseteq S^{(v_C)}$, see Observation~\ref{obs:hom-blowup}. By Observation~\ref{obs:hom-monotone} and Lemma~\ref{lemma:Hom-blowup},
  \[
    \hom(C,G) \ge \hom(S^{(v_C)},G) \ge \hom(S,G)^{v_C \cdot v_S} = \Omega(1).
  \]
  In particular, $G$ contains $\Omega(n^{v_C})$ rainbow constellations.
\end{proof}

We will now derive several sufficient conditions for the rainbow star-constellation property that are easier to verify than the abstract criterion provided by Proposition~\ref{prop:graph_rainbow_sc}.

\begin{cor}
  \label{cor:H-bipartite-rainbow-sc}
  If $H$ is bipartite, then $\cH_H$ has the rainbow star-constellation property.
\end{cor}

\begin{proof}
  We wish to show that every rainbow star $S$ of a bipartite graph admits a rainbow constellation $C$ such that $C \to S$.  Given a star $S$, we will build the required constellation $C$ in the following way.  Begin with a copy of $H$ with partite sets $U$ and $V$.  For all $u \in U$ and $v \in V$ that are adjacent in $H$, place a copy of $S$ centred at $uv$.  It is not hard to check that the function that maps every vertex in each copy of $S$ in $C$ to its corresponding vertex of $S$ is a homomorphism from $C$ to $S$.
\end{proof}

Next, observe that the rainbow star-constellation property for $r$ colours is monotone decreasing in $r$.  However, as it turns out, whenever it holds with $r = 3$, it also holds for all $r$ strictly greater than three.  We will show this by describing another equivalent property that does not mention rainbow structures and instead deals solely with the symmetries of $H$.

\begin{dfn}
  A graph $H$ is \emph{collapsible} if, for every edge $e \in H$ and every vertex $a \in e$, there is an edge $f \in H$ and a homomorphism $H\setminus f \to H \setminus e$ mapping both endpoints of $f$ to $a$.
\end{dfn}

\begin{cor}
  \label{cor:H-collapsible-rainbow-sc}
  Suppose that $r > 2$ and that $H$ is not bipartite. Then $\cH_H$ has the rainbow star-constellation property if and only if $H$ is collapsible.
\end{cor}

\begin{proof}
  Suppose first that $\cH_H$ has the rainbow star-constellation property for $r>2$.  Let $e=\{a,b\}$ be an edge of $H$.  We will construct a rainbow star $S$ centred at some $uv$ using copies of $H \setminus e$ as follows:  In the first colour, use a copy of $H \setminus e$ so that $u$ plays the role of $a$ and $v$ plays the role of $b$.  In every other colour (the number of colours is $r-1 \ge 2$), we switch the roles -- $u$ plays the role of $b$ and $v$ plays the role of $a$.  By Proposition~\ref{prop:graph_rainbow_sc}, we are guaranteed a rainbow constellation $C$ and a homomorphism $\varphi$ from $C$ to $S$.  Let $W$ be the set of vertices spanned by the centres of stars comprising $C$.  We claim that $\varphi(W) \subseteq \{u,v\}$.  Indeed, every vertex in $W$ is incident to edges of all $r-1 > 1$ colours and every vertex of $S$ other then $u$ and $v$ is incident to edges of solely one colour.  Since $W$ is the vertex set of a copy of $H$ whose every edge is the centre of a star in $C$, the assumption that $H$ is not bipartite implies that $\varphi$ maps both endpoints of at least one such centre to the same vertex (either $u$ or $v$).  Let $S'$ be the star in $C$ whose centere has this property.  The restriction of $\varphi$ to $S'$ (specifically, the first two colours of $S'$) describes two homomorphisms from copies of $H$ minus an edge, one of them mapping both endpoints of the missing edge to $a$, and the other mapping both of them to $b$.

  Suppose now that $H$ is collapsible.  Given a rainbow star $S$, centred at some $uv$, there are edges $a_1b_1, \dotsc, a_{r-1}b_{r-1} \in H$ such that $S$ comprises copies of $H \setminus a_1b_1, \dotsc, H \setminus a_{r-1}b_{r-1}$ glued together so that $a_1, \dotsc, a_{r-1}$ are all mapped to $u$ and $b_1, \dotsc, b_{r-1}$ are all mapped to $v$.  For each $i \in \br{r-1}$, we may find an $f_i \in H$ and a homomorphism $\varphi_i$ from $H \setminus f_i$ to $H \setminus a_ib_i$ that maps both endpoints of $f_i$ to $a_i$. Let $S'$ be a rainbow star comprised of copies of $H \setminus f_1, \dotsc, H \setminus f_{r-1}$ glued together along $f_1, \dotsc, f_{r-1}$.  It is not hard to check that the function $\varphi'$ extending all $\varphi_1, \dotsc, \varphi_{r-1}$ is a homomorphism from $S'$ to $S$ that maps both endpoints of the centre of $S'$ to $u$.  Finally, let $C'$ be a rainbow constellation whose each star is a copy of $S'$.  The function mapping each vertex in each copy of $S'$ according to $\varphi'$ is a homomorphism from $C'$ to $S$ (which maps all vertices spanned by the base of $C'$ to $u$).
\end{proof}

Last, in the case where $r=2$ the notion of rainbow stars and constellations degenerates -- a rainbow star is just a copy of $H$ minus an edge.  As a result, we have fewer restrictions than in the case $r > 2$ and even a weaker version of the collapsibility will be sufficient to imply the star-constellation property.

\begin{dfn}
  \label{dfn:semi-collapsible}
  A graph $H$ is \emph{semi-collapsible} if, for every edge $e \in H$, there is an edge $f \in H$ and a homomorphism $H\setminus f \to H \setminus e$ mapping both endpoints of $f$ to the same vertex.
\end{dfn}

\begin{cor}
  If $r=2$ and $H$ is semi-collapsible, then $\cH_H$ has the star-constellation property.
\end{cor}

\begin{proof}
  Suppose that $H$ is semi-collapsible and let $H \setminus e$, where $e \in H$, be an arbitrary star of $H$.  We let $C$ be a constellation constructed from copies of $H \setminus f$, for some $f \in H$ such that there exists a homomorphism $\varphi$ from $H \setminus f$ to $H \setminus e$ that maps both endpoints of $f$ to the same vertex.  It is not hard to check that the function that maps each vertex in each copy of $H \setminus f$ in $C$ according to $\varphi$ is a homomorphism from $C$ to $H \setminus e$.
\end{proof}

One can use these corollaries to establish the rainbow star-constellation property for certain families of graphs. We will give two examples.  Recall that a graph is called nearly-bipartite if it can be made bipartite by removing one edge; e.g.\ a cycle.  Together with trees, strictly $2$-balanced nearly-bipartite graphs make the largest family of graphs for which sharpness was previously established (albeit, only when $r = 2$). We will show that nearly-bipartite graphs are collapsible and so their respective hypergraphs $\cH_H$ all have the rainbow star-constellation property for every $r$. Secondly, we will note that cliques are also collapsible, thus giving another application to a natural family on the other end of the spectrum. We should also note that some graphs do not have the rainbow star-constellation property. One such example is the Petersen graph (see Appendix~\ref{apx:graph_application}).

\begin{cor}
  If $H$ is nearly-bipartite and with minimum degree at least two, then $\cH_H$ has the rainbow star-constellation property.
\end{cor}

\begin{proof}
  By Corollary~\ref{cor:H-bipartite-rainbow-sc}, we may assume that $H$ is not bipartite.  By Corollary~\ref{cor:H-collapsible-rainbow-sc}, it suffices to show that $H$ is collapsible.  Let $e$ be an edge of $H$ and let $u$ be an endpoint of $e$.  Since $\delta(H) \ge 2$, there is an edge $e' \neq e$ that also contains $u$.  Now take an edge $f$ such that $H \setminus f$ is bipartite.  Consider some bipartition of $H \setminus f$ and let $U$ be the colour class that contains both endpoints of $f$ (we assumed that $H$ is not bipartite).  We map $H \setminus f$ to $H \setminus e$ by sending all the vertices in $U$ to $u$ and by sending the vertices in $V(H) \setminus U$ to the second endpoint of $e'$.
\end{proof}

\begin{cor}
 If $H$ is a clique, then $\cH_H$ has the rainbow star-constellation property.
\end{cor}

\begin{proof}
  This follows from Corollary~\ref{cor:H-collapsible-rainbow-sc} and the fact that a clique is collapsible.  Indeed, given any edge $e$ of the clique and an endpoint $u$ of $e$, the mapping of $K_n \setminus e$ to itself that maps both endpoints of $e$ to $u$ and fixes every other vertex is a homomorphism.
\end{proof}

\subsubsection{List colouring graphs}

We wish to show that, when $p \le n^{-1/m_2(H)}$, a.a.s.\ every constant-sized subgraph of $G_{n,p}$ is \emph{$2$-choosable with respect to $H$}, that is, its edges can be coloured from arbitrary lists of size two without introducing a monochromatic copy of $H$.  Since a.a.s.\ any constant-sized subgraph $G \subseteq G_{n,p}$ satisfies $m(G)\le m_2(H)$, it suffices to prove the following proposition.

\begin{prop}
  \label{prop:graph_choosability}
  Let $H$ be a graph that is not a forest.  If a graph $G$ satisfies $m(G) \le m_2(H)$, then it is $2$-choosable with respect to $H$.
\end{prop}

Before we prove this, let us state a helpful lemma.  Given a graph $H$ and a set $W \subseteq V(H)$, it will be convenient to denote by $\bar{e}_H(W)$ the number of edges incident with a vertex of $W$, i.e., $\bar{e}_H(W) \coloneqq e_H - e_{H \setminus W}$.

\begin{lemma}[Helpful Lemma]
  Suppose that $H$ is strictly $2$-balanced and suppose that $W \subseteq V(H)$ satisfies $1 \le |W| \le v_H-3$. Then
  \[
    m_2(H)\cdot |W| < \bar{e}_H(W).
  \]
\end{lemma}

\begin{proof}
  Since $H-W$ is a proper subgraph of $H$ with at least three vertices and $H$ is strictly $2$-balanced,
  \[
    \frac{e_H - 1 - \bar{e}_H(W)}{v_H-2-|W|} = \frac{e_{H-W}-1}{v_{H-W}-2} < m_2(H) = \frac{e_H-1}{v_H-2},
  \]
  which means that $(e_H-1)\cdot |W|<(v_H-2) \cdot \bar{e}_H(W)$ and the result follows. 
\end{proof}

We now turn to proving the proposition. We will have separate arguments for three small graphs---specifically $K_3$, $K_4$, and $C_4$---and then a general argument for every other graph.  We will begin with the general argument and subsequently supply the remaining cases.

\begin{proof}[Proof of Proposition~\ref{prop:graph_choosability} (Part I: The general argument)]
  Without loss of generality, we may assume that $H$ is strictly $2$-balanced.  Indeed, otherwise we replace $H$ with one of its minimal subgraphs $H'$ satisfying $m_2(H') = m_2(H) > 1$.  Since the only strictly $2$-balanced graphs with fewer than five vertices (and a cycle) are $K_3$, $K_4$, and $C_4$, which will require a separate argument, we may further assume that $H$ has at least five vertices.

  Suppose that the statement is false and let $G$ be a minimal counterexample. Write $m_2(H) = k+\eps$, where $k$ is an integer and $\eps \in [0,1)$.  Since $H$ is not a forest, we must have $k \ge 1$.

  The idea of the proof is to use the upper bound on the density of $G$ to locate a sparse subgraph $S \subseteq G$; this will be achieved using a discharging argument.  By the minimality of $G$, we are able to colour $G \setminus S$ without monochromatic copies of $H$ and because $S$ is sufficiently sparse, we will be able to extend each such colouring to all of $G$.

  \begin{claim}
    \label{claim:discharging}
    One of the following holds:
    \begin{enumerate}[label=(\arabic*)]
    \item
      \label{item:discharging-1}
      $G$ has a vertex of degree at most $2k$,
    \item
      \label{item:discharging-2}
      $G$ has a vertex of degree $2k+1$ with a neighbour of degree at most $2k+2$ and $\eps \ge 1/2$, or
    \item
      \label{item:discharging-3}
      $G$ has a vertex of degree $2k+3$ with two neighbours of degree $2k+1$ and $\eps \ge 7/8$.
    \end{enumerate}
\end{claim}
\begin{proof}
  Since $\delta(G) \le 2m(G) \le 2m_2(H) = 2k+2\eps < 2k+2$, we have $\delta(G) \le 2k+1$ and, if $\eps < 1/2$, then $\delta(G) \le 2k$.  We may thus assume that $\delta(G) = 2k+1$ and $\eps \ge 1/2$, since otherwise~\ref{item:discharging-1} holds.  We may further assume that all neighbours of every vertex of degree $2k+1$ have degrees at least $2k+3$, since otherwise~\ref{item:discharging-2} holds.

  Assign to each $v \in V(G)$ a charge of $\deg(v) - 2(k+\eps)$. Note that the average charge is at most $2m(G)-2(k+\eps) \le 0$.  We define the following discharging rule: every vertex of degree $2k+1$ takes a charge of $\frac{2\eps-1}{2k+1}$ from each of its neighbours.  By our assumption, no vertex of degree $2k+1$ or $2k+2$ sends charge to any of its neighbours.  In particular, the final charge of a vertex of degree $2k+1$ is
  \[
    2k+1-2(k+\eps)+(2k+1) \cdot \frac{2\eps-1}{2k+1} = 0
  \]
  and the final charge of a vertex of degree $2k+2$ is $2k+2 - 2(k+\eps) > 0$.

  Since the total charge remains unchanged, the final charge of some vertex of degree at least $2k+3$ must be non-positive.  Let $v$ be one such vertex.  Suppose that $\deg(v) = 2k+t$, where $t \ge 3$, and that $v$ has $x$ neighbours with degree $2k+1$.  Since the final charge of $v$ is
  \[
    2k+t - 2(k+\eps) - x \cdot \frac{2\eps - 1}{2k+1} \le 0,
  \]
  we have
  \[
    (t-2)(2k+1) < \frac{t - 2\eps}{2\eps - 1} \cdot (2k+1) \le x \le 2k+t,
  \]
  which implies that $t < 3 + 1/k \le 4$.  Therefore, $t=3$ and $x > 2k+1 \ge 3$, which means that some vertex of degree $2k+3$ has more than three neighbours of degree $2k+1$.  Moreover, we also have $\frac{3-2\eps}{2\eps-1} \cdot (2k+1) \le 2k+3$, which implies that $\eps \ge \frac{4k+3}{4k+4} \ge \frac{7}{8}$.
\end{proof}

We split the argument into three cases, depending on which item in Claim~\ref{claim:discharging} holds.

  \smallskip
  \noindent
  \textit{Case 1. Item~\ref{item:discharging-1} in Claim~\ref{claim:discharging} holds.}
  Since the Helpful Lemma implies that $\deg(a) = \bar{e}_H(\{a\}) > m_2(H) = k+\eps$ for every vertex $a \in V(H)$, we have $\delta(H) \ge k+1$.  Let $v$ be a vertex of smallest degree in $G$ and let $S$ comprise all the edges incident with $v$.  Any colouring of $G \setminus S$ may be extended to $G$ in the following way:  Because all lists have two colours, we may choose colours for the edges of $S$ so that every colour is selected at most $\lceil |S|/2 \rceil = \lceil \delta(G)/2 \rceil = k$ times.  Since $\delta(H) > k$, this means that $v$ cannot belong to a monochromatic copy of $H$.

  \smallskip
  \noindent
  \textit{Case 2. Item~\ref{item:discharging-2} in Claim~\ref{claim:discharging} holds.}
  Let $uv$ be an arbitrary edge of $G$ satisfying $\deg(u) = 2k+1$ and $\deg(v)\le 2k+2$.  We claim that any colouring of $G \setminus uv$ can be extended to $G$.  Suppose that it cannot.  This means that $uv$ completes a monochromatic copy of $H$ in both its colour options.  But this means that, for some $ab \in H$,
  \[
    4k+1 \ge \deg(u) + \deg(v) - 2 = \bar{e}_G(\{u,v\})-1 \ge 2 \cdot \big(\bar{e}_H(\{a,b\}) - 1\big).
  \]
  However, the Helpful Lemma implies that $\bar{e}_H(\{a,b\})> 2m_2(H) = 2k+2\eps  \ge 2k+1$, a contradiction.

  \smallskip
  \noindent
  \textit{Case 3. Item~\ref{item:discharging-3} in Claim~\ref{claim:discharging} holds.}
  Let $u$, $v_1$, and $v_2$ be distinct vertices of $G$ satisfying $uv_1, uv_2 \in G$, $\deg(u) = 2k+3$, and $\deg(v_1) = \deg(v_2) = 2k+1$.  We may assume that $v_1v_2 \notin G$, since otherwise~\ref{item:discharging-2} in Claim~\ref{claim:discharging} holds.  We claim that any proper colouring of $G' \coloneqq G \setminus \{uv_1, uv_2\}$ can be extended to $G$.  To this end, let $G_1 \coloneqq G' \cup \{uv_1\}$ and $G_2 \coloneqq G' \cup \{uv_2\}$.  Since $\deg_{G_i}(u) = 2k+2$ and $\deg_{G_i}(v_i) = 2k+1$ for both $i \in \{1, 2\}$, the argument presented in Case~2 shows that any colouring of $G'$ can be (separately) extended to both $G_1$ and $G_2$.  If some extensions of the colouring of $G'$ to $G_1$ and $G_2$ assign different colours to $uv_1$ and $uv_2$, then their common extension is an $H$-free colouring of $G$.  Therefore, we may assume that there is a colour $j$ such that, for both $i \in \{1, 2\}$, every extension of the colouring of $G'$ to $G_i$ assigns this colour $j$ to $uv_i$.  This means that $uv_i$ completes a copy of $H$ in $G_i$ whose all remaining edges are assigned a colour other than $j$ (the second colour from the list of $uv_i$) in the colouring of $G'$.  In particular, there are vertices $a_1$, $a_2$, and $b_2$ of $H$ such that $a_2b_2 \in H$ and, letting $m$ denote the number of edges of $G'$ incident to $\{u, v_1, v_2\}$ that are not coloured $j$, 
  \[
    m \ge \bar{e}_H(\{a_1\}) - 1 + \bar{e}_H(\{a_2,b_2\}) - 1
  \]
  As before, the Helpful Lemma implies that $\bar{e}_H(\{a_1\}) > k$ and $\bar{e}_H(\{a_2, b_2\}) > 2k+2\eps > 2k+1$, which means that $m \ge 3k+1$.  It follows that the number $m'$ of edges of $G'$ incident to $\{u, v_1, v_2\}$ that are coloured $j$ satisfies
  \[
    m' = \bar{e}_{G'}(\{u,v_1,v_2\}) -m = \deg(u) + \deg(v_1) + \deg(v_2) - 4 - m \le 3k.
  \]
  Assign the colour $j$ to both $uv_1$ and $uv_2$.  If this is not a proper colouring of $G$, then there must be a copy of $H$ in colour $j$ that contains both $uv_1$ and $uv_2$.  This means that there are vertices $a$, $b$, and $c$ of $H$ such that $ab, bc \in H$, $ac \notin H$, and
  \[
    \bar{e}_H(\{a, b, c\}) - 2 \le m' \le 3k.
  \]
  However, since $\eps \ge 1 - 1/8$, we must have $v_H \ge 8+2$ and the Helpful Lemma implies that $\bar{e}_H(\{a,b,c\}) > 3m_2(H) = 3k+3\eps > 3k+2$, a contradiction.
\end{proof}

\begin{proof}[Proof of Proposition~\ref{prop:graph_choosability} (Part II: the small graphs)]

All that remains is to prove the statement for $C_4$, $K_4$, and $K_3$. Assume that $G$ is a minimal counterexample.

\begin{itemize}
\item[($C_4$)]
  Since $m(G)\le m_2(C_4) = 3/2$, then either $\delta(G) < 3$ or $G$ is $3$-regular. However, since $\delta(C_4)=2$, the former leads to a contradiction: we would take a vertex $v$ of $G$ of degree at most two, find a $C_4$-free colouring for $G \setminus \{v\}$, and extend this colouring to $G$ by choosing different colours for the two edges incident with $v$.  The same argument works in the case where $\deg(v)=3$ and the lists of colours for the three edges incident to $v$ are not identical (we can still choose a different colour for each of these edges).  Thus, we may assume that $G$ is $3$-regular and all the colour lists are identical (the minimality of $G$ implies that it is connected) -- they all contain the colours red and blue.

  Let $v$ be an arbitrary vertex of $G$ and let $u_1$, $u_2$, and $u_3$ be its neighbours.  Fix a colouring of $G - v$.  Since the colouring cannot be extended to $G$, every pair of vertices among $\{u_1, u_2, u_3\}$ is connected in $G  - v$ by both a red and a blue path of length two.  (Indeed, if $u_1$ and $u_2$ were not connected by a red path, say, then colouring $vu_1$ and $vu_2$ red and $vu_3$ blue would yield a $C_4$-free colouring.)  However, since $G$ is $3$-regular, each $u_i$ is incident to at most one red and at most one blue edge of $G'$.  This means that there are vertices $v_r, v_b \neq v$ such that $v_r$ is connected to all $u_i$ in red and $v_b$ is connected to all $u_i$ in blue.  Since $G$ is $3$-regular, $\{v,v_r,v_b,u_1,u_2,u_3\}$ is a connected component of $G$; by minimality, $G = G[\{v, v_r, v_b, u_1, u_2, u_2\}] \cong K_{3,3}$.  However, $K_{3,3}$ has many $2$-edge-colourings without a monochromatic $C_4$ (e.g., $K_{3,3}$ can be decomposed into $C_6$ and $3K_2$).

\item[($K_4$)]
  Since $m(G) \le m_2(K_4)=5/2$,  then either $\delta(G) \le 4$ or $G$ is $5$-regular.  Since $\delta(K_4) = 3$, if $G$ has a vertex $v$ of degree at most four or a vertex of degree five whose incident edges have nonidentical colour lists, we may extend any $K_4$-free colouring of $G - v$ to $G$ by colouring edges incident to $v$ in such a way that every colour is used at most twice.  Thus, we may assume that $G$ is $5$-regular and all the colour lists contain the colours red and blue.

  Let $v$ be an arbitrary vertex and fix a colouring of $G - v$.  Since $G - v$ has no red $K_4$'s, there must be a $3$-element subset $T \subseteq N(v)$ that does not induce a red triangle.  Colouring the three edges connecting $v$ to $T$ red and the remaining two edges incident to $v$ blue yields a $K_4$-free colouring of $G$.

\item[($K_3$)]
  Since $m(G) \le m_2(K_3) = 2$, then either $\delta(G) \le 3$ or $G$ is $4$-regular.  Suppose that, for some $v \in V(G)$, there was an orientation of the edges of $G[N(v)]$ in which every vertex had out-degree at most one.  We could then extend every $K_3$-free colouring of $G - v$ to $G$ as follows: For every $u \in N(v)$, the edge $uv$ gets a colour that is different from the colour of the out-edge from $u$.  (Since each triangle involving $v$ contains an edge of $G[N(v)]$, this colouring is $K_3$-free.)  As every graph with at most four edges has such an orientation, we may assume that $e(N(v)) \ge 5$ for every $v \in V(G)$; in particular, $\delta(G) \ge 4$, so $G$ is $4$-regular.

  We claim that $G = K_5$.  If $e(N(v)) > 5$ for some $v \in v(G)$, then $G = K_5$, since $G$ is $4$-regular and connected.  We may thus further assume that $e(N(v)) = 5$ for every $v$.  Pick some $v$ and denote $N(v) = \{u_1, u_2, u_3, u_4\}$ so that $u_1u_3 \notin G[N(v)]$. Since $G$ is $4$-regular, there must be a $w \in V(G) \setminus (\{v\} \cup N(v))$ such that $N(u_1) = \{v, w, u_2, u_4\}$. Moreover, $w$ is not adjacent to either of $v$,  $u_2$, and $u_4$, as they all have $4$ neighbours in $\{v\} \cup N(v)$, and thus $e(N(u_1)) \le 3$, a contradiction.

  Finally, we show how that $K_5$ is $2$-choosable with respect to $K_3$.  If some colour, say red, contains a $5$-cycle, then we may colour this $5$-cycle red and the complementary $5$-cycle not red.  If some colour class, say red, contains an edge, say $e$, not in a triangle, then we may colour $K_5 \setminus e$ without monochromatic triangles (this is possible as $K_5$ is minimally non-$2$-choosable) and colour $e$ red.  If none of the above is true, then each colour induces one of the following graphs: $K_3$, $K_4$, $K_4^-$, $K_5 \setminus K_3$, or two triangles sharing a vertex.  If some colour, say red, induces $K_5 \setminus K_3$, then we colour $K_{2,3}$ with red, the remaining edge of $K_5 \setminus K_3$ with not red and the edges of the $K_3$ in the complement with two different colours other than red.  If one of the colours, say red, induces $K_4$ or $K_4^-$, then colour a $C_4$ with red and its diagonal with a colour other than red.  Each of the remaining, uncoloured four edges can close at most one monochromatic triangle,  as red is not available anywhere outside of the $K_4$ we have already coloured; thus we may colour them one-by-one.  This leaves the case where every colour class is either $K_3$ or two triangles sharing a vertex.  But this is impossible, since $3$ does not divide $2e(G) = 20$. \qedhere
\end{itemize}
\end{proof}

\subsection{Arithmetic progressions}
\label{sec:arithm-progr}

In this section, we prove Theorem~\ref{thm:main-vdW}, which asserts that $\vdW(k,r)$ has a sharp threshold in $(\ZZ_N)_p$.  Even though the most natural setting for van der Waerden's theorem is the interval $\br{N}$, the corresponding hypergraph lacks the required symmetry.  As a result, we have to work in $\ZZ_N$ instead, where transitivity is guaranteed by translations.  (We note that this was also the case in~\cite{FriHanPerSch16}, which proved sharpness of the threshold for van der Waerden's theorem in two colours.)  We will consider the $k$-uniform hypergraph $\cH_{\kAP}$ of proper $k$-term arithmetic progressions in $\ZZ_N$.  It is easily verified that $\cH_{\kAP}$ has $\Theta(N^2)$ edges, and thus $p_{\cH_{\kAP}} = \Theta(N^{-1/(k-1)})$, and that it is non-clustered and symmetric.  The threshold for van der Waerden's theorem in $\br{N}_p$ is known to lie at $N^{-1/(k-1)}$; this was first proved by Graham, R\"odl, and Ruci\'nski in~\cite{GraRodRuc96}, for $3$-APs, and later extended by R\"odl and Ruci\'nski~\cite{RodRuc95,RodRuc97} to general $k$-APs.  The $1$-statement implicit in~\ref{item:ass-weak-threshold} in our setting is an immediate consequence of this, as every $k$-AP in $\br{N}$ is also a $k$-AP in $\ZZ_N$; independently, it can also be recovered using Proposition~\ref{prop:robust-non-col-1-statement} --  robust non-colourability of $\cH_{\kAP}$ follows from van der Waerden's theorem and Varnavides's averaging argument.  The $0$-statement in~\ref{item:ass-weak-threshold} requires extra consideration since $\Z_N$ contains more progressions than $\br{N}$.  The arguments of~\cite{RodRuc95, RodRuc97} can be echoed here, but we instead prove the stronger Theorem~\ref{thm:list-vdW}, which asserts that, when $p < cN^{-1/(k-1)}$ for a sufficiently small constant $c$, then  $(\ZZ_N)_p$ is a.a.s.\ list-$k$-van der Waerden.

This will leave us with verifying the last two assumptions of Theorem~\ref{thm:main}.  The choosability assumption~\ref{item:ass-choosability} is a simple corollary of Theorem~\ref{thm:list-vdW}, see Corollary~\ref{cor:0statement_constant} in Section~\ref{sec:choosability-almost-linear}.  Finally, the rainbow star-constellation property follows readily from Szemer\'edi's theorem and another application of Varnavides's averaging argument.  Indeed, suppose that a partial colouring of $\ZZ_N$ contains $\Omega(N^r)$ rainbow stars.  Then, for some $i \in \br{r}$,  the set $A_i$ of all elements of $\ZZ_N$ that are the centres of $\Omega(N^{r-1})$ many $i$-rainbow stars has $\Omega(N)$ elements.  Therefore, $A_i$ contains $\Omega(N^2)$ many $k$-term APs and each such $k$-AP is the base of $\Omega(N^{(r-1)k})$ rainbow constellations.

\subsection{Schur's theorem}
\label{sec:schurs-theorem}

In this section, we prove Theorem~\ref{thm:main-Schur}, which asserts that the property of being $r$-Schur has a sharp threshold in $(\ZZ_N)_p$.  As in the case of van der Waerden's theorem, we have to work in $\ZZ_N$, as the interval $\br{N}$ lacks the symmetries required by Theorem~\ref{thm:main}.  Define the Schur hypergraph $\cH$ on $\ZZ_N$ whose edges are \emph{Schur triples}, i.e., triples of distinct $x,y,z \in \ZZ_N$ such that $x+y=z$.  This hypergraph has $\Theta(N^2)$ edges, and thus $\pH = \Theta(N^{-1/2})$, and it is easily seen to be non-clustered.  However, it is not symmetric.  Indeed, every automorphism of $\cH$ is of the form $x \mapsto c \cdot x$ for some invertible $c \in \ZZ_N$.  Since $0$ is always mapped to itself, the automorphism group is non-transitive.  However, when $N$ is a prime number, this is the only obstruction.  In other words, the hypergraph $\cH' \coloneqq \cH - \{0\}$ is symmetric.  This is very fortunate because a.a.s.\ $0$ is not in $(\ZZ_N)_p$ when $p = o(1)$, and as a result $\cH_p$ acts very much like $\cH'_p$.  Put succinctly, we have the following observation: 

\begin{obs}
Given $p = o(1)$ and a property $\cP$, the probability of $\cH_p \in \cP$ tends to $0$ if and only if the probability that $\cH'_p \in \cP$ tends to $0$.
\end{obs}

We aim to use this observation to prove a sharp threshold for $\cH$ by proving it first for the symmetric $\cH'$.  To do so, we will prove that $\cH$, on top of being non-clustered, satisfies assumptions~\ref{item:ass-weak-threshold}--\ref{item:ass-star-constellation}.  These properties will then transfer to $\cH'$ and, together with symmetry, we will be able to argue that it has a sharp threshold.

Regarding~\ref{item:ass-weak-threshold}, Graham, R\"odl, and Ruci\'nski~\cite{GraRodRuc96} located the threshold for Schur's theorem in $\br{N}_p$ at $N^{-1/2}$.  Similarly as in the context of van der Waerden's theorem, since $\ZZ_N$ has more Schur triples than $\br{N}$, this result implies the $1$-statement in~\ref{item:ass-weak-threshold}, but the $0$-statement needs extra work.  Instead of adjusting the arguments of~\cite{GraRodRuc96}, we prove the stronger Theorem~\ref{thm:list-Schur}, which asserts that $N^{-1/2}$ is a threshold for the `stronger' property of being list-Schur.  This solution has the advantage that the choosability assumption~\ref{item:ass-choosability} is a simple corollary of this stronger theorem, see Corollary~\ref{cor:0statement_constant} in Section~\ref{sec:choosability-almost-linear}.  This leaves us with establishing the rainbow star-constellation property for the Schur hypergraph, which we do in the remainder of this section.

\subsubsection{Rainbow star-constellation property for Schur triples}

In this short subsection, we verify that the Schur hypergraph has the rainbow star-constellation property.  We write $X$ instead of $\ZZ_N$, noting that our arguments remain valid if we replace it with an arbitrary Abelian group of order $N$.

It will be convenient to define structures that offer a slight relaxation of the notions of stars and constellations in that their elements may not be distinct:  Given a sequence $Y$ of $t$ subsets of $X$ define a \emph{$Y$-prestar} to be a pair $x, y$ of sequences of $t$ elements of $X$ together with an element $a \in X$, such that $x_i, y_i \in Y_i$ and $x_i$, $y_i$, and $z$ form a sum for every $i \in \br{t}$.  The element $a$ is called the \emph{centre} of the prestar and each pair $x_i, y_i$ is called a \emph{ray}.  Similarly, define a \emph{$Y$-preconstellation} to be a triplet of $Y$-prestars whose centres form a sum.  Since fixing any two coordinates of a $Y$-preconstellation leaves at most $O(N^{3t})$ options for completing it, there are at most $O(N^{3t+1})$ many $Y$-preconstellations that have a repeating coordinate.  In particular, the following statement (with $Y$ being the sequence of some $r-1$ colour classes) implies the rainbow star-constellation property.

\begin{prop}
  \label{prop:rainbow-sc-Schur}
  For every $\beta > 0$, there exists a $\gamma > 0$ such that the following holds.  For any sequence $Y$ of $t$ subsets of $X$, if there are $\beta N^{t+1}$ many $Y$-prestars, then there are $\gamma N^{3t+2}$ many $Y$-preconstellations.
\end{prop}

\begin{proof}
  We may always order the elements of the $i^{\text{th}}$ ray of any prestar as $x_i,y_i$ so that one of $x_i \pm y_i$ equals the centre of the prestar.  Define the \emph{sign pattern} of the prestar to be the sequence of signs that $y_i$ appeared with for each $i \in \br{t}$.  If there are $\beta N^{t+1}$ many $Y$-prestars, the pigeonhole principle dictates that there are at least $2^{-t} \beta N^{t+1}$ many $Y$-prestars sharing a specific sign pattern.  Call these prestars the \emph{popular} $Y$-prestars.  Note that every popular $Y$-prestar $x, y$ with centre $a$ is uniquely determined by $x$ and $a$ since knowing the sign pattern allows us to compute $y$.

  Let $f(x, a)$ be the indicator of whether $x \in X^t$ and $a \in X$ determine a popular $Y$-prestar and note that $\sum_{x,a} f(x,a) \ge 2^{-t} \beta N^{t+1}$.  Next, for an $x \in X^t$ and $a,b \in X$, define $f_1(x,a,b) \coloneqq f(x,a) \cdot f(x,b)$.  Using Jensen's inequality we learn that 
  \[
    \frac{1}{N^t} \sum_{x,a,b} f_1(x,a,b) = \frac{1}{N^t}\sum_{x} \left( \sum_{a} f(x,a) \right)^2 
    \ge \left(\frac{1}{N^t} \sum_{x,a} f(x,a) \right)^2 \ge \left(2^{-t}\beta N \right)^2,
  \]
  which implies that $\sum_{x,a,b} f_1(x,a,b) \ge 2^{-2t}\beta^2 N^{t+2}$.

  Now, for $x, x', x'' \in X^t$ and $a, b \in X$, write
  \[
    f_2(x,x',x'',a,b) \coloneqq f_1(x,a,b) \cdot f_1(x',a,b) \cdot f_1(x'',a,b).
  \]
  Using Jensen's inequality again, we get the bound
  \[
    \begin{split}
      \frac{1}{N^2} \sum_{x,x',x'',a,b} f_2(x,x',x'',a,b) & = \frac{1}{N^2} \sum_{a,b} \left(\sum_x f_1(x,a,b) \right)^3 \\
      & \ge \left(\frac{1}{N^2} \sum_{a,b,x} f_1(x,a,b) \right)^3 \ge \left(2^{-2t} \beta^2 N^t \right)^3,
    \end{split}
  \]
  which implies that $\sum_{x,x',x'',a,b} f_2(x,x',x'',a,b) \ge 2^{-6t} \beta^6 N^{3t+2}$.
  
  Now, suppose that $f_2(x,x',x'',a,b) = 1$ for some $x, x', x'' \in X^t$ and $a, b \in X$.  Then $ f(x',a) = f(x'',b) = f(x, a) = f(x,b) = 1$. By definition, this ensures the existence of the following four popular $Y$-prestars: $x, y$ and $x', y'$ centred at $a$ and $x, z$ and $x'', z''$ centred at $b$ (the sequences $y, y', z, z'' \in X^t$ are uniquely determined).  In particular, for every $i \in \br{t}$, either $x_i + y_i =a$ and $x_i + z_i = b$ or $x_i - y_i = a$ and $x_i - z_i = b$; this means that either $y_i - z_i = a-b$ or $z_i - y_i = a -b$. Therefore, $y,z$ is a $Y$-prestar centered at $a-b$ and, consequently, $(x',y',a), (x'',z'',b), (y,z, a-b)$ is a $Y$-preconstellation.

  Note that the function mapping $x,x',x'',a,b$ as above to the $Y$-preconstellation comprising $(x',y',a)$, $(x'',z'',b)$, and $(y,z, a-b)$ is injective.  Indeed, one can reconstruct $x$ given $y$ and $a$, as we know the sign-pattern of the popular prestar $x,y$ centered at $a$.  As a result, we learn that there are at least $2^{-6t}\beta^6 N^{3t+2}$ many $Y$-preconstellations. 
\end{proof}

\subsection{Choosability of almost-linear hypergraphs}
\label{sec:choosability-almost-linear}

In this section, we complete the derivations of Theorems~\ref{thm:main-vdW} and~\ref{thm:main-Schur} by proving Theorem~\ref{thm:choosability-linear-hypergraphs}, which immediately implies Theorems~\ref{thm:list-vdW} and~\ref{thm:list-Schur}.

\choosability*

Even though this is not the exact statement we need for verifying assumption~\ref{item:ass-choosability} in the context of van der Waerden's and Schur's theorem, we will be able to obtain the latter as a straightforward corollary.

\begin{cor}
  \label{cor:0statement_constant}
  Given $s \ge 3$ and a non-clustered $s$-uniform hypergraph $\cH$ with $\Delta_2(\cH) = O(1)$, a sequence $p = \Theta(\pH)$, and any constant $K$, the probability that $\cH_p$ contains a non-$2$-choosable set of size $K$ tends to zero.
\end{cor}

\begin{proof}
  Let $\cF$ be the family of all non-$2$-choosable sets with at most $K$ elements and let $\mu(p) \coloneqq \Pr(\exists B \in \cF \colon B \subseteq V(\cH)_p)$.  Theorem~\ref{thm:choosability-linear-hypergraphs} tells us that there is some $p = \Theta(\pH)$ for which $\mu(p) = o(1)$.  By monotonicity, $\mu(p') = o(1)$ also for any $p' < p$.  Now, for any constant $C > 1$ we can use Lemma~\ref{lemma:local-coarseness} to bound $\mu(p) \ge C^{-K} \mu(Cp)$, so $\mu(Cp) = o(1)$ as well.
\end{proof}

We will describe a process for revealing connected subsets of the vertices of an $s$-uniform hypergraph $\cH$ by layers.  Let $N \coloneqq v(\cH)$ and identify the vertices of $\cH$ with the set $\br{N}$.  This labeling induces a total ordering of the vertices and also of the edges of $\cH$, via the lexicographic ordering. Given a subset $S$ of the vertices of $\cH$ which induces a connected subhypergraph of $\cH$ we define the following procedure.

\medskip
\begin{algorithm}[H]
  Let $v_1$ be the smallest vertex in $S$ and write $S' = \{v_1\}$.
  
  \For{$i = 1, 2, \dots$}{
    \tcp{Degenerate steps}
    \While{ there exists an edge $e \in \cH[S]$ such that $2 \le |e \cap S'| \le s-1$} { Pick the smallest such edge $e$ and let $S' \gets S' \cup e$.}
    
    \tcp{Finish the layer}
    Write $S_i = S'$.
    
    \eIf{$S_i = S$}{
      Terminate.
    }{
      \tcp{Start a new layer}
      Take the smallest vertex $v_i \in S_i$ which has an edge $e \in \cH[S]$ that intersects $S_i$ exactly in $v_i$.
      
      Pick the smallest such edge $e_i$ and let $S' \gets S_i \cup e_i$.
    }
  }
\end{algorithm}
\medskip

Observe that connectivity of $S$ ensures that the process will indeed terminate. We say that a vertex in $S$ is \emph{degenerate} if it was added in a degenerate step. Let $d(S)$ denote the number $d$ such that $S = S_d$.

We will apply this procedure to sets $S$ that are minimally non-$2$-choosable.  (Note that minimality implies that these sets are connected.)  The proof of Theorem~\ref{thm:choosability-linear-hypergraphs} has two steps.  First, using a deterministic argument, we will show that each minimally non-$2$-choosable set must either have at least $s-1$ degenerate elements or contain a structure which we will call a clot.  Second, we will see that, when $\cH$ satisfies the assumptions of the theorem, connected sets containing either $s-1$ degenerate elements or a clot are too rare to appear in $V(\cH)_p$.

\begin{dfn}
  A set $A$ of $2s-3$ vertices in an $s$-uniform hypergraph is a \emph{nucleus} if, for every $(s-1)$-element subset $A' \subseteq A$, there are two distinct vertices $v_1(A'), v_2(A') \notin A$ such that $A' \cup \{v_i(A')\}$ is an edge for both $i \in \{1,2\}$.  A \emph{clot} around a nucleus $A$ is the union of $A$ together with all the vertices $v_i(A')$.
\end{dfn}

\begin{lemma}
  \label{lemma:non-2-choosable-obstructions}
  If $s \ge 3$, then every minimally non-$2$-choosable set of vertices of an $s$-uniform hypergraph contains either at least $s-1$ degenerate elements or a clot.
\end{lemma}
\begin{proof}
  Let $S$ be a minimally non-$2$-choosable set of vertices of an $s$-uniform hypergraph $\cH$.  We run the process of revealing $S$ described above and let $d = d(S)$.  Since $S_1 = \{v_1\}$ is $2$-choosable, $d$ must be greater than $1$.  We may assume that $S$ contains at most $s-2$ degenerate elements, as otherwise there is nothing left to prove.  We will show that $A \coloneqq S \setminus S_{d-1}$ is the nucleus of a clot.

  Note that $A$ contains at most $2s-3$ elements: the $s-1$ elements of $e_{d-1} \setminus S_{d-1}$ plus at most $s-2$ additional degenerate elements included in subsequent degenerate steps.  Further, note that every edge of $S$ which is not contained in $S_{d-1}$ must contain at least $s-1$ elements from $A$, since otherwise it would intersect $S_{d-1}$ in at least two elements and would have therefore been absorbed into $S_{d-1}$ in a degenerate step.  For that reason, every colouring of $S_{d-1}$ may be extended to $S$ unless we are forced to colour some $s-1$ of $A$ with the same colour.  However, this may only happen if $|A| = 2s-3$ and all the colour lists of the elements of $A$ are identical, say they comprise the colours red and blue.  Moreover, every set of $s-1$ elements in $A$ must be contained in at least two edges: one with an element of $S_{d-1}$ already coloured red, and a second with an element of $S_{d-1}$ already coloured blue.  Therefore, $A$ must be the nucleus of a clot.
\end{proof}

\begin{lemma}
  \label{lemma:no-clots}
  Suppose that $s \ge 3$ and $\cH$ is a sequence of $s$-uniform hypergraphs with $\Delta_2(\cH) = O(1)$.  If $p = O\big(v(\cH)^{-1/(s-1)}\big)$, then a.a.s.\ $\cH_p$ does not contain a clot.
\end{lemma}
\begin{proof}
  As before, we write $N \coloneqq v(\cH)$ and identify $V(\cH)$ with $\br{N}$.  We give two different arguments, depending on whether or not $s > 3$.

  We first show that, when $s > 3$, there are only $O(N^2)$ clots.  This will be sufficient as every clot contains at least $2s-1$ vertices (the nucleus and at least two additional vertices) and thus the expected number of clots in $\cH_p$ is $O(N^2p^{2s-1}) = O(p) = o(1)$.  There are at most $O(N^2)$ ways to choose the two smallest elements of the nucleus $A$ of a clot.  Since $s-1 > 2$, every other element of $A$ belongs to an edge containing these two elements and there are only $\Delta_2(\cH) = O(1)$ such edges.  Similarly, every element of the clot that is not in $A$ belongs to one of the at most $\binom{2s-3}{s-1} \cdot \Delta_{s-1}(\cH) = O(1)$ edges that intersect $A$ in $s-1$ elements.

  The second case is where $s=3$, implying that every nucleus has three vertices and every clot has at least five vertices.  In particular, there are at most $N^3$ ways to choose the nucleus $A$ of a clot and every element of the clot that is not in $A$ belongs to one of the at most $3 \cdot \Delta_2(\cH) = O(1)$ edges that intersect $A$ in two vertices.  Consequently, the expected number of clots with at least seven vertices is $O(N^3 p^7) = O(p) = o(1)$.  If a clot has fewer than seven vertices, then, by the pigeonhole principle, there must be some vertex $v$ not in its nucleus that forms edges with two different pairs of vertices from the nucleus.  This implies that there are only $O(N^2)$ such clots:  We may pick one such $v$ and an element of the nucleus with at most $N^2$ options.  The remaining vertices of the clot can be added one-by-one in such a way that the added element forms an edge with two previously added vertices.  This means that the expected number of such clots is $O(N^2 p^5) = O(p) = o(1)$.
\end{proof}

\begin{lemma}
  \label{lemma:few-degenerate-elements}
  Suppose that $s \ge 3$ and $\cH$ is a sequence of $s$-uniform hypergraphs with $\Delta_2(\cH) = O(1)$.  If $p \le cv(\cH)^{-1/(s-1)}$ for a sufficiently small positive constant $c$, then a.a.s.\ any connected $S \subseteq V(\cH)_p$ contains at most $s-2$ degenerate elements.
\end{lemma}
\begin{proof}
  As before, we write $N \coloneqq v(\cH)$ and identify $V(\cH)$ with $\br{N}$.  If a connected set $S$ has at least $s-1$ degenerate elements, then it must contain a connected subset with at least $s-1$ degenerate elements for which our procedure executed at most $s-1$ degenerate steps.  Indeed, one obtains such subset by simply halting the procedure after $s-1$ degenerate elements are revealed, while noting that every degenerate step must introduce at least one new degenerate element.  We may thus restrict our attention to connected sets $S$ with this additional property.  We further claim that $d(S) \le \frac{|S|-1}{s-1}$ for every $S$ with at least $s-1$ degenerate elements.  To see this, note that the number of degenerate elements in $S$ is $|S| - 1 - (d(S)-1)(s-1)$, as precisely $(d(S)-1)(s-1)$ vertices are added in non-degenerate steps.  Summarising, it is enough to show that the expected number of connected sets $S$ with $d(S) \le \frac{|S|-1}{s-1}$ and at most $s-1$ degenerate steps that appear in $V(\cH)_p$ tends to zero.

  Let $X_k$ be the number of such sets that have exactly $k$ elements.  We can bound $X_k$ using the following logic:  First, we fix integers $d \le \frac{k-1}{s-1}$ and $d' \le s-1$ and bound the number of $k$-element sets $S$ with $d(S) = d$ for which the procedure runs $d'$ degenerate steps.  Given $d$ and $d'$, we decide, for each of the $d+d'-1$ steps of the procedure revealing $S$ whether it is degenerate or not; there are at most $2^{d+d'-1} \le 2^k$ options.  Second, we pick $v_1 \in \br{N}$.  Third, for every degenerate step, we choose some two vertices of $S'$ that witnessed $|e \cap S'| \ge 2$ and choose the edge $e$; there are at most $k^2 \cdot \Delta_2(\cH)$ options.  Fourth, we choose the number of times each element of $S$ plays the role of $v_i$ to start a new layer.  We may represent this as a multiset $M$ of $\br{k}$ with $d-1$ elements, where $j \in \br{k}$ corresponds to the $j$th vertex in the order of arrival to $S$;  thus, there are at most $\binom{k + d-1 -1}{d-1} \le 2^{2k}$ options.  Crucially, note that, for each $i \in \br{d-1}$, the identity of the vertex $v_i$ is determined by $M$ and $S_i$.   Finally, we pick, for every $i \in \br{d-1}$, the edge $e_i$ that intersects $S_i$ in $v_i$; there are at most $N \cdot \Delta_2(\cH)$ options.  Summarising,
  \[
    X_k \le \sum_{d \le \frac{k-1}{s-1}} \sum_{d' \le s-1}  2^k \cdot N \cdot \big(k^2 \cdot \Delta_2(\cH)\big)^{d'} \cdot 2^{2k} \cdot \big(N \cdot \Delta_2(\cH)\big)^{d-1} \le 2^{Ck} \cdot N^{\frac{k-1}{s-1}},
  \]
  where $C$ is a constant that depends only on $s$ and the constant implicit in the upper bound $\Delta_2(\cH) = O(1)$.

  Suppose now that $p \le 2^{-C} N^{-1/(s-1)}$ and let $\cB$ denote the event that $\br{N}_p$ contains a connected subset with at least $s-1$ degenerate elements.  We have
  \[
    \Pr(\cB) \le \sum_{k \ge 1} X_k \cdot p^k \le \sum_{k \ge 1} 2^{Ck} \cdot N^{\frac{k-1}{s-1}} \cdot 2^{-Ck} N^{-\frac{k}{s-1}} = N^{-\frac{1}{s-1}} = o(1),
  \]
  as claimed.
\end{proof}

\bibliographystyle{amsplain}
\bibliography{sharp-th}
  
\newpage
\appendix
\section{Weak threshold -- the 1-statement}
\label{apx:1-statement}

\begin{prop}
  \label{prop:apx-1state}
  Suppose that $\cH$ is weakly non-clustered and robustly non-$r$-colourable. Then there are constants $C, \alpha$ such that, for any $p > C \cdot \pH$, we have
  \[
    \Pr \left( \cH_p \text{ is $r$-colourable} \right) \le \exp(- \alpha \cdot p \cdot v(\cH)).
  \] 
\end{prop}

In particular, if $v(\cH) \cdot \pH$ tends to infinity, then the proposition tells us that the probability of $\cH_p$ being $r$-colourable tends to zero. Before turning to the proof, we pause to make a couple of remarks concerning this point.
 
\begin{remark}
  First, we claim that when $\cH$ satisfies the stronger requirement of being mildly non-clustered, then $v(\cH) \cdot \pH$ does tend to infinity. Indeed, if there is an index $i \in \br{2,s-1}$ such that $\Delta_i(\cH) \ll \pH^{i-1} \Delta_1(\cH)$, then, bounding $\Delta_i(\cH)$ from below by the average degree of an $i$-element subset of $V(\cH)$, we can write
  \[
    \frac{e(\cH)}{v(\cH)^i} \le \Delta_i(\cH) \ll \pH^{i-1} v(\cH)^{i-1} \cdot \frac{e(\cH)}{v(\cH)^i}.
  \]
  Therefore, $\pH \cdot v(\cH) \gg 1$ as requested.
\end{remark}

\begin{remark}
  Next, it is possible for a sequence of hypergraphs $\cH$ and probabilities $p$ fulfilling the requirements of Proposition~\ref{prop:apx-1state} that the probability that $\cH_p$ is $r$-colourable to be bounded away from zero. Indeed, let $\cH$ be the complete $s$-uniform hypergraph on $n$ vertices. Since $e(\cH) = \Theta(n^s)$, we have that $\pH = \Theta(n^{-1})$, and therefore that $v(\cH) \cdot \pH = \Theta(1)$.

  We verify that $\cH$ satisfies the assumptions of the proposition. First, for every $i \in \br{s}$, we have $\Delta_i(\cH) = \Theta(n^{s-i}) = \Theta(p^{i-1} \Delta_1)$, so $\cH$ is weakly non-clustered. Second, by the pigeonhole principle, every $r$-colouring of the vertices of $\cH$ must have at least $n/r$ elements sharing the same colour.  These vertices alone induce $\Theta(n^s)$ monochromatic edges.

  However, for every $C > 0$ and $p = C \cdot \pH$, the probability that $\cH_p$ is in fact empty is $(1-p)^{v(\cH)} \ge \exp(- \Theta(v(\cH) \cdot p) ) = \exp( -\Theta(1))$. This of course implies that the probability that $\cH_p$ is $r$-colourable is bounded away from zero.
\end{remark}

\begin{proof}[{Proof of Proposition~\ref{prop:apx-1state}}]
  Let $\delta$ be the constant for which $\cH$ is robustly non-$r$-colourable. Since $\cH$ is weakly non-clustered, we may apply the Container Lemma to it.  Write $V = V(\cH)$. Taking any $\eps < \delta/r$, there are constants $c_1 = c_1(\eps)$, $t = t(\eps)$, and a function $f \colon \cP(V)^t \rightarrow \cP(V)$ such that the following hold:
\begin{enumerate}
\item
  For every independent set $I$ there are $T_1, \dots, T_t \subseteq I$ with at most $c_1 \cdot \pH \cdot v(\cH)$ elements, such that $I \subseteq f(T_1, \dots, T_t)$.
\item
  The set $f(T_1, \dots, T_t)$ induces at most $\eps e(\cH)$ edges in $\cH$.  
\end{enumerate}

Suppose now that $\cH_p$ is $r$-colourable. This of course means that $V_p$ is a union of $r$ independent sets $I_1, \dots, I_r$. Following the Container Lemma, this implies that, for each $i \in \br{r}$, there are sets $T^i_1, \dotsc, T^i_t \subseteq I_i$, each containing at most $c_1 \cdot \pH \cdot v(\cH)$ vertices, such that $I_i \subseteq C_i \coloneqq f(T^i_1, \dots, T^i_t)$. Of course, that would mean that all subsets $T^i_j$ are contained in $V(\cH_p)$ and that $V(\cH_p)$ is covered by the containers $C_1, \dots, C_r$.  Using the fact that the containers induce few edges, together with the robust non-colourability, we will be able to prove the following claim.

\begin{claim}
  There is a constant $\alpha$ such that the number of vertices in $V \setminus (C_1 \cup \dots \cup C_r)$ is at least $\alpha v(\cH)$.
\end{claim}

\begin{proof}[Proof of the claim]
  Write $W \coloneqq V \setminus (C_1 \cup \dots \cup C_r)$ and consider the following $r$-colouring of the vertices of $\cH$. Colour each vertex according to the index $i$ of its container $C_i$ (if it is contained in more then one container, choose one arbitrarily), and colour the vertices of $W$ by the colour $1$. Since $\cH$ is robustly non-$r$-colourable, we know that there are at least $\delta e(\cH)$ monochromatic edges. Since every container induces at most $\eps e(\cH)$ edges, there must be at least $(\delta - r \eps) \cdot e(\cH)$ monochromatic edges in the colour $1$ that have at least one vertex in $W$.  On the other hand, there are at most $\Delta_1(\cH) \cdot |W|$ such edges.  By our assumption on $\cH$, there is a constant $K$ such that
  \[
    (\delta - r \eps)\cdot e(\cH) \le \Delta_1(\cH) \cdot |W| \le K \cdot \frac{e(\cH)}{v(\cH)} \cdot |W|.
  \]
  Remembering that $\eps < \delta/r$, there must be a constant $\alpha > 0$ such that $|W| > \alpha v(\cH)$, as promised.
\end{proof}

Following the previous discussions, and letting $c_2 = r t c_1$, we can bound
\[
  \begin{split}
    \Pr & \left( \cH_p \text{ is $r$-colourable} \right)  \\
    & \le \sum_{|T^i_j| \le c_1 \cdot \pH \cdot v(\cH) } \Pr \left( \bigcup_{j \in \br{t}, i \in \br{r}} T^i_j \subseteq V_p \subseteq C_1 \cup \dots \cup C_r \right) \\
 & = \sum_{|T^i_j| \le c_1 \cdot \pH \cdot v(\cH)} \Pr\left( \bigcup T^i_j \subseteq V_p \right) \cdot \Pr \left( V_p \cap \left( V \setminus ( C_1 \cup \dots \cup C_r \right)) = \emptyset \right) \\
 & \le \sum_{k \le r t c_1 \cdot \pH \cdot v(\cH)} \binom{v(\cH)}{k} \cdot 2^{r t k} \cdot p^{k} \cdot (1-p)^{\alpha v(\cH)} \\ 
 & \le c_2 \cdot \pH \cdot v(\cH) \cdot \max_{k \le c_2 \cdot \pH \cdot v(\cH)} \left( \frac{2 ^{rt} e v(\cH) p}{k} \right)^k \cdot \exp \left(- \alpha v(\cH) \cdot p \right).
\end{split}
\]

The function $k \mapsto \left( \frac{e M}{s} \right) ^ k$ is increasing until it reaches its maximum at $k = M$. Therefore, supposing that $C$ is large enough so that $c_2 < 2^{rt} \cdot C$, the maximum is achieved at $s = c_2 \cdot \pH \cdot v(\cH)$, allowing us to bound \[
\begin{split}
\max_{k \le c_2 \cdot \pH \cdot v(\cH)} \left( \frac{2 ^{rt} e \cdot v(\cH) \cdot p}{k} \right)^k & \le \left( \frac{2^{rt} e \cdot  v(\cH) \cdot p}{c_2 \cdot v(\cH) \cdot \pH} \right)^{c_2 \cdot  v(\cH) \cdot \pH} \\
& = \exp \left( c_2 \log \left(\frac{2^{rt} e}{c_2} \cdot \frac{p}{\pH} \right) \cdot \frac{\pH}{p} \cdot v(\cH) \cdot p \right).
\end{split}
\]
Recalling that $C \le \frac{p}{\pH}$ and $x^{-1} \log x \rightarrow 0$ as $x$ tends to infinity, taking $C$ sufficiently large, we have $c_2 \log \left(\frac{2^{rt} e}{c_2} \cdot \frac{p}{\pH} \right) \cdot \frac{\pH}{p} < \frac{1}{2} \alpha$ and, therefore,
\[
\Pr\left( \cH_p \text{ is } r \text{-colourable} \right) \le c_2 \cdot \pH \cdot v(\cH) \cdot \exp  \left( - \frac{1}{2} \alpha \cdot p \cdot v(\cH) \right) \le \exp\left(-\alpha' \cdot p \cdot v(\cH)\right). \qedhere
\]

\end{proof}

\section{A container lemma for sparse sets}
\label{sec:omitted-proofs}

%Restating the hypergraph container lemma from the tools section
\containers*

\begin{thm}[{\cite[Corollary~3.6]{SaxTho15}}]
  \label{thm:containers-Saxton-Thomason}
  For every positive integer $k$ and positive real $\eps$, there exists an integer $s$ such that the following holds.
  Suppose that a nonempty $k$-uniform (multi)hypergraph $\cG$ and $\tau \in (0,1/2)$ satisfy
  \[
    \delta(\cG, \tau) \coloneqq 2^{\binom{k}{2}-1} \sum_{j=2}^k 2^{-\binom{j-1}{2}} \delta_j(\cG, \tau) \le \frac{\eps}{12k!},
  \]
  where
  \[
    \delta_j(\cG, \tau) \coloneqq \frac{\tau^{1-j}}{k e(\cG)} \cdot \sum_{v \in V} \max\{\deg T : v \in T \subseteq V \text{ and } |T| = j\}.
  \]
  Then there exists a function $C \colon \cP(V)^s \to \cP(V)$ such that, letting
  \[
    \cT \coloneqq \big\{(T_1, \dotsc, T_s) \in \cP(V)^s : |T_i| \le s\tau|V| \text{ for all } i \in \br{s} \big\},
  \]
  we have:
  \begin{enumerate}[{label=(\alph*)}]
  \item
    \label{item:containers-ST-1}
    For every set $I \subseteq V$ satisfying $e(\cG[I]) \le 24\eps k! k \tau^k e(\cG)$, there exists $T = (T_1, \dotsc, T_s) \in \cT \cap \cP(I)^s$ with $I \subseteq C(T)$.
  \item
    \label{item:containers-ST-2}
    For every $T \in \cT$, the set $C(T)$ induces at most $\eps e(\cG)$ edges in $\cG$.
  \end{enumerate}
\end{thm}

\begin{proof}[Derivation of Theorem~\ref{thm:containers} from Theorem~\ref{thm:containers-Saxton-Thomason}]
  Let $\cG$ be a nonempty $k$-uniform hypergraph with vertex set $V$ and let $\eps$ and $K$ be positive reals. We set $s \coloneqq s_{\ref{thm:containers-Saxton-Thomason}}(k, \eps)$ and let
  \[
    L \coloneqq \left\lceil \frac{12(k-1)!2^{\binom{k}{2}}K}{\eps} \right\rceil,
    \qquad
    t \coloneqq 2L^2s,
    \qquad
    \text{and}
    \qquad
    \delta \coloneqq 24 \eps k! k L^k
  \]
  Suppose that the maximum degrees of $\cG$ satisfy the assumptions of the theorem for some $\tau$.  Note that, for every $j \in \{2, \dotsc, k\}$,
  \[
    \delta_j(\cG, L \tau) = \frac{(t \tau)^{1-j}}{ke(\cG)} \cdot v(\cG) \Delta_j(\cG) \le \frac{KL^{1-j}}{k}
  \]
  and thus, as $L \ge 2$,
  \[
    \delta(\cG, L\tau) \le 2^{\binom{k}{2}-1} \cdot \sum_{j=2}^{k}\frac{Kt^{1-j}}{k} \le \frac{2^{\binom{k}{2}}K}{kL} \le \frac{\eps}{12k!}.
  \]
  Consequently, Theorem~\ref{thm:containers-Saxton-Thomason}, invoked with $\tau_{\ref{thm:containers-Saxton-Thomason}} = L\tau$, implies that there exist a function $C \colon \cP(V)^s \to \cP(V)$ that satisfies~\ref{item:containers-ST-1} and~\ref{item:containers-ST-2}.  We define $f \colon \cP(V)^t \to \cP(V)$ by letting
  \[
    f(S_1, \dotsc, S_t) \coloneqq C\big(S_1 \cup \dotsb \cup S_{t/s}, \dotsc, S_{(s-1)t/s+1} \cup \dotsb \cup S_t\big).
  \]
  In particular, if $I \subseteq V$ satisfies
  \[
    e(\cG[I]) \le \delta \tau^k e(\cG) \le 24\eps k!k (L \tau)^k e(\cG),
  \]
  then there are $T_1, \dotsc, T_s \subseteq I$, with $|T_i| \le s L\tau |V|$ for each $i$, such that $I \subseteq C(T_1, \dotsc, T_s)$.  This gives the assertion of the theorem, as we may partition each $T_i$ into $t/s$ sets $S_{t(i-1)/s+1}, \dotsc, S_{t(i+1)/s}$, each of size at most $\lceil s/t \cdot sL\tau |V| \rceil \le \tau |V|$.
\end{proof}

\section{A graph without the star-constellation property}
\label{apx:graph_application}

In this appendix we provide an example of a graph whose corresponding hypergraph does not have the rainbow star-constellation property. We need not look further than the usual suspect, the Petersen graph.

\begin{figure}[htb]
\centering
\begin{tikzpicture}

\foreach \i in {1,2}  {
	\node (v1\i) at (0.95*\i, 0.31*\i) {};
	\node (v2\i) at (0.587*\i, -0.81*\i) {};
	\node (v3\i) at (-0.587*\i, -0.81*\i) {};
	\node (v4\i) at (-0.95*\i, 0.31*\i) {};
	\node (v5\i) at (0, 1*\i) {};

	\foreach \t in {1,2,3,4,5} {
			\fill (v\t\i) circle (0.1);
		}
}

\draw (v12) -- (v22);
\draw (v32) -- (v42) -- (v52) -- (v12);
\draw (v11) -- (v31) -- (v51) -- (v21) -- (v41) -- (v11);

\foreach \t in {1,2,3,4,5} {
	\draw (v\t1) -- (v\t2);
}

\fill (v12) circle (0.1) node [right] {$b$};
\fill (v21) circle (0.1) node [above right] {$a$};
\fill (v32) circle (0.1) node [ left] {$u$};

\fill (v42) circle (0.1) node [left] {$b'$};
\fill (v31) circle (0.1) node [above left] {$a'$};
\fill (v22) circle (0.1) node [right] {$u'$};

\end{tikzpicture}
\caption{A Petersen star for $r=2$}
\label{fig:petersen_star}
\end{figure}
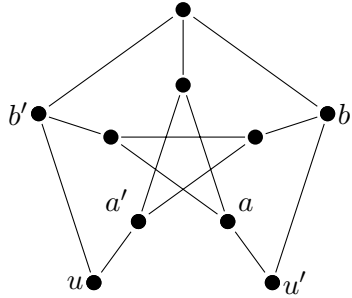

\begin{claim}
The hypergraph corresponding to the Petersen graph does not have the rainbow star-constellation property.
\end{claim}

\begin{proof}
It is enough to show this for $r=2$. Note that since the Petersen graph is edge-transitive there is only one star $S$ up to isomoporphism and therefore only one constellation $C$. to show that there is no homomorphism $C \to S$ we will first claim that every homomorphism $S \to S$ must be an isomorphism. To see why this helps us, observe that every homomorphism $S \to S$ must then map the center vertices to themselves, as they are the only vertices of degree $2$. Therefore, any homomorphism $C \to S$, which induces a homomorphism from each of its stars to $S$, must send all the center vertices of $C$ to the two center vertices of $S$. Viewing this is a $2$-colouring of the center vertices of $C$, we may use the fact that the Petersen graph is not $2$-colourable to find a star in $C$ where both of the center vertices were mapped to the same vertex. However, this would mean that the homomorphism from that star was not an isomorphism, in contradiction to the previous claim.

Suppose $\varphi \colon S \to S$ is a homomorphism. To prove that it is also an isomorphism we show it is injective. Since the Petersen graph has no cycles of length $3$, we learn that every $5$-cycle must be mapped to a $5$-cycle. As a corollary, we learn that whenever $x,y \in S$ are in a $5$-cycle, $\varphi(x) \neq \varphi(y)$. So we only need to worry about pairs of elements that are not in a $5$-cycle. However, one can verify that there are only five pairs of vertices not in a $5$-cycle. Using the labels from Figure~\ref{fig:petersen_star}, these are the center vertices $\{u,u'\}$, the pairs $\{u, a\}$ and $\{u,b\}$, and by symmetry the pairs $\{u',a'\}$ and $\{u',b'\}$. Of course if $\varphi$ were to map two vertices $x,y$ to the same vertex $z$, then it would have to send all neighbours of $x$ and $y$ to the neighbourhood of $z$. In all the above cases the pair $x,y$ has at least $4$ neighbours of which every pair is in a $5$-cycle, so they cannot be mapped to the same vertex, meaning that $z$ would have to be of degree $\ge 4$, however all vertices in $S$ have degree at most $3$.
\end{proof}
\end{document}